\documentclass[a4paper,10pt]{article}
\usepackage[utf8]{inputenc} 
 
\RequirePackage[OT1]{fontenc}
\RequirePackage{amsthm,amsmath}
\RequirePackage[numbers]{natbib} 
\RequirePackage[colorlinks,citecolor=blue,urlcolor=blue]{hyperref}

 \usepackage{amsfonts,latexsym,amstext}
\usepackage{amsmath}
\usepackage[active]{srcltx}      
\usepackage[active]{srcltx}  
\usepackage{color}
\usepackage{amsmath,amsfonts,calrsfs,amssymb,color}
\usepackage{amsthm}
\usepackage{mathrsfs}
\usepackage{pdfsync}
\usepackage{amscd}
\usepackage{amssymb}   
\usepackage{amsmath}
\usepackage{amsfonts}  
\usepackage{color}
\usepackage{mathrsfs}   
\usepackage{amssymb}
\newtheorem{theorem}{Theorem}[section]
\newtheorem{lemma}[theorem]{Lemma}
\newtheorem{proposition}[theorem]{Proposition}

\newtheorem{hypothesis}[theorem]{Hypothesis} 
\newtheorem{Hypothesis}{Hypothesis}

\newtheorem{remark}[theorem]{Remark}

\numberwithin{equation}{section}

\def\qed{{\hfill\hbox{\enspace${ \square}$}} \smallskip}
\def\sqr#1#2{{\vcenter{\vbox{\hrule height .#2pt \hbox{\vrule
 width .#2pt height#1pt \kern#1pt \vrule
width .#2pt} \hrule height .#2pt}}}}
\def\square{\mathchoice\sqr54\sqr54\sqr{4.1}3\sqr{3.5}3}

\def\ds{\begin{displaystyle}}
\def\eds{\end{displaystyle}}
\def\dis{\displaystyle }
\def\<{\langle }
\def\>{\rangle }

\def\dd{\mathcal D}
\def\dim{\noindent \hbox{{\bf Proof.} }}

\def\R{\mathbb R}

\def\E{\mathbb E}
\def\P{\mathbb P}

\def\calf{{\cal F}}
\def\calg{{\cal G}}

\def\caln{{\cal N}}
\def\calp{{\cal P}}

\def\call{{\cal L}}

\title{} 

\begin{document} 

\title{Correction to 
``Well-posedness of semilinear stochastic  wave 
equations with H\"{o}lder continuous 
coefficients''   
 \footnote{This is a correction of    {\sl Masiero and Priola,  J. Differential Equations, 263,  2017} \cite{MaPr}.  In that paper there is a mistake in the  proof of the main uniqueness result.  Indeed  we have used  the strong feller property for the Markov semigroup associated to the linear stochastic wave equation which does not hold (cf.  Theorem 9.2.1 in \cite{DZergo}). We thank  Paolo De Fazio (Parma)    for pointing out this mistake.  Here we correct the proof in \cite{MaPr} avoiding the strong Feller property. We  obtain   pathwise uniqueness  starting from any initial condition as in \cite{MaPr}.  The assumptions on the drift term of the semilinear stochastic wave equation are the same as the ones  considered in \cite{MaPr}. The main changes are in Section 4, in Appendix  and in the proof of Theorem \ref{uni1}.    
}     }   
\date{} 
 \author{
 Federica Masiero\\  
 Dipartimento di Matematica e Applicazioni, Universit\`a di Milano Bicocca\\
 via Cozzi 55, 20125 Milano, Italy\\
 e-mail: federica.masiero@unimib.it,\\
\\
Enrico Priola\\ Dipartimento di Matematica  \\  Universit\`a di Pavia
\\ 
 Pavia, Italy  \\ e-mail: enrico.priola@unipv.it
}
  

\maketitle   
\begin{abstract}
We prove that semilinear stochastic abstract wave equations, including wave and plate equations, are well-posed in the strong sense  with an $\alpha$-H\"{o}lder
continuous drift coefficient, if $\alpha \in (2/3,1)$. 
The uniqueness may fail for the corresponding  
deterministic PDE and 
well-posedness is  restored 
by adding an external random forcing of  
  white noise type.
This shows   
a      
kind of 
regularization by noise
for the semilinear wave equation.
To prove the result we introduce  an 
approach based on backward stochastic differential equations, differentiability along subspaces  and  control theoretic results.  
 We stress that  the well-posedness holds despite  the Markov semigroup associated to the linear stochastic wave equation is not strong Feller. 
\end{abstract}

{\vskip 6mm }

 
  
\section{Introduction   }     
  
    
We prove well-posedness in the strong sense for 
semilinear stochastic abstract wave equations, including wave and plate equations. Let us consider
the following non-linear stochastic wave equation with Dirichlet boundary conditions:
\begin{equation}
 \left\{
 \begin{array}
 [c]{l}%
 \frac{\partial^{2}}{\partial\tau^{2}}y\left(  \tau,\xi\right)  =\frac
 {\partial^{2}}{\partial\xi^{2}}y\left(  \tau,\xi\right)  +b\left(  \tau
 ,\xi, y(\tau, \xi)
\right) +\dot{W}\left(  \tau,\xi\right), \;\;\; \xi \in (0,1),\\
 y\left(  \tau,0\right)  =y\left(  \tau,1\right)  =0,\\
 y\left(  0,\xi\right)  =x_{0}\left(  \xi\right)  ,\\
 \frac{\partial y}{\partial\tau}\left(  0,\xi\right)  =x_{1}\left(  \xi\right),\;\;\; \tau \in (0,T],\;\;\; \xi \in [0,1],
 \end{array}
 \right.  \label{waveequation-holder}
 \end{equation}
  where  $x_0 \in H^1_0([0,1]) $, $x_1 \in L^2([0,1])$  and $\dot{W}\left(  \tau,\xi\right)  $ is a space-time white noise on $\left[
0,T\right]  \times\left[  0,1\right]  $ which  describes an external random forcing;
we treat it as a time-derivative of a cylindrical Wiener process with values in  $L^2([0,1])$. Moreover   $b$  can be  a bounded measurable function
which is H\"{o}lder continuous of exponent $\alpha \in (2/3,1)$ with respect to the $y$-variable;
 see Hypothesis \ref{ip-b} for the more general assumption.
 
 To get pathwise uniqueness for \eqref{waveequation-holder} (see Theorem \ref{uni1}) we introduce  an 
approach based on backward stochastic differential equations. 
{   
 Our main result holds despite the Markov semigroup associated to the linear stochastic wave equation (the so-called Ornstein-Uhlenbeck semigroup) does not have 
  the strong Feller property
 (cf. Theorem 9.2.1 in \cite{DZergo}  and  Remark  \ref{cont}). 
  This is in contrast with other papers dealing with strong uniqueness (see, for instance,  \cite{DPF}, \cite{DPFR}, \cite{DPFPR15}, \cite{DFRV}). 
  
  We will 
 use  the regularizing effect of the  Ornstein-Uhlenbeck semigroup 
  along 
  special directions and   
  interpolation results involving 
 spaces of H\"older continuous functions along 
  special directions (cf. Section 4.1). 
    Related results  have been  considered in \cite{canndap}
    in a different context to investigate   infinite dimensional elliptic equations involving the Gross Laplacian; see in particular Lemma \ref{interpola}. 
    We partially extend Lemma \ref{interpola}   obtaining   Lemma \ref{intervector} which deals with interpolation of  vector-valued functions. This  lemma  will be important   
  in our proof of  pathwise uniqueness (see in particular Theorem \ref{forse111} and the proof of Theorem \ref{uni1}).

 Without the  noise $\dot{W}\left(  \tau,\xi\right)$ the corresponding nonlinear deterministic equation is in general not well-posed; see Section 
\ref{subsec:counterexamples}.
Thus  our result is a kind of regularization by additive noise 
for semilinear stochastic  wave equations.
 There are already  results in this direction
at  the level of SPDEs  of parabolic type 
(see \cite{GP},  
\cite{DPF}, \cite{DPFR}, 
\cite{MN}, \cite{DFRV}, \cite{WZ}    and the references therein).
For  related results on well-posedness of SPDEs by a kind of  multiplicative noise perturbations, see \cite{FGP}, \cite{DT}, \cite{FF13}, \cite{DFV} and the references therein.
 Coming into the details of the problem we treat in the present paper, we study    semilinear abstract wave equations of the form
 \begin{equation} \label{wave111}
\left\{
\begin{array}
[c]{l}
\frac{d^{2} y}{d \tau^{2}} (  \tau)  =
 - \Lambda y (\tau ) + B (t, y(\tau), \frac{d y}{d\tau}(\tau) )  + \dot{W}(\tau), 
 \\
y\left(  0\right)  =x_{0} ,\\
\frac{d y}{d\tau}(  0)  =x_{1},\;\;\; \tau \in (0,T],
\end{array}
\right.  
\end{equation} 
where  
 $\Lambda : \dd (\Lambda) \subset U \to U $ is  a positive self-adjoint operator on a separable Hilbert space $U$ (see, for instance, Example 5.8 and Section 5.5.2
 in \cite{DPsecond}, \cite{Bre} and the references therein) and 
$\left\{ W(\tau)= W_{\tau},\tau\geq0\right\}  $ is a cylindrical Wiener process with values 
in $U$. 
Many 
linear stochastic equations modelling the vibrations of elastic structures 
can be written in  the form \eqref{wave111} with $B=0$
where
$y$
stands for the displacement field (for instance, we consider the stochastic plate equation in Section 3.2).

Comparing with   \eqref{waveequation-holder}, we have that $\Lambda = - \dfrac{d^2}{dx^2}$ with Dirichlet boundary conditions, 
\begin{equation} \label{sett}
 U = L^{2}\left(  \left[
0,1\right]  \right), \;\; 
\dd (\Lambda) = H_{0}^{1}\left(  \left[  0,1\right]  \right)  \cap  H_{}^{2}\left(  \left[  0,1\right]  \right), 
\;\; \dd (\Lambda^{1/2}) = H_{0}^{1}\left(  \left[  0,1\right]  \right) = V 
\end{equation}
and $\dd(\Lambda^{-1/2}) 
 =  H_{}^{-1}\left(  \left[  0,1\right]  \right) $ (cf. Section 2 and \cite{krabs}).  
 
 To study equations \eqref{wave111} we consider  two basic Hilbert spaces: $K$ and $H.$ The first one is 
 $$
 K=\dd(\Lambda^{1/2})\times U = V \times U. 
 $$
 This is the usual space for the deterministic wave equation obtained when $B=0$ (removing $\dot{W}(\tau)$ from the equation). This space is also denoted by $V \oplus U$.
 However  even if $B=0$ 
solutions to  stochastic wave equations \eqref{wave111}  do not evolve in $K$ 
but in the larger space 
$$ 
H=U \times \dd(\Lambda^{-1/2}) = U \times V'  
$$ 
(see  Example 5.8 in \cite{DPsecond}).     Here $V'$ is the dual space of $V$.
On the other hand, as we mention before, the Ornstein-Uhlenbeck semigroup 
has a regularizing effect only along the directions of $K$ (see Section 3).

The existence of a weak solution 
$X_{\tau}^{0,x} = $ $\big ( y(\tau), \frac{d y }{d \tau}(\tau) \big)$  to \eqref{wave111} for any initial condition  
$x = (x_0, x_1) \in H$,  taking values in $H$  and with
continuous paths is well known if $B: [0,T] \times H \to U$ is Borel and bounded; see Section 2 for more details.  It follows by the Girsanov theorem (cf.         
 \cite{DPsecond}, \cite{Ondre04},  \cite{Mas1} and 
Remark \ref{girsa}) 
writing    
\eqref{wave111} as
\begin{equation}
\label{wa1}
dX_{\tau}^{0,x}  = A X_{\tau}^{0,x} d\tau+GB(\tau,X_{\tau}^{0,x})d\tau+GdW_\tau
,\text{ \ \ \ }\tau\in\left[  0,T\right], \text{ \ \ \ } \dis
X_{0}^{0,x} =x \in H,
\end{equation} 
where $A$ is the generator of the wave group in $H$ and    $GdW_\tau =$ $\left(
\begin{array}
[c]{c} 
0\\
dW_\tau
\end{array} 
\right) $.    
  To prove  pathwise uniqueness
 we require that $B : [0,T] \times H \to U$ is Borel, bounded and $\alpha$-H\"older continuous in the $x$-variable, $\alpha \in (2/3,1)$, uniformly in $t \in [0,T]$ (cf. Hypothesis \ref{ip-b1}).

 Our strategy to show pathwise uniqueness requires first to  investigate regularizing properties of the $J$-valued  Ornstein-Uhlenbeck  semigroup $(R_t)$
(see  Section 4). Here  $J$ can be  any real separable Hilbert space.
   We have 
 $R_{\tau} [\Phi] (x) = \E [\Phi (X_{\tau}^{0,x})]$, $\tau \ge 0$,
 $\Phi \in B_b(H,J)$,
where $X_{\tau}^{0,x}$ is the Ornstein-Uhlenbeck process solving \eqref{wa1} when $B=0$. 
   To prove
   the  differentiability  of $R_{\tau} [\Phi]$, along the directions of 
  $K$,  $\tau >0$,  we use  sharp results 
on the behaviour of the  minimal  energy  for the 
 linear control   system 
\begin{equation} 
\left\{   
\begin{array}    
[c]{l}%
\overset{\cdot}{w}\left(  t\right)  =Aw\left(  t\right)  +Gu(  t )
,\\
w\left(  0\right)  =h \in K,
\end{array}
\right.  \label{wave null cont syst}
\end{equation}
with controls $u \in L^2_{loc} (0, \infty ; U)$   
 (see 
  Theorem \ref{ci1},  Section \ref{refer}   and the references mentioned  in Appendix). We also need interpolation results involving H\"older functions along the directions of $K$ 
  (see 
   in particular Lemma \ref{intervector}). 
    We  will also   consider  
second directional derivatives of $R_{\tau} [\Phi]$; see in particular  the estimate   for  
  $$
   \sum_{m \ge 1} \sup_{\, a \in U, |a|_U =1}| \nabla_k \nabla_{Ga}  \E [\, \langle \Phi(X_{\tau}^{0, \cdot \, }) ,  f_m\rangle_J \, ](x)|^2,  
  $$
 $x  \in H,$ $k \in K$, $\tau >0$, given in Lemma \ref{der12a}  
   (here $(f_m)$ denotes  any basis of $J$).   
 At the end of Section 4 we  establish 
 a regularity result for the following    Kolmogorov  integral equation:
\begin{align} \label{bsde2uno2}
u(t,x)
=&\int_t^T R_{s-t}\left[e^{-(s-t) {A}}G B(s,\cdot)\right](x)\,ds+
\int_t^T R_{s-t}\left[e^{-(s-t){A}} \nabla^Gu(s,\cdot) B(s,\cdot) \right](x)\,ds, 
\end{align}
where the  unknown function  $u(t,x)$ takes values in $K$ and  $ \nabla^G u(s,x)B(s,x) =  \nabla_{GB(s,x)}u(s,x) \in K $, $(s,x) \in [0,T] \times H$. 
  To this purpose we use the regularizing effects of $R_{\tau}$ when 
$ 
  J=K
 $
 (see Theorem \ref{forse111}).  In Remark \ref{remark-final1} we will compare our Kolmogorov  equation with the one used in \cite{DPF} to study parabolic SPDEs . 
     
   In Section 5 we introduce backward stochastic differential equations (BSDEs from now on) for the unknown pair of processes $(Y^{t,x},Z^{t,x})$,  
coupled with the  Ornstein-Uhlenbeck process $\Xi^{t,x}$ starting from $x$ at time $t$:
\begin{equation} \label{dicc}
 \left\lbrace\begin{array}
[c]{l}%
d\Xi_{\tau}^{t,x}  =A\Xi_{\tau}^{t,x} d\tau+GdW_\tau
,\text{ \ \ \ }\tau\in\left[  t,T\right], \\ \dis
\Xi_{t}^{t,x} =x,\\ \dis
 -dY_{\tau}^{t,x}=-AY_{\tau}^{t,x} d\tau+G B(\tau,\Xi^{t,x}_\tau)\,d\tau+Z^{t,x}_{\tau}B(\tau,\Xi_{\tau}^{t,x})d\tau-Z^{t,x}_{\tau}\;dW_\tau,
 \qquad \tau\in [t,T],
  \\\dis
  Y_{T}^{t,x}=0.
\end{array}
\right.  
\end{equation}
The  process $Y^{t,x}$  takes values in {$K$} and $Z^{t,x}$ in the space {$L_2 (U,K)$ }
(cf. \cite{HuPeng}, \cite{BP} and \cite{fute}). 
 We study first differentiability of $(Y^{t,x},Z^{t,x})$ with respect to $x$ assuming in addition that the coefficient $B$ is  regular. Such type of results, together with the identification of $Z^{t,x}$ with the directional derivative of $Y^{t,x}$, are  known also in the infinite dimensional case  when $Y^{t,x}$ 
is real, see \cite{fute};  here we extend these results to the case when $Y^{t,x}$ is Hilbert space valued (see Proposition \ref{Teo:ex-regdep-BSDE} and Lemma \ref{prop-identific-Z-diffle}).
Then, using  the results of Section 4 and an approximation argument, we are able to  
study regularity properties of solutions $(Y^{t,x},Z^{t,x})$ together with 
the  identification of $Z^{t,x}$
in the case of an H\"older continuous drift $B$ (see Theorem \ref{teo-identific-Z} which holds under  more general assumptions on $B$ and also Lemma \ref{forse11}).
  
The results of Sections 4 and 5  allow to get in Section \ref{sez-uniq-wave}
the important identity    
 \begin{equation}\label{waveeq-nobad1}
  X_\tau ^{0,x}= e^{\tau A}x+e^{\tau A}v(0,x)-v(\tau,X_\tau^{x})+\int_0^\tau e^{(\tau-s)A}\nabla^G v(s,X_s^x)\;dW_s 
  +\int_0^\tau e^{(\tau-s)A}GdW_s,
 \end{equation}
  which holds for any weak mild solution  $(X_\tau ^{0,x})$. Note that 
the irregular  coefficient  $B$ is not present in \eqref{waveeq-nobad1}. This identity 
involves a deterministic function  $v  $ related to $Y^{t,x}$ (indeed $v(t,x)= Y_{t}^{t,x} \in K$,  $(t,x) \in [0,T] \times H$; see \eqref{v}). Moreover 
$v$  is ``very regular'' because it solves the Kolmogorov equation \eqref{bsde2uno2}.

Identities 
like \eqref{waveeq-nobad1} are established in 
  \cite{FGP}, \cite{DPF}, \cite{DPFR}, 
\cite{DFRV}, \cite{WZ} 
by the so-called  It\^o-Tanaka trick which is a variant of the Zvonkin method used in \cite{Ver} (see also our Remark \ref{remark-final1} and \cite{Fl}).  Here we prove \eqref{waveeq-nobad1} by using 
 the mild form of the BSDE, which, together with the group property of $A$, allows to remove the ``bad term'' $B$ 
  of the semilinear stochastic wave equation.
 We stress that  in contrast with the previous papers  which use the  It\^o-Tanaka trick here we have a   function $v$ which is regular only along the directions of $K$ (see Theorem \ref{forse111}).  
   We can use the previous 
  identity and  prove pathwise uniqueness  
 noting that  
  (see \eqref{wa1}) 
$$ 
  X^{0,x_1}_{\tau} -   X^{0,x_2}_{\tau} \in K,\;\;\; \tau \in [0,T],
 $$
 if $x_1, x_2 \in H$ and $x_1 - x_2 \in K$ (i.e., the difference of two solutions evolves in $K$ but not the single solution; cf. \eqref{evolve}).
      Finally note that  by Theorem  \ref{uni1}, using   an extension   of  the Yamada-Watanabe theorem 
 (see \cite{Ondre04} and \cite{LR15}), one can   obtain   that   \eqref{wave111}   
has a unique {  strong mild solution,}   { for  any $x \in H$.}

  \begin{remark}\label{gen} {\em  Using a localization argument as in  \cite{DPFPR15}, the boundeness of $ B$ could  be dispensed.  In particular, one can prove strong well-posedness of \eqref{wave111}, for any $x \in H$, under 
  Hypothesis  \ref{base}           
 and  assuming that $B : [0,T] \times H \to U$  is continuous on $[0,T] \times H $  
  and  growths at most linearly, uniformly in $t \in [0,T]$;
  moreover, one requires that for any ball $S \subset H$
  the function $B (t, \cdot ): S \to U$ is $\alpha$-H\"older continuous,  for some $\alpha > 2/3$, uniformly in $t \in [0,T]$ (cf. \eqref{sqq}).}
\end{remark}

\section{Notations and 
preliminary results}\label{sec-prel}

Given two real separable Hilbert spaces $H$ and $J$ we denote by  
$L(H,J)$  the space of bounded linear operators
from $H$ to $J$, endowed with the usual operator norm; 
$L_2(H,J)$ is the subspace of all Hilbert-Schmidt operators  endowed with the Hilbert-Schmidt norm $\| \cdot\|_{L_2(H,J)}$. Let $E$ be a  Banach space.  $B_b(H,E)$ is the space of all Borel  and bounded functions from $H$ into  $E$ endowed with the  supremum norm $\| \cdot \|_{\infty}$, $\| f\|_{\infty}$ $= \sup_{x \in H} |f(x)|_E$, $f \in B_b(H,E)$.
 $C_b(H,E)$ is its subspace consisting of all uniformly continuous and bounded functions from $H$ into  $E$.
The space $C^1_b(H,E)$ is the space of all functions in $C_b(H,E)$ which are Fr\'echet differentiable on $H$ with  bounded and uniformly continuous Fr\'echet derivative $\nabla f : H \to L(H,E)$; it is a Banach space endowed with the norm
 $\| \cdot \|_{C^1_b}$, $\| f \|_{C^1_b}$ $= \| f \|_{\infty}$
  $+ \| \nabla f \|_{\infty}$, $f \in C^1_b(H,E)$.
 We  define, for $0<\alpha<1$, the space $C_b^\alpha(H, E) $  of all functions $f$ in $C_b(H, E)$ such that   
\begin{equation}\label{kk10}
[f]_{\alpha} = \sup_{x',\; x \in H , x-x' \not = 0} {| f(x) - f(x')|_E}\, {|x - x'|_H^{-\alpha}}  < \infty.   
\end{equation}
It is a Banach space endowed with the  norm $\| \cdot \|_{\alpha} = \| \cdot\|_{\infty} + [\cdot]_{\alpha}$. 

By $C([0,T]\times H, E)$ we denote the space of continuous functions from the product space
$[0,T]\times H$ into  $E$. Moreover, $B_b([0,T]\times H, E)$ is   the Banach space of bounded Borel measurable  functions from 
$[0,T]\times H$ into  $E$ endowed with the sup norm.

Let $U$ be a real separable Hilbert space with inner product $\langle \cdot ,\cdot \rangle_U $ and norm $|\cdot|_U$. To study \eqref{wave111} we  assume that
\begin{Hypothesis} \label{base}
 $\Lambda : \dd (\Lambda) \subset U \to U $ is  a given positive self-adjoint operator   and there exists $\Lambda^{-1}$ which is a trace class operator from $U$ into $U$.
\end{Hypothesis}
Recall that positivity of $\Lambda$ means that 
there exists $m>0$ such that 
$\langle \Lambda u, u \rangle_U $ $\ge m |u|^2_U, $ $ u \in \dd (\Lambda)$ (see, for instance, Section 3.3 in \cite{TW}).
We also consider the Hilbert space 
$
V = \dd (\Lambda^{1/2})
$  $=$Im$(\Lambda^{-1/2})$ endowed with the inner product 
$$
\langle  h, k \rangle_V 
= \langle \Lambda^{1/2}
h, \Lambda^{1/2} k \rangle_U, \;\;
h,k \in V
$$
and its dual space $V' $ which is again a Hilbert space. Note that $|\cdot |_{V'} $ is equivalent to $| \Lambda^{-1/2}  \cdot|_{U}$. Moreover,   $V' $ can be identified with the completion of $U$ 
with respect to the norm $| \Lambda^{-1/2} \cdot |_{U}$ (see Section 3.4 in \cite{TW}).
 $V' $ is also denoted by 
$\dd (\Lambda^{-1/2})$. 
We have 
$
 V \subset U \simeq U' \subset V'
$
 with continuous inclusions; $\Lambda$ can be extended to an unbounded self-adjoint operator on $V'$ with domain $V$, which we still denote by $\Lambda$: 
\begin{equation}\label{lambda-est}
\Lambda : V  \to V'.   
\end{equation}  
We  consider the linear stochastic wave equation in a complete
probability space $\left(  \Omega,\mathcal{F},\mathbb{P}\right)  $ with a
filtration $\left(  \mathcal{F}_{\tau}\right)  _{\tau\geq0}$ satisfying the
usual conditions. We have 
 \begin{equation} \label{wave1}
\left\{
\begin{array}
[c]{l}
\frac{d^{2} y}{d \tau^{2}} (  \tau)  =
 - \Lambda y (\tau ) + \dot{W}(\tau), 
 \\
y\left(  0\right)  =x_{0} , \;\;\; 
\frac{d y}{d\tau}(  0)  =x_{1}.
\end{array}
\right.  
\end{equation} 
where $\left\{ W(\tau)= W_{\tau},\tau\geq0\right\}  $ is a cylindrical Wiener process
in $U$ with respect to the filtration $\left(  \mathcal{F}_{\tau}\right)
_{\tau\geq0}$.  The process $W_t$ is formally given by ``$W_t $ $ = \sum_{j \ge 1}
  \beta_j(t) e_j$'' where $\beta_j(t)$ are independent real
   Wiener processes and $(e_j)$ denotes a basis in $U$ (see \cite{DPsecond} for more details).
We introduce  the reference Hilbert space for the solutions to \eqref{wave1}: 
\[ 
H=  U \times  V'    
\] 
endowed with the inner product $\langle x,y \rangle_H$ 
$= \langle x_1 , y_1  \rangle_U $ + $\langle x_2 , y_2  \rangle_{V'}$
 $=  \langle x_1 , y_1  \rangle_U $ + $\langle \Lambda^{-1/2} x_2 , \Lambda^{-1/2} y_2  \rangle_{U}$
 and norm $|x|_H = 
(\langle x,x \rangle_H)^{1/2}$, $x,y \in H$. This space is also denoted by $U \oplus V'$.

In the sequel we will also denote $\langle \cdot, \cdot\rangle_H$ and $|\cdot|_H$ by $\langle \cdot, \cdot\rangle $ and $|\cdot|$.
According to   \cite{DPsecond}, the equation \eqref{wave1} is well-posed in $H$ thanks to Hypothesis \ref{base}. 
On the other hand, \eqref{wave1} is not well-posed in the usual  space $K$ for the deterministic   wave equation:
\begin{equation}\label{kkk}
K = V \times U =  {\dd(\Lambda^{1/2})} \times {U}
\end{equation}
(i.e., solutions to  (\ref{wave1}) do not evolve in $K=
V \times U$ even if $x_0 \in V$ and $x_1 \in U$; see  Example 5.8 in \cite{DPsecond}).   Recall the inner product $\langle x,y \rangle_K$  
$= \langle x_1 , y_1  \rangle_V $ + $\langle x_2 , y_2  \rangle_{U} $, $x,y \in K.$  
In $H$ one considers the unbounded wave operator $A$  which generates a unitary group 
$e^{tA}$:
\[
\mathcal{D}\left(  A\right)  = V 
\times U,
\text{ \ \ \ \ }%
A\left(
\begin{array}
[c]{c}%
y\\
z
\end{array}
\right)  =\left(
\begin{array}
[c]{cc}%
0 & I\\
-\Lambda & 0
\end{array}
\right)  \left(
\begin{array}
[c]{c}%
y\\
z
\end{array}
\right)  ,\text{ \ for every }\left(
\begin{array}
[c]{c}%
y\\
z
\end{array}
\right)  \in\mathcal{D}\left(  A\right),
\]
\[
e^{tA}\left(
\begin{array}
[c]{c}%
y\\
z
\end{array}
\right)  =\left(
\begin{array}
[c]{cc}%
\cos\sqrt{\Lambda}t & \frac{1}{\sqrt{\Lambda}}\sin\sqrt{\Lambda}t\\
-\sqrt{\Lambda}\sin\sqrt{\Lambda}t & \cos\sqrt{\Lambda}t
\end{array}
\right)  \left(
\begin{array}
[c]{c}%
y\\
z
\end{array}
\right)  ,\text{ \ \ }t\in\mathbb{R},\;\; \left(
\begin{array}
[c]{c}%
y\\
z
\end{array}
\right)  \in H.
\]
Let $G : U \to H$, 
\begin{equation} \label{G}
G u=\left(
 \begin{array}
 [c]{c}%
 0\\
 u
 \end{array}
 \right)  =\left(
 \begin{array}
 [c]{c}%
 0\\
 I
 \end{array}
 \right)  u,\;\;\; u \in U.
\end{equation}
Notice that   $e^{tA}:K\rightarrow K$ and $e^{tA}:H\rightarrow H$, and moreover since $(e^{tA})_{t}$ is a group of linear operators, then
\begin{equation}\label{immagine-semigruppo}
e^{tA}(K)= K, \quad e^{tA}(H)= H,\;\;\;  t \in \R.   
\end{equation}  
 Equation (\ref{wave1}) can be rewritten in an abstract form as
\begin{equation} 
\left\{ 
\begin{array}
[c]{l}%
dX_\tau  =AX_\tau  d\tau + G dW_{\tau},\text{ \ \ \ }\tau\in\left[  0,T\right].
 \\
X_0  =x \in H,
\end{array}
\right.  
\label{wave eq ab det}
\end{equation}
A solution to \eqref{wave eq ab det} is a  particular 
Ornstein-Uhlenbeck process.
  We study  \eqref{wave1}
 in  $H$ since  
 the operators 
 \begin{equation} \label{qttt}
Q_{\tau } = \dis\int_{0}^{\tau}e^{sA}GG^{\ast}e^{sA^{\ast}}ds, \;\;\; \tau \ge 0,
 \end{equation} 
are of trace class from $H$ into $H$ thanks to Hypothesis \ref{base} (cf. Example 5.8 in \cite{DPsecond}); here $G^*$ denotes the adjoint operator of $G$ in $H$. Thus  the stochastic convolution (i.e., the solution to \eqref{wave eq ab det} when $x=0$)
\begin {equation} \label{stoc1}
S_{\tau} = \int_{0}^{\tau}e^{\left(  \tau-s\right)  A}GdW_s \;\; \text{is well defined in $H$.} 
\end{equation}
Its law at time $\tau$ is the Gaussian measure ${\cal N}(0, Q_{\tau})$ with mean 0 and covariance operator $Q_{\tau}$ (cf. \cite{DPsecond}). 
 Note that we have 
\begin{gather*} 
 S_{\tau} = \left(
 \begin{array}
 [c]{c}%
   C_{\tau}
   \\
 D_{\tau} 
 \end{array}
 \right) \;\;\; \text{where } \; C_{\tau} =  \int_{0}^{\tau}  \frac{ \sin \big ( \sqrt{\Lambda} (\tau -s)\big ) }{\sqrt{\Lambda}}   dW_s,  \;\;\;  D_{\tau} =  \int_{0}^{\tau}  { \cos \big ( \sqrt{\Lambda} (\tau -s)\big ) }
  dW_s. 
\end{gather*}  
 Moreover,  since
$$
 \sup_{t \in [0,T]} \| e^{tA} G \|_{L_2(U, H)} < \infty, \;\;\; T>0,
$$
we can apply   Theorem 5.11  in \cite{DPsecond} and deduce that the process $(S_{\tau})$ has a continuous version with values in $H$.
  Concerning  the semilinear stochastic equation \eqref{wa1}, we assume that
  \begin{Hypothesis}\label{ip-b1}
$B: [0,T] \times H \to   
U$   
 is (Borel) measurable and bounded
 and moreover there exists $C>0$ such that
 \begin{equation}\label{sqq}
|B(t, x+ h) - B(t,x) |_{U} \le C |h|^{\alpha}_H,\;\; x,h \in H ,  \; t \in [0,T],
 \end{equation} 
for some $\alpha \in (2/3,1)$. We also write that $B \in B_b([0,T]; C^{\alpha}_b(H,U))$ with $\alpha \in (2/3,1)$. 
\end{Hypothesis}
   
\vskip 1mm  
  Let $x \in H$. Recall that a  (weak) mild solution  to
(\ref{wa1}) is a tuple   $\left(
\Omega, {\mathcal F},
 ({ \mathcal F}_{t}), \P, W, X\right) $, where $\left(
\Omega, {\mathcal F},
 ({\mathcal F}_{t}), \P \right)$ is a  stochastic basis  
  on which it is defined a
 cylindrical $U$-valued ${\mathcal F}_{t}$-Wiener process $W$ and
 a continuous ${\mathcal F}_{t}$-adapted $H$-valued
process $X = (X_t) = (X_t)_{t \in [ 0,T] }$ such that, $\P$-a.s.,
\begin{equation}
X_{t}=e^{tA}x+\int_{0}^{t}e^{\left(  t-s\right)  A}GB\left (s,
X_{s}\right)
ds+\int_{0}^{t}e^{\left(  t-s\right)  A}GdW_{s},\;\;\; t \in [0,T].\label{mild}
\end{equation}
According to Chapter 1 in \cite{Ondre04}  (see also \cite{LR15})   we say that  strong existence holds for  
equation \eqref{wa1}   if, for every stochastic basis   
$(\Omega,{\cal F}, ({\cal F}_t), \P)$ on which there is defined an $U$-valued cylindrical
${\cal F}_t$-Wiener process $W$, for any initial condition $x \in H,$ 
there exists an $H$-valued continuous $({\cal F}_t)$-adapted
process $X= (X_t)= (X_t)_{t \in [0,T]}$ such that
 $(\Omega ,{\cal F}, ({\cal F}_t), \P,W, X)
 $ is a weak mild solution.  We also write $X_t^{0,x}$ or $X_t^x$  instead of $X_t$. Similarly, we denote by $(X_{\tau}^{t,x})_{\tau \ge t}$ the solution to \eqref{wa1} starting from $x \in H$ at time $t \in [0,T]$.   
{Note that if  $a \in U$ 
\begin{equation*}
 Ga = \left (\begin{array}
[c]{c}
0\\
a
\end{array}
\right) \in K  \;\; \text{and } \;\; e^{tA} 
\left (\begin{array}
[c]{c}
0\\
a
\end{array}
\right)   =\left(
\begin{array}
[c]{c}
\frac{1}{\sqrt{{\Lambda}}}\sin(\sqrt{{\Lambda}}\, t) \, a \\
 \cos(\sqrt{{\Lambda}}t)a
\end{array}
\right)  \in K, \;\; \text{ \ \ }t\in\mathbb{R}.
\end{equation*}   
Hence, since in Hypotheses \ref{ip-b1} we assume that the drift $B$ takes its values in $U$, then $\displaystyle\int_{0}^{t}e^{\left(  t-s\right)  A}GB\left (s,
X_{s}\right)
ds$ evolves in $  {K}$  
{: it is $K$- valued and the map $t\mapsto \displaystyle\int_{0}^{t}e^{\left(  t-s\right)  A}GB\left (s,
X_{s}\right)ds$ is continous due to the boundedness of $B$ (indeed, let $T>0$; for any $\omega$, $\P$-a.s., $s \mapsto GB\left (s,
X_{s}(\omega)\right)$ is Borel and bounded from $[0,T]$ with values in $K$ and so we can apply Lemma 3.1.5 in \cite{CW}).}    
\\   
{Therefore even if  in general  a solution $(X_t^x)$ does not evolve  in $K$ (cf. \eqref{stoc1})   we know  that  when  the initial conditions $x_1 $ and $x_2$ are in $H$  and verify       
$$ 
x_1 - x_2 \in K 
$$ 
 then any couple of weak mild solutions   $X_t^{x_1}$ and $ X_t^{x_2}$ (starting at $x_1$ and $x_2$ respectively) verifies the   property  
\begin{gather} \label{evolve} 
( X_t^{x_1} -  X_t^{x_2})  \;\; \text{evolves in} \; \; K,\;\;\; t \in [0,T].
\end{gather}    
Indeed           
$$ 
X_t^{x_1} -  X_t^{x_2}=e^{ t  A}(x_1-x_2)+\displaystyle\int_{0}^{t}e^{\left(  t-s\right)  A}G\Big(B\left (s,X^{x_1}_{s}\right)-B\left (s,X^{x_2}_{s}\right)\Big)\,ds,\;\;\; t \in [0,T];
$$ 
the stochastic integral has disappeared, and since $x_1 - x_2 \in K$, also $e^{ t  A}(x_1-x_2)\in K$; the other term we have already discussed that belongs to $K$.}   
    Note that, $\P$-a.s.,  the paths of $ (X_t^{x_1} -  X_t^{x_2})$
 are continuous functions from $[0,T]$ with values in $K$.     
    Property  \eqref{evolve} will be important  in the proof of  our  uniqueness result (see Section 6).  Indeed recall that    the  Ornstein-Uhlenbeck semigroup  regularizes only in the directions of $K$  (see Section 4).

 
\begin{remark} \label{girsa} {\em 
 Thanks to  the boundedness of $B$ we can  apply the Girsanov Theorem as in \cite{Mas1}. For the infinite dimensional Girsanov theorem we refer to Proposition 7.1 in \cite{Ondre04}      and Section 10.3 in \cite{DPsecond}.
 The Girsanov theorem allows to prove  Theorem 5 in \cite{Ondre04} which states that there always exists  a weak mild solution, starting from any $x \in H$ (Theorem 5 in \cite{Ondre04}  even shows weak existence  for random initial conditions).   
Moreover uniqueness in law holds for \eqref{wa1}. To deduce such results by  Theorem 5 of \cite{Ondre04}  we note  the following facts: as $f $ in \cite{Ondre04}    
 we can consider   
our $GB : [0,T] \times H \to {K} \subset H$; our space $H$ can be the space  $U = X =X_1$ used in \cite{Ondre04}; {the space $U_0$ in \cite{Ondre04}  can be our $U$; finally   as  cylindrical Wiener process of Theorem 5 in \cite{Ondre04} we can consider  our $ W$.}         
}
\end{remark}
\section {Examples  }\label{sez-examples}
 
   
We   present two classes of abstract semilinear stochastic wave equations that we can treat:
the stochastic semilinear wave and  plate equations. In Section 3.3
we also give a counterexample to uniqueness for deterministic semilinear wave equations with H\"older continuous coefficients. 
\subsection{Stochastic wave equations}
 We first deal with   the semilinear  stochastic wave equation as in Introduction, i.e., 
\begin{equation}
 \left\{
 \begin{array}
 [c]{l}%
 \frac{\partial^{2}}{\partial\tau^{2}}y\left(  \tau,\xi\right)  =\frac
 {\partial^{2}}{\partial\xi^{2}}y\left(  \tau,\xi\right)  +b\left(  \tau
 ,\xi, y(\tau, \xi)
\right) +\dot{W}\left(  \tau,\xi\right), \\
 y\left(  \tau,0\right)  =y\left(  \tau,1\right)  =0,\\
 y\left(  0,\xi\right)  =x_{0}\left(  \xi\right)  ,\\
 \frac{\partial y}{\partial\tau}\left(  0,\xi\right)  =x_{1}\left(  \xi\right),\;\;\; \tau \in [0,T],\; \xi \in [0,1].
 \end{array}
 \right.  \label{waveequation-holder11}
 \end{equation}
 Comparing with   \eqref{wave111}, $\Lambda = -\frac{d^2}{dx^2}$ with Dirichlet boundary conditions, i.e.,  $\dd (\Lambda) = H_{0}^{1}\left(  \left[  0,1\right]  \right)  \cap  H_{}^{2}\left(  \left[  0,1\right]  \right)$. Note that   
 $\Lambda^{-1}$ is of trace class since eigenvalues of $\Lambda$ are $\lambda_n = n^2$, $n \ge 1$. Thus Hypothesis \ref{base} holds. 

 We still denote by $\Lambda$ its extension on $H^{-1}\left(  \left[  0,1\right]  \right)$ with domain
\[
\mathcal{D}\left(  \Lambda\right)  =
H_{0}^{1}\left(  \left[  0,1\right]  \right)  ,\text{ \ \ \ \ }
\Lambda y    =-\frac{\partial^{2} y }{\partial\xi^{2}%
}  \in H^{-1}([0,1]) ,\text{ \ for every }y\in\mathcal{D}\left(
\Lambda\right).
 \]
We consider  
 $x_0 \in  U = L^{2}\left(  \left[
0,1\right]  \right)$, $x_1 \in H_{}^{-1}\left(  \left[  0,1\right]  \right)$.

Writing 
$ X_\tau(\xi):=
 \Big(
 \begin{array}{l}
  y(\tau, \xi)\\
  \frac{\partial}{\partial \tau}y(\tau,\xi)
 \end{array}
 \Big),$  according to Section 2,
the reference  Hilbert space for the solution 
is  $H= L^{2}\left(  \left[
0,1\right]  \right) \times H_{}^{-1}\left(  \left[  0,1\right]  \right)$.

By considering $G: L^2([0,1])  \longrightarrow H$,
$G u=\left(
 \begin{array}
 [c]{c}%
 0\\
 u
 \end{array}
 \right)  =\left(
 \begin{array}
 [c]{c}%
 0\\
 I
 \end{array}
 \right)  u
 $
(cf. \eqref{G}) 
we can rewrite \eqref{waveequation-holder11} 
in the abstract form \eqref{wa1} with
 $ B(\tau, h)$ $:=
 b(\tau, \cdot, h_1 (\cdot))$ and
\begin{equation}\label{B}
G B(\tau, h) (\xi):=
 \left(\begin{array}{l}
     0\\
 b(\tau, \xi, h_1(\xi)) 
 \end{array}\right), \;\;\; \xi \in [0,1],\; \tau \in [0,T], \; h = (h_1, h_2) \in H.
\end{equation}
It is easy to check that the next assumptions on $b$ imply the validity of Hypothesis \ref{ip-b1} for $B$.
 \begin{hypothesis}\label{ip-b}
The function $b : \left[  0,T\right]  \times\left[  0,1\right]
\times\mathbb{R} \to \R $  is measurable and,  for  $\tau
\in\left[  0,T\right]  ,$ a.e. $\xi\in\left[  0,1\right]  ,$ the map $b\left(
\tau,\xi,\cdot\right)  :\mathbb{\mathbb{R}}  \rightarrow\mathbb{R}$ is continuous. 
There exists $c_{1}$ bounded and measurable 
on $\left[  0,1\right] $, $\alpha \in (2/3,1)$,  such that,
 for  $\tau\in\left[
0,T\right]  $ and a.e. $\xi\in\left[  0,1\right],$   
\[
\left|  b\left(  \tau,\xi,x\right)  -b\left(  \tau,\xi,y\right)  \right|
\leq c_{1}\left(  \xi\right)  \left|  x-y\right|^\alpha,
\]
$ x,\,y \in \R$. Moreover
$\left|  b\left(  \tau,\xi,x\right)  \right|  \leq c_{2}\left(  \xi\right)
,$  for $\tau \in [0,T]$, $x \in \R$ and a.e. $\xi \in [0,1]$, with $c_{2} \in L^2 (\left[  0,1\right] ) $.
\end{hypothesis}
   

\subsection{Stochastic plate equations }

Let  $D \subset \R^2$ be  a bounded open domain  with smooth boundary  $\partial D$, which
represents an elastic plate.  
We  consider   the following semilinear  stochastic  plate  equation  
\begin{equation}
 \left\{
 \begin{array}
 [c]{l}%
 \frac{\partial^{2} y }{\partial\tau^{2}}\left(  \tau,\xi\right)  =- \triangle^2 y\left(  \tau,\xi\right)  +b\left(  \tau
 ,\xi, y(\tau, \xi)
\right) +\dot{W}\left(  \tau,\xi\right), 
\\ \\
y\left(  \tau, z\right)    =0,  \;\;\;\;
\frac{  \partial y}{\partial \nu } \left(  \tau, z\right)    =0,
\;\;\; z \in \partial D,
\\ \\
y\left(  0,\xi\right)  =x_{0}\left(  \xi\right)  , \;\;\;
 \frac{\partial y}{\partial\tau}\left(  0,\xi\right)  =x_{1}\left(  \xi\right),\;\;\; \tau \in (0,T],\; \xi \in \overline{D},
 \end{array}
 \right.  \label{plate}
 \end{equation}
where $\triangle $ is the Laplacian in $\xi$, $\triangle^2 = \triangle (\triangle)$ is a fourth order operator, $\frac{\partial}{\partial \nu}$ denotes the outward  normal derivative on the boundary (we are considering the so-called clamped boundary conditions) and $\dot{W}\left(  \tau,\xi\right)$ is a space-time white noise on $[0,T] \times D$. 
We remark that weak existence and uniqueness in law 
for  non-linear  stochastic plate equations with multiplicative noise  have been established  in   \cite{kim}.

Following Section III.8.4 in \cite{BD1}   we introduce   $U = L^2(D)$ (the $L^2(D)$ space is defined with respect to the Lebesgue measure);  the operator $\Lambda = \triangle^2$, with domain
$$
\dd (\Lambda) = H^4(D) \cap H^2_0(D)
$$
is a positive self-adjoint operator ($H^2_0(D)$ is the closure of $C_0^{\infty}(D)$ in $H^2(D)$, see Definition 13.4.6 in \cite{TW}). One can prove that $\dd (\Lambda^{1/2}) = H^2_0(D)$
 (see page 172 in \cite{BD1}). 
The topological dual of $H^2_0(D)$ will be indicated by $H^{-2}(D)$.  
In order to check that $\Lambda$ satisfies Hypothesis \ref{base} we refer to \cite{CH}. Indeed a classical  result by Courant (see page 460 of \cite{CH}) states that the  eigenvalues $\lambda_n$ of $\Lambda$ have the asymptotic behaviour
\begin{equation}
 \label{cou}
 \lambda_n \sim \frac{(4 \pi n)^2}{f^2}
 \end{equation} 
where $f$ denotes the area of $D$ (such behaviour depends on the size but not on the shape of the plate). It follows that $\Lambda^{-1}$ is a trace class operator in $L^2(D)$.
 Proceeding as in Sections 2 and 3.1 we consider an  extension  
of  $\Lambda$ to   $H^{-2}(D)$  with domain $H^2_0(D)$. 
 
The initial conditions of \eqref{plate} are   
 $x_0 \in   L^{2}\left( D  \right)$, $x_1 \in H_{}^{-2}\left(  D  \right)$.

The reference  Hilbert space for the solution $ X_\tau(\xi):=
 \Big(
 \begin{array}{l}
  y(\tau, \xi)\\
  \frac{\partial}{\partial \tau}y(\tau,\xi)
 \end{array}
 \Big)$  
is  $H= L^{2}\left( D  \right) \times H_{}^{-2}\left(  D  \right)$.  
By considering $G: L^2(D)  \longrightarrow H$,
$G u=\left(
 \begin{array}
 [c]{c}%
 0\\
 u
 \end{array}
 \right)
 $
(cf. \eqref{G}) 
we  rewrite \eqref{plate} 
in the abstract form \eqref{wa1} with
 $ B(\tau, h)$ $:=
 b(\tau, \cdot , h_1 (\cdot)) $, $h = (h_1, h_2) \in H$. 
 The assumptions we impose on   $b$ to verify Hypothesis \ref{ip-b1} and get well-posedness for \eqref{plate} are similar to 
Hypothesis \ref{ip-b}.
 \begin{hypothesis}\label{ip-b3}
The function $b : \left[  0,T\right]  \times D
\times\mathbb{R} \to \R $  is measurable and, for  $\tau
\in\left[  0,T\right]  $ and a.e. $\xi\in D  ,$ the map $b\left(
\tau,\xi,\cdot\right)  :\mathbb{\mathbb{R}}  \rightarrow\mathbb{R}$ is continuous. 
 There exists $c_{1}$ bounded and measurable 
on $D $, $\alpha \in (2/3,1)$,  such that,
 for  $\tau\in\left[
0,T\right]  $ and for a.e. $\xi\in D  ,$   
\[
\left|  b\left(  \tau,\xi,x\right)  -b\left(  \tau,\xi,y\right)  \right|
\leq c_{1}\left(  \xi\right)  \left|  x-y\right|^\alpha,
\]
$ x,\,y \in \R$. Moreover
$\left|  b\left(  \tau,\xi,x\right)  \right|  \leq c_{2}\left(  \xi\right)
,$ for $\tau \in [0,T]$ and a.e. $\xi \in D$,   with $c_{2} \in L^2(D)  $.
\end{hypothesis}

\subsection{A counterexample to well-posedness in the deterministic case}
\label{subsec:counterexamples}

Let us consider the following semilinear deterministic wave equation for $\tau \in [0,T]$:
\begin{equation}
\left\lbrace
\begin{array}
 [c]{l}%
 \frac{\partial^{2} y}{\partial\tau^{2}}\left(  \tau,\xi\right)  =\frac
 {\partial^{2}y}{\partial\xi^{2}}\left(  \tau,\xi\right)  +b\left(  
 \xi, y(\tau, \xi)\right) \\
 y\left(  \tau,0\right)  =y\left(  \tau,\pi\right)  =0,\\
 y\left(0,  \xi\right)  =0  , \;\; \;
 \frac{\partial y}{\partial\tau}\left(0, \xi\right)  =0,\;\; \xi \in [0, \pi].
 \end{array}
 \right.  \label{waveequation-holder-det1}%
 \end{equation}
with 
\begin{align*}
 b\left(  
 \xi, y \right)&=56  \sqrt[4]{\sin \xi\, y^3
 }\, I_{\left\lbrace\vert y \vert<2T^8\right\rbrace}+
 y \, I_{\left\lbrace\vert y \vert<2T^8\right\rbrace} 
+ 56  \sqrt[4]{8T^{24}\sin \xi }\, I_{\left\lbrace\vert y \vert\geq2T^8\right\rbrace}+
 2T^8\, I_{\left\lbrace\vert y \vert\geq 2T^8\right\rbrace},
\end{align*}
where $\xi \in [0, \pi]$, $y \in \R$;  $I_A$ is the indicator function of a set $A \subset \R$.
 Notice that $b$, which is independent of $\tau$ satisfies Hypothesis \ref{ip-b}.
It turns out that $y(\tau,\xi)\equiv 0$ and 
 $y(\tau,\xi)=\tau^8 \sin \xi $ are both 
 solutions to equation 
(\ref{waveequation-holder-det1}).


 \section{The $J$-valued transition semigroup for the stochastic wave equation}

 Let $J$ be a real separable Hilbert space. As in Section 2  we consider   the Hilbert spaces  
 $$
 H= U \times  V' , \;\; \; K = V \times U \subset H.
 $$
 Moreover   $(e_j)$ is a basis in $U$ such that   $(e_j) \subset \mathcal{D} (\Lambda) \subset U$ and    
  \begin{equation}\label{basee}
 \Lambda e_j = \lambda_j e_j, \;\; \lambda_j >0,\;\; j \ge 1; \;\;\;  \sum_{j \ge 1} \lambda_j^{-1} < \infty.  
  \end{equation}
  We will  prove some   regularizing effects for the 
Ornstein-Uhlenbeck semigroup $(R_t)$ related to stochastic wave   equation \eqref{wa1} with $B=0$ and acting on $J$-valued functions $\Phi$.  
Recall that 
 \begin{equation}
 R_{\tau}\left[  \Phi\right]  \left(  x\right)  =  R_{\tau} \Phi \left(  x\right)  =\mathbb{E}\Phi\left(
X_\tau^{0,x}  \right),\quad \Phi \in B_b(H,J), \;\; x \in H, \; \tau \ge 0,
\label{wave-Hsemigroup}
\end{equation}
where $X $, defined by (\ref{mild}), is the Ornstein-Uhlenbeck process (cf. \cite{DPF}). 
Since $X$ is time homogeneous, we have
 \[
 R_{\tau - t}\left[  \Phi\right]  \left(  x\right)  =\mathbb{E}\Phi\left(
X_\tau^{t,x}  \right),
\;\;  
\Phi \in B_b(H,J) , 
%
 \]
 $\tau \ge t \ge 0$, $x \in H$.   Similarly,  
  we  consider  the usual  Ornstein-Uhlenbeck semigroup $(P_t)$ acting on scalar functions $\phi \in B_b(H)$: 
   \begin{equation} \label{ptt1}
 P_{\tau}\left[  \phi\right]  \left(  x\right)  =  P_{\tau} \phi \left(  x\right)  =\mathbb{E}\phi\left(
X_\tau^{0,x}  \right),\quad \phi \in B_b(H), \;\; \tau \ge 0.
 \end{equation}
 Using also the results in Appendix, 
  for $t>0,$ we show the differentiability  
  of $R_t \Phi$ along  the directions of $K$. Moreover, 
   we  prove that,  for any $x  \in H,$ $k \in K$, $t>0$,
  $$
   \sum_{m \ge 1} \sup_{|a|_U =1}| \nabla_k \nabla_{a}^G  P_t [\, \langle \Phi (\cdot ),  f_m\rangle_J \, ](x)|^2  
  $$
  is finite (here $(f_m)$ denotes  any basis of $J$) and we provide a bound independent of $x$ and $k$ (see  Lemma \ref{der12a} and compare with  Chapter 6 of  \cite{DP3} and Section 3 of \cite{DPF}). 

  \vskip 1mm
 In order to  study  differentiability properties of $R_t[\Phi ]$ for $t>0$ we fix   some basic definitions.  
  If $F: H = U \times V'\to J$ is G\^ateaux  differentiable at $x \in H$ we denote by $\nabla F(x) \in L(H,J)$ its G\^ateaux  derivative at $x$ and by $\nabla_h F(x)= \nabla F(x) h$ its directional derivative along the direction $h \in H$: 
$$ 
 \lim_{s \to 0}  \frac{ F (x+ s h )- F (x)}{s} 
= \nabla_h F (x), \,\;\; x \in H,\, h \in H. 
$$
We say that $F:  H \to J$ is differentiable along the subspace $K = V \times U \subset H$ if there exists at any $x \in H$ the   directional derivative along any   direction $k \in K$ (i.e.,  $\lim_{s \to 0}  \frac{ F (x+ s k )- F (x)}{s}\in J$, for any $x \in H$, $k \in K$).  We denote the directional  derivative at $x$ along the direction $k \in K$ as  
\begin{gather*}   
\nabla_k F(x) \in J. 
\end{gather*}  
 If in addition 
 $$
 k \mapsto \nabla_k F(x) \;\; \text{  belongs to $L(K,J)$} 
 $$  
 we indicate  such linear  operator with  $\nabla^K F(x)$.    We say that $F$ is $K$-differentiable on $H$  if  it is   differentiable along the subspace  $K$ and there exists  $\nabla^K F(x) \in L(K,J)$ for any $x \in H$ (if $J=\R$ then $\nabla^K F(x)$ can be identified with an element in $K$ by the Riesz theorem).

 Note that the concept of   differentiability along subspaces    arises naturally in the Malliavin Calculus  (see also the related concept of  Gross differentiability; we refer to  \cite{ShigekawaBook} and the references therein).      
   
Let $G : U \to H$, $Ga = \Big(
\begin{array}
[c]{c}
0\\
a
\end{array}
\Big) \in K \subset H$, $a \in U$. If $F: H \to J $ is differentiable along the subspace $G (U)$ we set 
\begin{equation} \label{dg}
\nabla_a^G F(x) = \nabla_{Ga} F(x) \in J,\;\;\;\; a \in U,\; x \in H.
\end{equation}
If in addition $a \mapsto \nabla_a^G F(x)$ belongs to $L(U,J)$ we denote such linear  operator with  $\nabla^G F(x)$. We say that $F$ is $G$-differentiable on $H$ if it is differentiable along the subspace $G (U)$   and  there exists  $\nabla^G F(x) \in L(U,J)$ for any $x \in H$.  
   
  {  Note that  if $F: H \to J$ is $K$-differentiable on $H$ then it is also 
   $G$-differentiable on $H$ and   $\nabla^G F(x) = \nabla^K F(x) G \in L(U,J)$.  
     } 
 

  
 \vskip 1mm
For $0<\alpha<1$, we introduce  the space $C_K^\alpha(H,J)$ consisting   of all functions $F$ in $C_b(H,J)$ which are  $\alpha$-H\"older continuous along $K $, i.e., such that       
\begin{equation}\label{xii}
[F]_{\alpha, K} = \sup_{x \in H = U \times V',\; k \in K , k \not = 0} \frac{| F(x+ k) - F(x)|_J}{|k|_K^{\alpha}} < \infty.        
\end{equation} 
It is a Banach space endowed with the  norm $\| \cdot \|_{\alpha, K} = \| \cdot\|_{\infty} + [\cdot]_{\alpha, K}$.    
    

\subsection{Interpolation results involving $K$-differentiable functions
}

 We first   introduce  a function space related to the $K$-differentiability.
  We say that $f \in C_K^1(  H , J)$ if $f \in C_b(H, J)$, 
$f$ is $K$-differentiable on $H$ 
 and $\nabla^K f : H \to L(K, J)$ is uniformly continuous and bounded.  
 This is a Banach space endowed with the norm 
 \begin{gather*}    
 \| f \|_{C_K^1(  H , J)} = \| \nabla^K f  \|_{\infty} +  \| f  \|_{\infty},\;\; f \in C_K^1(  H,J),  
  \end{gather*}  
 setting $\sup_{x \in H} | \nabla^K f(x) |_J 
 = \| \nabla^K f  \|_{\infty}$. 
  When $J = \R$ we set  $ C_K^1(  H ,\R) =   C_K^1(  H )$. 
 Recall that for $f \in  C_K^1(  H )$ one has: $\nabla^K f : H \to K$  uniformly continuous and bounded.  
 
   Let us consider  the following  operator $Q: H \to H$,
 \begin{equation}\label{sdr}
 Q= \left(
\begin{array}
[c]{cc}%
\Lambda^{-1} & 0\\
0 &  \Lambda^{-1}
\end{array} 
\right),\;\;\;\; Q h =  \Big(
\begin{array}
[c]{c}
\Lambda^{-1} h_1 \\
 \Lambda^{-1} h_2
\end{array}
\Big),\;\;\; h = (h_1, h_2 ) \in H = U \times V'.  
\end{equation} 
Let   $(e_j)$ be the  basis in $U$ defined in \eqref{basee}. Then $(\sqrt{\lambda_j} e_j) $ is a basis of $V'$ and $\{(e_j,0)\}_{j \ge 1} \cup
\{(0, \sqrt{\lambda_j} e_j)\}_{j \ge 1} $ is a basis of $H$. 
 It is not difficult to check that  $Q$ is a symmetric positive trace class operator  and that  
\begin{equation}\label{coincide}
\text {$C^1_K(H)$ coincides  with the space $C^1_Q(H)$ introduced   in \cite{canndap} with equivalence of norms.} 
\end{equation}
 To this purpose we note that  $Q^{1/2} H = K$. Then we  consider   conditions (i), (ii) and (iii) used in the definition of $C^1_Q(H)$ in Section 2.1 of   \cite{canndap}. Let  $f \in C_b (H)$. Condition (i) says that there exist all the directional derivatives of $f$ in the directions of $K = Q^{1/2} H$.
Let $k = Q^{1/2}h$ with $h \in H $. The directional derivative in $x$  along the direction  $k$ is denoted by 
$$
 \nabla_k f(x) =  \nabla_{Q^{1/2}h} f(x). 
$$
Condition   (ii) says that for any $x \in H$,  there exists $D_Q f(x) \in H$ such that    
\begin{gather*}
 \nabla_{Q^{1/2}h} f(x) = \langle D_Q f(x), h \rangle_H. 
\end{gather*}
 If  $k \in K$ then  $k = Q^{1/2} h$  for a unique  $h \in H.$ 
 We have   $\langle D_Q f(x), h \rangle_H = \langle  Q^{1/2} D_Q f(x), Q^{1/2} h \rangle_K  = \langle Q^{1/2} D_Q f(x), k \rangle_K  $. Thus
condition  (ii) is equivalent to say that  $ k \mapsto  \nabla_k f(x)$ is  linear and continuous from $K$ into $\R$. Moreover such linear functional can be identified with  $Q^{1/2} D_Q f(x)$. According to our previous   notation we can  write  
$$
\nabla^{K} f(x) = Q^{1/2} D_Q f(x),\;\;\; x \in H.
$$
Condition (iii) requires that the mapping: $x \mapsto D_Q f(x)$ is uniformly continuous and bounded from $H$ into $H$. This is equivalent to say that 
 the mapping: $x \mapsto Q^{1/2} D_Q f(x) = \nabla^{K} f(x)$ is uniformly continuous and bounded from $H$ into $K$. This shows \eqref{coincide}.

\vskip 1 mm    
Similarly to \cite{canndap} we  define, for $0<\alpha<1$, the space $C_K^\alpha(H, J) $  of all functions $f$ in $C_b(H, J)$ such that   
\begin{equation}\label{kk1}
[f]_{\alpha, K} = \sup_{k',\; k \in K , k-k' \not = 0} \frac{| f(k) - f(k')|_J}{|k - k'|_K^{\alpha}}  < \infty.  
\end{equation}
It is a Banach space endowed with the  norm $\| \cdot \|_{\alpha, K} = \| \cdot\|_{\infty} + [\cdot]_{\alpha, K}$, where $ \| f\|_{\infty} = \sup_{x \in H} |f(x)|_J$.        
  Note that $C_b^\alpha(H, J) \subset C_K^\alpha(H, J) $  and in general the inclusion is strict (cf. Remark \ref{pz}). 
 
 \begin{remark} \label{tes} {\em 
  Condition \eqref{kk1} is equivalent   to  
\begin{equation*}\   
\sup_{x \in H = U \times V',\; k \in K , k \not = 0} \frac{| f(x+ k) - f(x)|_J}{|k|_K^{\alpha}}  < \infty    
\end{equation*} 
  (cf. \eqref{xii}). 
 Indeed  if $x \in H$ and $k \in K$ there exists a sequence $(k_n) \subset K$ such that $k_n \to x$ in $H$. Then by \eqref{kk1} we find 
 $
| f(k_n+ k) - f(k_n)|_J     \le C |k|_K^{\alpha}.
$  Passing to the limit as $n \to \infty$ we obtain ${| f(x+ k) - f(x)|_J} \le C |k|_K^{\alpha}.$ 
}    
 \end{remark}
 The space  $C_K^\alpha(H) = C_K^\alpha(H, \R)$  coincides with the space $C_Q^\alpha(H)$
introduced in Section 2.2 in \cite{canndap} as the space of all functions $f \in C_b(H) $  such that   
\begin{gather*} 
[f]_{\alpha , Q} = \sup_{h',\; h \in H , h-h' \not = 0} \; {| f( Q^{1/2} h ) - f(Q^{1/2} h')|}\, \cdot  {|h - h'|_H^{-\alpha}}  < \infty 
\end{gather*}
with equivalence of norms.   By \cite{canndap} we now obtain the following useful interpolation result.
\begin{lemma}\label{interpola}
We have, for $\alpha \in (0,1)$,  with equivalence of norms,
\begin{equation} \label{internos}
(C_b(H) , C^1_K(H))_{\alpha , \infty} =  C^{\alpha}_K(H).
\end{equation}  
\end{lemma}
\begin{proof} The result is proved  in Proposition 2.1 in \cite{canndap} in the  form 
 \begin{gather*}
(C_b(H) , C^1_Q(H))_{\alpha , \infty} =  C^{\alpha}_Q(H).
\end{gather*}
 We only recall that  $f \in X_{\alpha} = (C_b(H) , C^1_K(H))_{\alpha , \infty}  $  if $\| f \|_{X_{\alpha}} = \sup_{t \in (0,1]} t^{- \alpha}  L(t,f) < \infty$ where 
 $ L(t,f) =\inf \{ \|a \|_{C_b(H)} + t \|b \|_{C^1_K(H)},$ $f = a + b,$ $a \in C_b(H),\, b \in C_K^1(H) \}$ (see, for instance,  Section 2.3 in \cite{DPsecond}).
 \end{proof}   
    
\begin{remark} \label{pz}   {\em 
   Theorem  3.1 in \cite{prizambo}    implies  that $C^{\alpha}_b(H)$ (the space of real $\alpha$-H\"older continuous and bounded functions defined on $H$) is strictly included in $C^{\alpha}_Q(H)$,  $\alpha \in (0,1) $. 
  Indeed $C^{\alpha}_b(H)$ is contained in  the interpolation space ${\cal D}_{\cal A} (\alpha/2, \infty)$ (see the notation in \cite{prizambo})  which by Theorem 3.1 is strictly included  in  $C^{\alpha}_Q(H)$. 
 }       
 \end{remark}

 When $J$ is infinite dimensional it is an open problem to characterize both $\big(C_b(H,J) , C^1_K(H,J) \big)_{\alpha , \infty}$ and 
 $\big(C_b(H,J) , C^1_b(H,J) \big)_{\alpha , \infty}$. However we can prove  the following inclusion which will be important for the sequel (see in particular the proof of Lemma \ref{der12ab}).

\begin{lemma} \label{intervector} For any real separable Hilbert space $J$ we have
\begin{equation}\label{inter}
 C^{\alpha}_b(H, J) \subset  \big(C_b(H,J) , C^1_K(H,J) \big)_{\alpha , \infty},\;\; \alpha \in (0,1),  \;\; \text{ with a continous inclusion.}
\end{equation}
 \end{lemma}
 \begin{proof}
 We take into account  Remark 2.3.1 in \cite{DPsecond}.
   Let  
 $f \in  C^{\alpha}_b(H, J)$ and $t \in (0,1]$. We prove that 
 there exists $a_t \in C_b(H,J)$ and $b_t \in C^1_K(H,J)$ such that $f = a_t + b_t$ and
\begin{equation}\label{sfg}
 \| a_t\|_{ C_b(H,J)} + t \| b_t\|_{C^1_K(H,J)} \le c \| f\|_{C^{\alpha}_b(H, J)} \,  t^{\alpha}
\end{equation} 
with $c>0$ independent of $t$ and $f$. This gives \eqref{inter}.

Let us consider the trace class operator $Q: H \to H$ given in \eqref{sdr}. Recall that 
$Q $ is   injective and $Q^{1/2}(H) = K$.
 As in Chapter 3 of \cite{DPsecond} we consider the heat semigroup $(V_t)$ acting on functions in $C_b(H,J)$:
\begin{gather*}
 V_r g(x) = \int_{H} g(x+ y) \caln(0,rQ)(dy), \;\; x \in H,\;\; g \in  
 C_b(H,J),\; r \ge 0.   
\end{gather*} 
 For $t \in (0,T]$ we set  
\begin{gather*}
 a_t = f - V_{t^2} f,\;\;\; b_t =   V_{t^2} f
\end{gather*}
and prove that \eqref{sfg} holds. Let us first consider $a_t$. It is easy to prove that $a_t \in C_b(H,J) $. Moreover,  we have
\begin{gather*}
 |a_t (x)|_J \le \int_{H} |f(x+ y) - f(x) |_J \, \caln(0,t^2 Q)(dy)
 \le \| f\|_{C^{\alpha}_b(H, J)}   \int_{H} | t y |_J^{\alpha} \, \caln(0, Q)(dy) \le c_{\alpha} \| f\|_{C^{\alpha}_b(H, J)}  t^{\alpha}.
\end{gather*} 
To prove that  $b_t \in C^1_K(H,J)$  we consider
 $k = Q^{1/2}h \in K$ with $h = Q^{-1/2} k\in H $. Arguing as in Theorem 3.3.3 in \cite{DPsecond}, using the Cameron-Martin theorem, one can prove that, for any $x \in H,$ there exists the directional derivative
\begin{gather*} 
 \nabla_k b_t (x) 
  = \lim_{s \to 0}  \frac{ V_{t^2} f (x+ s Q^{1/2}h )- V_{t^2} f (x)}{s} 
\\ 
= \frac{1}{t}\int_{H} f(x+ y)  \langle (t^{2} Q)^{-1/2} y , h \rangle \caln(0,t^2 Q)(dy) 
= \frac{1}{t}\int_{H} f(x+ y)  \langle (t^{2} Q)^{-1/2} y ,  Q^{-1/2} k \rangle \caln(0,t^2 Q)(dy).
  \end{gather*} 
 It is not difficult to prove that $k \mapsto  \nabla_k b_t (x) $ is linear and continuous from $K $ into $J$. We  note that
 \begin{gather*}
 |\nabla_k b_t (x)|_J^2 \le  \frac{1}{t^2} \| f\|_{\infty}^2 \int_{H} |  \langle (t^{2} Q)^{-1/2} y ,  Q^{-1/2} k \rangle|^2\,  \caln(0,t^2 Q)(dy)
\le 
 \frac{1}{t^2} \| f\|_{\infty}^2 \, | Q^{-1/2} k|^2 = \frac{|k|_K^2}{t^2} \| f\|_{\infty}^2.
 \end{gather*} 
 Writing, for $k \in K$, $x,x' \in H$, 
 \begin{gather*}
|\nabla_k b_t (x) - \nabla_k b_t (x')|_J^2 \le  \frac{1}{t^2}  \int_{H} |  \langle (t^{2} Q)^{-1/2} y ,  Q^{-1/2} k \rangle|^2\,  \caln(0,t^2 Q)(dy)
\\ \cdot  \int_{H} | f(x + y ) - f(x' +y)  |^2_J\,  \caln(0,t^2 Q)(dy)
\le  \frac{|k|_K^2}{t^2} \| f\|_{C^{\alpha}_b(H, J)}^2\,  |x-x'|^{2\alpha}
\end{gather*} 
 (we have used that $| Q^{-1/2} k|= |k|_K$)
 we check easily that $\nabla^K b_t : H \to L(K, J)$ is uniformly continuous and bounded.   Finally,  
  since   
  \begin{gather*} 
 \frac{1}{t}\int_{H} f(x)  \langle (t^{2} Q)^{-1/2} y , h \rangle \caln(0,t^2 Q)(dy) =0,
\end{gather*}
  we obtain by the Cauchy-Schwarz inequality, for any $k \in K,$  $x \in H,$ 
\begin{gather*}
 |\nabla_k b_t (x)|_J^2  = \Big |  \frac{1}{t}  \int_{H}   \langle (t^{2} Q)^{-1/2} y ,  Q^{-1/2} k \rangle \, [f(x + y ) - f(x)]\caln(0,t^2 Q)(dy) \Big|_J^2   
 \\ 
  \le \| f\|_{C^{\alpha}_b(H, J)}^2   \frac{1}{t^2}  \int_{H} |  \langle (t^{2} Q)^{-1/2} y ,  Q^{-1/2} k \rangle|^2\,  \caln(0,t^2 Q)(dy) \cdot \int_{H} | t y |_J^{2\alpha} \, \caln(0, Q)(dy)  
  \\ \le c_{\alpha}    \| f\|_{C^{\alpha}_b(H, J)}^2  \, t^{2 \alpha -2}.   
\end{gather*} 
Collecting the previous estimates we get    \eqref{sfg} 
 \end{proof}

\subsection {Regularizing properties of the transition semigroup 
}


 We first collect some useful properties proved in Section \ref{refer} by classical  control theoretic arguments for  the abstract wave equation. 
  \begin{remark}\label{rem-wave operator2} {\em   
 For any $t>0$, we have  
 $e^{tA}( K) =  Q_t^{1/2}(H)=K$
 and  
  $Q_t^{-1/2} e^{tA} $      
 belongs to $L(K,H)$.  
  Let $T>0$. 
There exists $c= c_T>0$ such that  for any $t \in (0,T]$  
  we have   
\begin{gather} \label{se19}   
 |Q^{-1/2}_t e^{tA} k|_H \le   \frac{c}{t^{3/2}} |k|_K,\; \;\; k \in K = V \times U;  
\\
 \label{vv} |Q^{-1/2}_t e^{tA} Ga|_H     \le \frac{c}{t^{1/2}} |a|_U,\; \;\; a \in  U. \qed 
 \end{gather} 
  } 
  \end{remark}    
Let $\Phi \in B_b(H,J )$ and $x \in H$. Arguing as in the proof of Theorem 6.2.2 of \cite{DP3} (similar arguments are used in Section 9.4 of \cite{DPsecond} and Section 3 of \cite{DPF})    one can prove the existence of the directional derivative of    
$R_t[\Phi]$ along the directions of $K$: 
$$ 
 \lim_{s \to 0}  \frac{ R_t[\Phi](x+ s k )-R_t[\Phi](x)}{s}   
$$
\begin{equation}\label{nabla-R} 
= \nabla_k R_t[\Phi](x) =
 \int_H
\<Q_t^{-\frac{1}{2}}e^{tA}k,Q_t^{-\frac{1}{2}}y\>\, \Phi(e^{tA}x+y)\, \caln(0,Q_t)(dy),\;  { k\in K, } \,\;  t>0. 
\end{equation}  
 In the sequel we often write $\mu_t = \caln(0,Q_t)$ and $|\cdot|_H = |\cdot|$.

 In the next result we will use \eqref{nabla-R} together with the estimates in Remark \ref{rem-wave operator2}. 
   \begin{lemma}\label{teo:derHsemigroup} Assume Hypothesis \ref{base} and let  $R= (R_t)$ be the OU  semigroup defined 
in (\ref{wave-Hsemigroup}). If   $\Phi \in {B}_b(H,J)$ and $t>0$
then $R_t \Phi$ is $K$-differentiable on $H$.
The  directional derivative  $\nabla_k R_t[\Phi](x)\in J$   is given by \eqref{nabla-R}, for $x \in H$.   
  In particular    $R_t[\Phi]$  is $G$-differentiable on $H$;
further     
\begin{equation}\label{nablaG-R}  
\nabla^G_{a} R_t[\Phi](x) =\int_H 
\<Q_t^{-\frac{1}{2}}e^{tA}Ga,Q_t^{-\frac{1}{2}}y\> \, \Phi(e^{tA}x+y)\caln(0,Q_t)(dy),\,\;\;  a\in U. 
\end{equation}
Moreover,  for $t \in (0,T]$,  we have
\begin{align}\label{estimatess}
& \sup_{x \in H}| \nabla_k  R_t[\Phi](x)  \vert_J \leq \frac{c}{t^{\frac{3}{2}}}\Vert \Phi\Vert_\infty\vert k\vert_K,\,k\in K;
\\
 & \sup_{x \in H} \vert \nabla^G_{a} R_t[\Phi](x) \vert_J \leq 
 \frac{c}{t^{\frac{1}{2}}}\Vert \Phi\Vert_\infty\vert  G a \vert_{K}  = 
 \frac{c}{t^{\frac{1}{2}}}\Vert \Phi\Vert_\infty\vert  a \vert_{U},\, a \in U. \label{hi11} 
\end{align}
If in addition  $\Phi \in {C}_b(H,J)$  
then 
$ \nabla^K  R_t [\Phi] \in C_b(H, L(K,J)) $, $\nabla^G  R_t[\Phi] \in C_b(H, L(U,J)) $  for $t >0$.
\end{lemma}
\dim   
 Let us fix $t \in (0,T] $ and $x \in H$.  
The integral in (\ref{nabla-R}) defines a linear operator in
$L(K,J)$.       
Let    
\[
 I_{t,x}k:=\int_H
\<Q_t^{-\frac{1}{2}}e^{tA}k,Q_t^{-\frac{1}{2}}y\>\Phi(e^{tA}x+y)\caln(0,Q_t)(dy),\;\; k \in K.
\]
We have the well-known estimate (cf. the proof of Theorem 6.2.2  in \cite{DPsecond}) 
\begin{align} \label{rit1}
 \vert I_{t,x}k\vert_J    &\leq\int_H
\vert\<Q_t^{-\frac{1}{2}}e^{tA}k,Q_t^{-\frac{1}{2}}y\>\Phi(e^{tA}x+y)\vert_J   \caln(0,Q_t)(dy)
\\ 
\nonumber
&\leq   \Big(\int_H \vert \Phi(e^{tA}x+y) \vert^2_J 
\caln(0,Q_t)(dy) \Big)^{1/2} \cdot     
 \Big(\int_H\vert\<Q_t^{-\frac{1}{2}}e^{tA}k,Q_t^{-\frac{1}{2}}y\>\vert^2
\caln(0,Q_t)(dy) \Big)^{1/2}   
\\ \nonumber 
&\leq\Vert\Phi\Vert_\infty 
 \Big(\int_H\vert\<Q_t^{-\frac{1}{2}}e^{tA}k,Q_t^{-\frac{1}{2}}y\>\vert^2
\caln(0,Q_t)(dy) \Big)^{1/2} 
=\Vert\Phi\Vert_\infty\vert Q_t^{-\frac{1}{2}}e^{tA}k\vert_H
\leq   \frac{c}{t^{\frac{3}{2}}}\Vert \Phi\Vert_\infty\vert k\vert_K. 
\end{align}
 Similarly,  we get \eqref{hi11} using \eqref{vv} since
 \begin{equation}\label{s} 
 \int_H 
 \vert\<Q_t^{-\frac{1}{2}}e^{tA} G a,Q_t^{-\frac{1}{2}}y\>\vert^2
\caln(0,Q_t)(dy) \le |Q_t^{-\frac{1}{2}}e^{tA} G a|^2 \le  \frac{c |a|^2_U }{t},\;\;   a \in {U}.    
\end{equation}   
Computing the directional derivative as in \eqref{nabla-R}
we  
obtain 
the differentiability of $R_t[\Phi]$ along the directions of $K$ at $x$.
 We also obtain that $R_t[\Phi]$ is $G$-differentiable and $K$-differentiable  on $H$. 

If  $\Phi \in {C}_b(H,J)$    
 we  compute, for any $k \in K$, $|k|_K=1$, $z \in H$, 
\begin{gather*}
 |I_{t,x}k - I_{t, x+ z}k |^2_J  
\\
\leq \Big | \int_H
\<Q_t^{-\frac{1}{2}}e^{tA}k,Q_t^{-\frac{1}{2}}y\> \, [\Phi(e^{tA}x+y) - \Phi(e^{tA}(x+z) +y) ]\, \caln(0,Q_t)(dy)
\Big|^2_J
\\ 
\leq     
 \int_H\vert\<Q_t^{-\frac{1}{2}}e^{tA}k,Q_t^{-\frac{1}{2}}y\>\vert^2
\caln(0,Q_t)(dy)  \, \int_H | \Phi(e^{tA}x+y) - \Phi(e^{tA}x + e^{tA}z +y) |^2_J \caln(0,Q_t)(dy)
\\
\le \frac{c^2 |k|^2_K}{t^{3} } \, 
\int_H \sup_{x\in H}| \Phi(e^{tA}x+y) - \Phi(e^{tA}x + e^{tA}z +y) |^2_J \, \caln(0,Q_t)(dy)
\end{gather*}
and so by the dominated convergence  theorem  we obtain  easily  
\begin{equation} \label{s2}
\lim_{z \to 0}  \; \sup_{y \in H} \, \sup_{|k|_K =1} |I_{t,y}k - I_{t, y+ z}k | _J=0.
\end{equation}  
By \eqref{s2} we deduce that $
\nabla^K  R_t[\Phi] \in C_b(H, L(K,J))$ and   
 $\nabla^G  R_t[\Phi] \in C_b(H, L(U,J)) $.   \qed 

  
 In a similar way   we get
\begin{lemma}\label{der1}   
Under the assumptions of Lemma \ref{teo:derHsemigroup} 
let $t  >0$. 
 If $\Phi \in C_b(H, J)$ and $\xi \in U$ the mapping: 
$$ 
x \mapsto \nabla_{\xi}^G  R_t[\Phi](x)
$$  
with values in $J $ is  $K$-differentiable on $H$.    
The second order directional derivatives are
\begin{equation}\label{nabla-R1}    \nabla_k \nabla_{\xi}^G  R_t[\Phi](x)
=\int_H
\big (\< \Gamma_t k,Q_t^{-\frac{1}{2}}y\> \, \< \Gamma_t G \xi ,Q_t^{-\frac{1}{2}}y\> -  \< \Gamma_t k,  \Gamma_t G \xi \> \big)
\Phi(e^{tA}x+y)\mu_t(dy),  
\end{equation}
for $x\in H,  { k \in K}$, $\xi \in U$.
Moreover, for each $x \in H$, $k \in K$,        the map:  $ \xi \to  \nabla_k \nabla_{ \xi \,}^G R_t[\Phi](x)$ belongs to   $ L(U,J)$ and,  for any $t \in (0,T]$,   
\begin{gather}\label{sii}
  \sup_{x \in H}  \Vert  \nabla_k \nabla_{\cdot }^G R_t[\Phi](x) \Vert_{L(U,J)} \leq \frac{c_T |k|_{K}}{t^{2}}\Vert \Phi\Vert_\infty, \,  { k\in K. } 
\end{gather}  
\begin{gather}\label{estimates1}
\lim_{x \to 0}    \sup_{y \in H }  \, \|   \nabla^K \nabla_{\xi}^G R_t[\Phi](x +y) -  \nabla \nabla_{\xi}^G R_t[\Phi](y) \|_{L(K,J)}   
\\ \nonumber = 
\lim_{x \to 0}  \sup_{y \in H } \sup_{|k|_K = 1} \, |   \nabla_k \nabla_{\xi}^G R_t[\Phi](x +y) -  \nabla_k \nabla_{\xi}^G R_t[\Phi](y) |_{J}   =0,\;\; \xi \in U.
\end{gather}     
\end{lemma}
\begin{proof}  
Let us  fix $T>0$ and  $t \in (0,T]$, $x \in H$. Let  $\xi \in U$. First  define  $\Gamma_{t,x,k, \xi}$ as the integral in the right hand side of \eqref{nabla-R1}.
 Proceeding as in the proof      of Lemma \ref{teo:derHsemigroup} it is not difficult to show that 
\begin{equation} \label{233} 
k \mapsto \Gamma_{t,x,k, \cdot} 
\end{equation} 
is defined from  $K$ into $L(U,J)$.       Using \eqref{se19}  and the Cauchy-Schwarz inequality, we get  
\begin{gather} \label{w2} 
 |\Gamma_{t,x,k, \xi} |^2_J 
\le  \Big | \int_H
[\< \Gamma_t k ,Q_t^{-\frac{1}{2}}y\> \, \< \Gamma_t G \xi ,Q_t^{-\frac{1}{2}}y\> -  \< \Gamma_t k,  \Gamma_t G \xi  \> \big ]
  \,        
\Phi(e^{tA}x+y) \, \mu_t(dy) \Big|_J^2
\\ \nonumber 
\le \int_H
\big |\< \Gamma_t k ,Q_t^{-\frac{1}{2}}y\> \, \< \Gamma_t G \xi ,Q_t^{-\frac{1}{2}}y\> -  \< \Gamma_t k,  \Gamma_t G \xi  \> \big
 |^2 \mu_t(dy)
  \cdot \int_H
\,        
  |\Phi(e^{tA}x+y)|^2_J \, \mu_t(dy)
\\  \nonumber   
\le  \frac{c |k|^2_K }{t^{4}}   |\xi|^2_U 
\Vert \Phi\Vert_\infty^2.  
\end{gather} 
Thus we have proved \eqref{sii}.   
  Arguing as in Section 9.4 of \cite{DPsecond} and Section 3 in \cite{DPF} we find that   
\[
 \lim_{s \to 0}  \frac{ \nabla^G_{\xi} R_t[\Phi](x+ s k )- \nabla^G_{\xi} R_t[\Phi](x)}{s} 
 =   \Gamma_{t,x,k, \xi}, \;\; k \in K.
\]
Moreover, for any $z \in H$, $\xi \in U$,
\begin{gather*}
|\Gamma_{t,x, k, \xi} - \Gamma_{t, x+ z, k, \xi} |^2_J 
\\
   =  \Big |  \int_H
 \big (\< \Gamma_t k,Q_t^{-\frac{1}{2}}y\> \, \< \Gamma_t G \xi ,Q_t^{-\frac{1}{2}}y\> -  \< \Gamma_t k,  \Gamma_t G \xi \> \big) \, [ \Phi(e^{tA}x + e^{tA}z  +y) - \Phi(e^{tA}x   +y) ] \mu_t(dy) \Big|^2_J 
 \\
\le  \frac{c |k|^2_K}{t^{4}} 
 \vert  \xi \vert_{U}^2 \int_H | \Phi(e^{tA}x+y) - \Phi(e^{tA}x + e^{tA}z +y) |^2_J \mu_t(dy) 
\end{gather*} 
and so  
\begin{equation} \label{via}
 \lim_{z \to 0} \sup_{x \in H} \sup_{|k|_K =1} |\Gamma_{t,x, k, \xi} - \Gamma_{t, x+ z, k, \xi} |^2_J =0.
\end{equation}
This shows in particular that the mapping 
$
x \mapsto \nabla_{\xi}^G  R_t[\Phi](x)
$
verifies \eqref{estimates1}.   
\end{proof}

Now we  improve the previous estimates in the case when $\Phi$
is  H\"older continuous along the directions of $K$ using 
  Lemma \ref{interpola}.  
  \begin{lemma}\label{der12} Let $T>0.$ Under the assumptions of Lemma \ref{teo:derHsemigroup}   
let  $\Phi \in C^{\alpha}_K (H, J)$, $\alpha \in (0,1)$, see \eqref{xii}. We have all  the assertions of Lemmas \ref{teo:derHsemigroup} and \ref{der1} and  
the following new estimates, for $t \in (0,T]$,
\begin{align}\label{estimates12}   
& \sup_{x \in H} \vert \nabla_k R_t[\Phi](x)\vert_J \,  \leq \frac{c}{t^{\frac{3}{2} (1- \alpha)} }\Vert \Phi\Vert_{\alpha, K} \vert k\vert_K,\,k\in K;
\\   
 &  
 \sup_{x \in H}  \Vert  \nabla_k \nabla_{\cdot}^G R_t[\Phi](x) \Vert_{L(U,J)} 
\leq \frac{c}{t^{\frac{4 - 3\alpha}{2}}}\Vert \Phi\Vert_{\alpha,K} \vert k\vert_K,\,k\in K.  
\nonumber 
\end{align} 
\end{lemma}  
 \begin{proof} Let us fix $t \in (0,T]$, $k \in K$ and $\xi \in U$. Using the OU process  $X $   defined by (\ref{mild}) we can define the Ornstein-Uhlenbeck semigroup $(P_t)$ acting on scalar functions $\phi \in B_b(H)$ (see \eqref{ptt1}).  
 
 For $h \in J$,  we  introduce the scalar function
$
\Phi_h(x) = \langle \Phi(x), h \rangle_J     $, $x \in H$, which belongs to 
$C_K^{\alpha}(H)$ with $\| \Phi_h \|_{{\alpha,K}} $ $\le \| \Phi \|_{{\alpha,K}} \, |h|_J $.
 We note arguing as in Section 3 of \cite{DPF} that   
$$  
\langle \nabla_k R_t[\Phi](x), h \rangle_J  = \nabla_k P_t[\Phi_h](x),\;\; x \in H.
$$ 
To prove the first estimate we  consider the linear operators
$$
 \nabla_k P_t : C^1_K (H) \to  {C}_b (H),\;\;\; \nabla_k P_t : {C}_b(H) \to  {C}_b (H).   
$$  
 These operators are well defined by Lemma \ref{teo:derHsemigroup}.   
 When $\phi \in C^1_K(H)$   we find that  (recall that $e^{tA} k \in K$) 
 \begin{gather*}    
  \nabla_k P_t[\phi](x) = 
\lim_{s \to 0}\int_H
 \frac{  \phi(e^{tA}x + s e^{tA}k +y) -  \phi(e^{tA}x  +y)} {s} \mu_t(dy)
=  \int_H 
   \nabla_{ e^{tA} k} \, \phi(e^{tA}x+y) \,      \mu_t(dy) 
\\   =
 \int_H
\langle    \nabla^{K } \, \phi(e^{tA}x+y), e^{tA} k \rangle_K \,  \mu_t(dy) 
\end{gather*} 
 (by the Riesz theorem we identify  $\nabla^{K } \, \phi(e^{tA}x+y)$ with an element in $K$).
 We get the estimate  
\begin{equation} \label{est1}
\sup_{x \in H} \vert \nabla_k P_t[\phi](x) \vert \leq C \Vert \nabla^K \phi\Vert_\infty \vert k\vert_K,\;\;\; \phi \in C^1_K(H).
\end{equation} 
On the other hand we have (cf. \eqref{estimatess})
\begin{align}\label{estimates24}
 \sup_{x \in H}| \nabla_k  P_t[f](x)  \vert \leq \frac{c}{t^{\frac{3}{2}}}\Vert f \Vert_\infty\vert k\vert_K,\;\;\; f \in C_b(H). 
\end{align} 
Interpolating between \eqref{est1} and \eqref{estimates24} (see Theorem A.1.1 in \cite{DPsecond}) we obtain that, for any $\alpha \in (0,1),$ 
$$   
 \nabla_k P_t  : (C_b(H) , C^1_K(H))_{\alpha , \infty} \, \to \,   {C}_b (H)   
$$ 
is linear and bounded; further 
$$
\sup_{x \in H} \vert \nabla_k P_t[\psi](x)\vert \leq \frac{\tilde c}{t^{\frac{3}{2} (1- \alpha)} }\Vert \psi\Vert_{ (C_b(H) , C^1_K(H))_{\alpha , \infty} } \vert k\vert_K, \;\;\; \psi \in  (C_b(H) , C^1_K(H))_{\alpha , \infty}.
$$  
  Now thanks to  Lemma \ref{interpola}    we deduce         
$$
\sup_{x \in H} \vert \nabla_k P_t[\psi](x)\vert \leq \frac{c}{t^{\frac{3}{2} (1- \alpha)} }\Vert \psi\Vert_{\alpha, K} \, \vert k\vert_K, \;\;\; \psi \in C^{\alpha}_K(H).  
$$ 
If we consider now $\psi = \Phi_h$, we have, for each $x \in H$, $h \in J$,
$$  
|\langle \nabla_k R_t[\Phi](x), h \rangle_J | =
\vert \nabla_k P_t[\Phi_h](x)\vert 
\leq \frac{c}{t^{\frac{3}{2} (1- \alpha)} }\Vert \Phi_h\Vert_{\alpha, K} \, \vert k\vert_K \le \frac{c}{t^{\frac{3}{2} (1- \alpha)} }\Vert \Phi\Vert_{\alpha, K} \, \vert k\vert_K |h|_J. 
$$ 
By taking the supremum over $\{h \in J \, : \,  |h|_J =1\}$ we get 
the first estimate in \eqref{estimates12}.
 
\vskip 1mm  
  To prove the second estimate we fix $t>0,$ $k \in K$, with $|k|_K=1$, $\xi \in U$ with $|\xi|_U=1$  and 
  argue as before.
    We first    introduce the following  linear operators (cf. Lemma \ref{der1}) 
    \begin {equation} \label{ope1}
  \nabla_k \nabla_{\xi}^G P_t : C^1_K(H) \to   {C}_b (H),\;\;\;  \nabla_k \nabla_{\xi}^G  P_t : {C}_b(H) \to  {C}_b (H).  
 \end{equation}  
When $\phi \in C^1_K(H)$ we know that
  \begin{equation} \label{ciser}  
  \nabla_k \nabla_{\xi}^G   P_t[\phi](x)
=\int_H
   \<\Gamma_t G \xi,Q_t^{-\frac{1}{2}}y\>  \nabla_{e^{tA}k} \phi(e^{tA}x+y)  \mu_t(dy). 
\end{equation}     
 Moreover, we have, with $\Gamma_t =  Q^{-1/2}_t e^{tA}$,  
  \begin{gather}    \nonumber     
  \sup_{x \in H}   | \nabla_k  \nabla^G_{\xi } P_t[\phi](x)|^2  \le  \int_H
|\< \Gamma_t   G \xi  ,Q_t^{-\frac{1}{2}}y\> |^2 \, | \nabla_{e^{tA}k} \phi(e^{tA}x+y)|^2    \,   \mu_t(dy)
\\ \label{css} 
\le   \Vert \phi \Vert_{C^1_K}^2 \,    \vert Q^{-1/2}_t e^{tA} G  \xi  \vert_{H}^2 \le  \frac{c}{t} |\xi|_U^2       \Vert \phi \Vert_{C^1_K}^2
  = \frac{c}{t}       \Vert \phi \Vert_{C^1_K}^2.
 \end{gather}        
 On the other hand if $\phi \in C_b(H)$ then   
 \begin{gather*}   
  \nabla_k \nabla_{\xi}^G  P_t[\phi](x)
=\int_H
\big (\< \Gamma_t k,Q_t^{-\frac{1}{2}}y\> \, \< \Gamma_t G \xi ,Q_t^{-\frac{1}{2}}y\> -  \< \Gamma_t k,  \Gamma_t G \xi \> \big)
\phi(e^{tA}x+y)\mu_t(dy),     
\end{gather*} 
  \begin{equation}\label{s7}
  \sup_{x \in H}  | \nabla_k  \nabla^G_{\xi } P_t[\phi](x)|^2 \le C \Vert \phi \Vert_{\infty}^2 
 \,   | \Gamma_t k|_H^2 \,  \vert Q^{-1/2}_t e^{tA} G  \xi  \vert_{H}^2 
 \le  
    c \Vert \phi \Vert_{\infty}^2
 \;  \frac{1}{t^{3}} |k|_K^2 \,   \frac{1}{t} |\xi|_U^2 \le c \Vert \phi \Vert_{\infty}^2
 \;  \frac{1}{t^{4}}. 
 \end{equation}   
   Interpolating between \eqref{css} and \eqref{s7}  as we have done before  we get 
    \begin{equation} \label{es23}
 \sup_{x \in H}  \vert  \nabla_k (\nabla_{\xi}^G P_t[\psi])(x) \vert_{}
\leq \frac{c}{t^{\frac{4 - 3\alpha}{2}}} \Vert \psi\Vert_{{\alpha, K  }},\;\; \psi \in C^{\alpha}_K(H).  
\end{equation}   
Now for  $x \in H$, $k \in K$,  $\Phi \in C^{\alpha}_K (H, J)$, we compute  
\begin{gather*} 
 \Vert  \nabla_k (\nabla_{}^G R_t[\Phi])(x) \Vert_{L(U,J)}^2
= 
\sup_{a \in U, \, |a|_U =1} |\nabla_k (\nabla_{a}^G R_t[\Phi])(x)|^2_J
\\ 
=  \sup_{|a|_U =1}  \, \, \, \sup_{h \in J,  \; |h|_J =1}| \langle 
\nabla_k (\nabla_{a}^G R_t[\Phi])(x), h \rangle_J  |^2
=  \sup_{|a|_U =1}   \sup_{|h|_J=1}|   
\nabla_k (\nabla_{a}^G P_t[\Phi_h])(x) |^2
\\
\le  \frac{c\,  \vert k\vert^2_K}{t^{{4 - 3\alpha}}}  
  \; \sup_{|h|_J=1} \Vert \Phi_h\Vert_{{\alpha, K  }}^2  \; 
\le \;   \frac{c\,  \vert k\vert^2_K  }{t^{{4 - 3\alpha}}} \, \Vert \Phi\Vert_{{\alpha, K}}^2.     
\end{gather*}    
 The second estimate in \eqref{estimates12} follows easily.
 \end{proof}

The next result is  crucial  for our approach to get pathwise uniqueness (see in particular Theorem \ref{forse111}   and   the proof of Theorem \ref{uni1}).   
 We can only prove the result when $\Phi \in  C^{\alpha}_b (H, J)$
using Lemma \ref{intervector}. 
 We do not know if such result holds more generally when $\Phi \in  C^{\alpha}_K (H, J)$.  

We fix an basis $(f_m)$ of $J$ and set $\Phi_m = \langle \Phi , f_m \rangle_J $. We will use the OU semigroup $(P_t)$ given  in \eqref{ptt1}.
   \begin{lemma}\label{der12ab} 
    Let $T>0.$  Under the assumptions of Lemma \ref{teo:derHsemigroup}   
 let  $\Phi \in C^{\alpha}_b (H, J)$,  $\alpha \in (0,1)$. We have, for $t  \in (0,T]$, $k \in K$,  
    \begin{equation}\label{as22}
 \Big ( \sum_{m \ge 1} \sup_{|a|_U =1}| \nabla_k \nabla_{a}^G  P_t[\Phi_m](x)|^2  \Big)^{1/2}\le
    \frac{C}{t^{\frac{4 - 3\alpha}{2}}}  |k|_K  \, \| \Phi\|_{ C^{\alpha}_b(H, J)}.     
\end{equation}
where $C >0$ is independent of $x,k, t$, $\Phi$ and the basis $(f_m)$ of $J$ . 
 \end{lemma}   
  \begin{proof} 
 We recall  that for $k \in K$, $a \in U$
    $$  \nabla_k \nabla_{a}^G  P_t[\Phi_m](x) =   
  \int_H
 \big (\< \Gamma_t k,Q_t^{-\frac{1}{2}}y\> \, \< \Gamma_t G a ,Q_t^{-\frac{1}{2}}y\> -  \< \Gamma_t k,  \Gamma_t G a \> \big) \,   \Phi_m (e^{tA}x   +y) \,  \mu_t(dy).
$$
 We fix $k \in K$.  We have  by the Cauchy-Schwarz inequality (cf. \eqref{w2}) for any $x \in H,$ $m \ge 1,$  
\begin{gather*}
 \sup_{|a|_U =1}| \nabla_k \nabla_{a}^G  P_t[\Phi_m](x)|^2 
 \\
 \le 
  \sup_{|a|_U =1}    \int_H
 \big | \big (\< \Gamma_t k,Q_t^{-\frac{1}{2}}y\> \, \< \Gamma_t G a ,Q_t^{-\frac{1}{2}}y\> -  \< \Gamma_t k,  \Gamma_t G a \> \big) \big |^2  \, \mu_t(dy)\, \cdot 
 \, \int_H |\Phi_m (e^{tA}x   +y)|^2 \,  \mu_t(dy)
 \\
 \le \frac{c}{t^4} |k|_K^2    \, \cdot 
 \, \int_H |\Phi_m (e^{tA}x   +y)|^2 \,  \mu_t(dy).
\end{gather*} 
   Hence  
  $$
  \sum_{m \ge 1}   \sup_{a \in U } \Big | \int_H
 \big (\< \Gamma_t k,Q_t^{-\frac{1}{2}}y\> \, \< \Gamma_t G a ,Q_t^{-\frac{1}{2}}y\> -  \< \Gamma_t k,  \Gamma_t G a \> \big) \,   \Phi_m (e^{tA}x   +y) \,  \mu_t(dy) \Big|^2
 $$   
$$
\le \frac{c}{t^4} |k|_K^2   \sum_{m \ge 1}  \int_H |\Phi_m (e^{tA}x   +y)|^2 \,  \mu_t(dy)  =  \frac{c}{t^4} |k|_K^2   \int_H |\Phi (e^{tA}x   +y)|^2_J \,  \mu_t(dy). 
$$
 It follows that 
\begin{equation}\label{se1}
 \sum_{m \ge 1} \sup_{|a|_U =1}| \nabla_k \nabla_{a}^G  P_t[\Phi_m](x)|^2 \le 
  \frac{c}{t^4} |k|_K^2 \, \| \Phi\|_{\infty}^2. 
\end{equation} 
 Now we fix also $x \in H$. For any $l \in J$ we define $\Phi_l : H \to \R$,
 \begin{gather*} 
 \Phi_l(y) = \langle \Phi(y) , l\rangle_J,\;\;\; l \in J. 
\end{gather*}
 We have  $\Phi_m =   \Phi_ {f_m}  $. 
  We can consider the linear operator  $T_{x,k}:   C_b^{} (H, J) \to  
 L_2 (J, U)$, 
 $$
  T_{x,k}(\Phi) (l) =  \nabla_k \nabla_{\cdot }^G  P_t[   \Phi_l ](x) \in U, \;\;\; \Phi \in C_b^{} (H, J),\;\; l \in J
  $$ 
 (we identify $U$ with $L(U, \R)$).  We have  
  $$  
  | T_{x,k}(\Phi) (l)|^2_U
  = \sup_{|a|_U =1} \Big |  \int_H
 \big (\< \Gamma_t k,Q_t^{-\frac{1}{2}}y\> \, \< \Gamma_t G a ,Q_t^{-\frac{1}{2}}y\> -  \< \Gamma_t k,  \Gamma_t G a \> \big) \,  \langle  \Phi (e^{tA}x   +y) , l \rangle_J  \,\,   \mu_t(dy) \Big|^2.  
 $$ 
 and 
 \begin{gather*}
 \| T_{x,k}(\Phi) \|_{ L_2 (J, U)}^2 =  \sum_{m \ge 1} | T_{x,k}(\Phi) (f_m)|^2_U
 =  \sum_{m \ge 1} \sup_{|a|_U=1} | \nabla_k \nabla_{a}^G  P_t[   \Phi_m ](x)|^2.
\end{gather*}
 By the bound  in \eqref{se1} we deduce that $\| T_{x,k}(\Phi) \|_{ L_2 (J, U)}^2 \le  \frac{c}{t^4} |k|_K^2 \, \| \Phi\|_{\infty}^2 $, where $c$ is independent of $t>0$, $x \in H$, $\Phi$  and $k \in K.$     
  
 The linear operator: $T_{x,k}:   C_b^{} (H, J) \to  
 L_2 (J, U)$ is well defined and continuous;  we have 
 \begin{equation}\label{prima}
 \| T_{x,k} \|_{L (  C_b^{} (H, J),L_2 (J, U)  )} \le   \frac{c}{t^2} |k|_K. 
\end{equation} 
 Now if $\Phi \in C^1_{K}(H, J)$ we find 
 for $k \in K$, $a \in U$ (cf. \eqref{ciser})
  $$ 
  \nabla_k \nabla_{a}^G   P_t[\Phi_m ](x) 
= \int_H 
   \<\Gamma_t G a,Q_t^{-\frac{1}{2}}y\> \,  \langle   \nabla^{K}\Phi_m(e^{tA}x+y),
   {e^{tA}k} \rangle_K \;  \mu_t(dy).  
$$
  We find  (cf. \eqref{s})
\begin{gather*}    
 \sup_{a \in U, |a|_U =1}| \nabla_k \nabla_{a}^G  P_t[\Phi_m](x)|^2 
 \\
 \le 
  \sup_{|a|_U =1}    \int_H
 \big | \< \Gamma_t G a ,Q_t^{-\frac{1}{2}}y\> \big| ^2  \, \mu_t(dy)\, \cdot 
 \, \int_H |  \langle   \nabla^{K}\Phi_m(e^{tA}x+y),
   {e^{tA}k} \rangle_K |^2 \,  \mu_t(dy)  
 \\
 \le \frac{c}{t}     \, \cdot 
 \, \int_H |  \langle   \nabla^{K}\Phi_m(e^{tA}x+y),
   {e^{tA}k} \rangle_K   |^2 \,  \mu_t(dy)
\end{gather*} 
and  so 
  $$
  \sum_{m \ge 1}   \sup_{ |a|_U =1 } | \nabla_k \nabla_{a}^G  P_t[\Phi_m](x)|^2 
 $$     
$$ 
\le \frac{c}{t}   \sum_{m \ge 1}  
\int_H | \langle   \nabla^{K}\Phi_m(e^{tA}x+y),  
   {e^{tA}k} \rangle_K   |^2 \,  \mu_t(dy) 
=  \frac{c}{t}    \int_H |  \nabla^{K} \Phi (e^{tA}x   +y) [e^{tA}k] |^2_J \,  \mu_t(dy). 
$$
  Since $|  \nabla^{K} \Phi (e^{tA}x   +y) [e^{tA}k] |_J  \le 
   \|  \nabla^{K} \Phi (e^{tA}x   +y)\|_{L(K, J)} \,  |k |_K
  $, $t \ge 0,$  it follows that 
\begin{equation}\label{se15}
 \sum_{m \ge 1} \sup_{|a|_U =1}| \nabla_k \nabla_{a}^G  P_t[\Phi_m](x)|^2 \le 
  \frac{c}{t} |k|_K^2 \, \| \Phi\|_{C^1_K(H, J)}^2. 
\end{equation} 
Hence 
 \begin{equation}\label{prima1}
 \| T_{x,k} \|_{L (  C_K^{1} (H, J),L_2 (J, U)  )} \le   \frac{c}{t^{1/2}} |k|_K.  
\end{equation}   
  Interpolating,  between \eqref{prima} and \eqref{prima1} as in the proof of Lemma \ref{der12}, for any $\Phi \in \big(C_b(H,J) , C^1_K(H,J) \big)_{\alpha , \infty}$, we get   
\begin{equation}\label{as2}
\Big ( \sum_{m \ge 1} \sup_{|a|_U =1}| \nabla_k \nabla_{a}^G  P_t[\Phi_m](x)|^2  \Big)^{1/2} \, \le
    \frac{c}{t^{\frac{4 - 3\alpha}{2}}}  |k|_K   \, \| \Phi\|_{ \big(C_b(H,J) , C^1_K(H,J) \big)_{\alpha , \infty}} . 
\end{equation}   
Now we consider any  $\Phi \in  C^{\alpha}_b(H, J) \subset  \big(C_b(H,J) , C^1_K(H,J) \big)_{\alpha , \infty} $
  (see   Lemma \ref{intervector}). We finally  obtain  
  \begin{equation*}
 \Big (\sum_{m \ge 1} \sup_{|a|_U =1}| \nabla_k \nabla_{a}^G  P_t[\Phi_m](x)|^2 \Big)^{1/2} \, \le
    \frac{C}{t^{\frac{4 - 3\alpha}{2}}}  |k|_K \, \| \Phi\|_{ C^{\alpha}_b(H, J),}\;\; t \in (0,T] . 
\end{equation*}  
\end{proof} 
   
\begin{remark} \label{hilsch}{\em We provide here  an  equivalent formulation of 
Lemma \ref{der12ab}    when $J=K$ (this   will be used   in the proof of Theorem \ref{forse111}).
 Let $\Phi \in C^{\alpha}_b (H, K)$,  $\alpha \in (0,1)$. 
Let us fix $t  \in (0,T]$, $x \in H$ and  $k \in K$.  

\vskip 1mm 
Recall the notation  $\Phi_l $ $
 = \langle \Phi , l \rangle_K $, $l \in K.$
 We know that 
 the linear operator $\nabla_k \nabla_{}^G  R_t[ \Phi](x)$ is well defined 
 from $K$ into $U$  by the formula
\begin{gather*}
 l \mapsto \nabla_k \nabla_{}^G  P_t[ \Phi_l](x)
\end{gather*}
(cf. Lemma \ref{der1};  note that the operator $a \mapsto  \nabla_k \nabla_{a}^G  P_t[ \Phi_l](x)$ belongs to 
$L (U , \R)$ and so it can be identified with an element of $U$). Assertion 
of Lemma \ref{der12ab} is equivalent to say that 
\begin{gather}\label{l222} 
 \nabla_k \nabla_{}^G  R_t[ \Phi](x) \in  L_2(K,U) \;\; \text{and}
\\ \nonumber  
\|  \nabla_k \nabla_{}^G  R_t[ \Phi](x) \|_{L_2(K,U)} 
=  \Big (\sum_{m \ge 1} \sup_{|a|_U =1}| \nabla_k \nabla_{a}^G  P_t[\langle \Phi , f_m \rangle_K ](x)|^2 \Big)^{1/2} 
\\ \nonumber \le    \frac{C}{t^{\frac{4 - 3\alpha}{2}}}  |k|_K \, \| \Phi\|_{ C^{\alpha}_b(H, K),}\;\; t \in (0,T];
\end{gather}
here $(f_m)$ is a basis in $K$. Recall that  
 the constant  $C >0$ is independent of $x,k, t$ and $\Phi$. 
   } 
\end{remark}

 We will also use   the following  additional regularity result. 
   \begin{lemma}\label{der12a} Let $T>0.$  Under the assumptions of Lemma \ref{teo:derHsemigroup}   
 let  $\Phi \in C^{\alpha}_b (H, J)$,  $\alpha \in (0,1)$. We have  the following estimate, for $t \in (0,T]$, $x, y \in H$
\begin{align}\label{est12}    
   \Vert  \nabla_{}^G R_t[\Phi](x) -  \nabla_{}^G R_t[\Phi](x')  \Vert_{L(U,J)} 
\leq \frac{c}{t^{\frac{1}{2}}} |x-x'|^{\alpha}  \Vert \Phi\Vert_{\alpha} 
    \end{align}   
\end{lemma} 
\begin{proof} The assertion follows easily by the formula
\begin{gather*}
\nabla^G_{a} R_t[\Phi](x) -  \nabla^G_{a} R_t[\Phi](x') =\int_H 
\<Q_t^{-\frac{1}{2}}e^{tA}Ga,Q_t^{-\frac{1}{2}}y\> \, [\Phi(e^{tA}x+y) - 
 [\Phi(e^{tA}x'+y)]
\caln(0,Q_t)(dy),
\end{gather*} 
 $ a\in U $, using that 
 \begin{gather*}
 |\Phi(e^{tA}x+y) - 
 \Phi(e^{tA}x'+y)|_J \le \Vert \Phi\Vert_{\alpha}  |x-x'|^{\alpha},\;\; t \ge 0, \, y \in H.
\end{gather*}
 \end{proof}   
 In the sequel we will apply  the previous regularity results when 
\begin{equation} \label{jjj}
J = K = V \times U. 
\end{equation} 
 Let $T>0.$ 
  We consider  the following integral equation 
which will be important in  the sequel:
\begin{align} \label{bsde2uno}
u(t,x)
=&\int_t^T R_{s-t}\left[e^{-(s-t) {A}}G B(s,\cdot)\right](x)\,ds+
\int_t^T R_{s-t}\left[e^{-(s-t){A}} \nabla^Gu(s,\cdot) B(s,\cdot) \right](x)\,ds, 
\end{align}
where $u(t,x)$ takes values in $K$ and  $ \nabla^G u(s,x)B(s,x) =  \nabla_{GB(s,x)}u(s,x) \in K $, $(s,x) \in [0,T] \times H$.  

Using the previous lemmas, we will solve the equation in the Banach space $E_0 $   
consisting of all $u \in B_b([0,T]\times H, K)$
such that $u(t, \cdot) $ is   
  { $K$-differentiable  on $H$, with
    $\nabla^{K} u \in B_b([0,T]\times H, L(K,K))$.   } 
     
 We also require that  there exists $C =C_T >0$  such that for any $x, y \in H$ 
 \begin{equation}\label{ass}
\sup_{s \in [0,T]}|\nabla^G u(s, x) - \nabla^G u(s, y)|_{L(U,K)}  \,  \le  \, C |x-y|_H^{\alpha},  
\end{equation}
 where $\alpha \in (2/3, 1)$ is given in Hypothesis \ref{ip-b1}.   
         Finally, to define $E_0$ we require that,             
 for each $\xi \in U$, $t \in [0,T]$, the mapping:  
\begin{equation} \label{ci111}
 x \mapsto \nabla_{\xi}^G u (t,x) \;\; \text{is $K$-differentiable on $H$     
 }    
 \end{equation}    
 $$ \text{
 with  $ \;\; \sup_{(t, x) \in [0,T] \times H } \; \sup_{|\xi|_U =1}  \| \nabla^K 
 \nabla^G_{\xi} u(t,x) \Vert_{L\left(K, K \right)} < \infty.$ }
$$        
Let $\beta \ge 0$ to be  fixed later. It is not difficult to prove that $E_0$ is a Banach space endowed with the norm
$$   
\| u\|_{E_0, \beta} =  
 \sup_{(t, x) \in [0,T] \times H} e^{\beta t} | u (t,x) |_K
+   \sup_{(t, x) \in [0,T] \times H} e^{\beta t} \| \nabla^K u (t,x) \|_{L(K,K)} 
$$
\begin {equation} \label{bett}
 + \sup_{t \in [0,T]} e^{\beta t} \| \nabla^G u (t, \cdot ) \|_{C^{\alpha}_b(H, L(U,K))}     + \sup_{(t, x ) \in [0,T] \times H} \, \sup_{|\xi|_U =1} e^{\beta t}  \Vert \nabla^K \nabla^G_{\xi} u(t,x) \Vert_{L\left(K, K \right)}.
\end{equation}

\begin{lemma} \label{forse13} Let Hypotheses
\ref{base} and  \ref{ip-b1} hold true. 
There exists a unique solution $u \in E_0$  to \eqref{bsde2uno}.  Moreover,  there exists  a function $h(r) = h(r,\alpha) > 0$, $r\ge 0$, such that 
$h(r) \to 0$ as $r \to 0^+$ and  
 if $S \in [0,T]$ verifies 
 $h(T- S) \cdot ( \sup_{t \in [0,T]}\| B(t,\cdot) \|_{\alpha  }) $ $  \le 1/4$,  then 
\begin{equation} \label{de22}
 \sup_{t\in [S,T],\,x\in H} \|\nabla^K    u  (t,x) \|_{L(K,K)} \leq 1/3. 
\end{equation} 
\end{lemma}        

\begin{proof}  We  introduce the following operator ${\cal T}$ defined on $E_0$:
$$
{\cal T} u (t,x) = \int_t^T R_{s-t}\left[e^{-(s-t){A}}G B(s,\cdot)\right](x)\,ds+
\int_t^T R_{s-t}\left[e^{-(s-t){A}} \nabla^G u(s,\cdot) B(s,\cdot)\right](x)\,ds, 
$$
 $u \in E_0,$ $(t,x) \in [0,T] \times H$. In order to apply Lemmas \ref{der12} and \ref{der12a} with $J=K$     we first check the H\"older regularity
  of the terms $G B(s, \cdot)$ and $ \nabla^G u(s,\cdot) B(s,\cdot).$
 
  Since, for any $x \in H$, $h \in H,$  $s \in [0,T]$,
  \begin{gather*}
 |G B(s, x) - G B(s, x+h ) |_K  = | B(s, x) - B(s, x+ h) |_U \le C |h|_H^{\alpha}   
\end{gather*}  
 we get  that $G B \in B_b([0,T]; C^{\alpha}_b(H,K))$, see the definition  in  Hypothesis \eqref{ip-b1}.  
  We need to prove that also 
\begin{equation} \label{sww}  
 \nabla^G u(s,\cdot)B(s,\cdot) 
 \;\; \text{ belongs to  $B_b([0,T]; C^{\alpha}_b(H,K)).$  } 
\end{equation} 
 We write  for $x, y \in H$ with  $|x-y|_H \le 1$
\begin{gather*}
[\nabla^G u(s, x)B(s,x) - \nabla^G u(s, y)B(s,x)] +  \nabla^G u(s, y)[B(s,x)   - B(s,y)].      
\end{gather*} 
 We  bound the second term with    
\begin{gather*}
  | \nabla^G u(s, y)[B(s,x)   - B(s,y)] |_K  \le   \sup_{(t, x) \in [0,T] \times H} e^{\beta t} \| \nabla^G u (t,x) \|_{L(U,K)} \,  | B(s,x)   - B(s,y)|_U
   \le  C |x-y|_H^{\alpha}.
\end{gather*}  
Moreover  we have   
 \begin{gather*}
|[\nabla^G u(s, x) - \nabla^G u(s, y)]  B(s,x)|_K \le |B(s,x)|_U
 |\nabla^G u(s, x) - \nabla^G u(s, y)|_{L(K,U)} 
 \\
 \le \| B \|_{\infty} \,\sup_{t \in [0,T]} e^{\beta t} \| \nabla^G u (t, \cdot ) \|_{C^{\alpha}_b(H, L(U,K))}    \, |x-y|_H^{\alpha}. 
\end{gather*}
  By   Lemmas \ref{der12} and \ref{der12a},  using  that $\alpha >2/3$, it is not difficult to prove that ${\cal T}: E_0 \to E_0$. Let us check that for a suitable value of $\beta $  the map ${\cal  T}$ is a strict contraction (see \eqref{bett}).
We  have to  consider $\| {\cal T}u_1 - {\cal T}u_2\|_{E_0, \beta}$, $u_1, u_2 \in E_0$; we   only treat
 the  term   
$$
\sup_{t,x} \sup_{|\xi|_U =1} e^{\beta t} \| \nabla^K  \nabla_{\xi}^G [{\cal T} u_1 - {\cal T} u_2] (t,x)  
 \|_{L(K,K)}.
$$
Indeed the other terms of $\| {\cal T}u_1 - {\cal T}u_2\|_{E_0, \beta}$ can be estimated in a similar way.  
We have  for any $k \in K$ with $|k|_K =1$ 
 $$
e^{\beta t} | \nabla_k \nabla_{\xi}^G [{\cal T} u_1(t,x) -
 {\cal T} u_2(t,x)] \, |_{K}   
  $$
  $$  
\le \int_t^T e^{-\beta (s-t)} \Big |  \nabla_k \nabla^G_{\xi}  R_{s-t}\left[e^{-(s-t){A}}\, e^{\beta s} 
\{ \nabla^G u_1(s,\cdot) - \nabla^G u_2(s,\cdot) \}  B(s,\cdot) \right](x) \,\Big |_{K}ds 
$$       
$$
\le \int_t^T \frac{c  e^{- \beta  (s-t)}}{(s-t)^{\frac{4 - 3\alpha}{2}}} ds \sup_{t \in [0,T]}\| B(t,\cdot) \|_{\alpha }  \| u_1 - u_2 \|_{E_0, \beta}
\le     
 C_{\beta,T}     \sup_{t \in [0,T]}\| B(t,\cdot) \|_{\alpha}  \| u_1 - u_2 \|_{E_0, \beta}, 
$$
where $C_{\beta,T} >0$ tends to $0$ as $\beta \to + \infty.$  
 Therefore taking the supremum over $k \in K$ with $|k|_K =1$ we infer
$$
e^{\beta t} \| \nabla^K \nabla_{\xi}^G [{\cal T} u_1(t,x) -
 {\cal T} u_2(t,x)] \|_{L(K,K)}  
\le     
 C_{\beta,T}     \sup_{t \in [0,T]}\| B(t,\cdot) \|_{\alpha}  \| u_1 - u_2 \|_{E_0, \beta}. 
$$
 Choosing $\beta$ large enough,  we can apply the   fixed point theorem
and obtain  that there exists a unique solution $u \in E_0$.
 
  In order to prove \eqref{de22},  we first introduce $\| u\|_{E_0, 0, S,T}$ which is defined as $\| u\|_{E_0, 0}$ in \eqref{bett} (with $\beta =0)$ but taking all the supremums over $[S, T] \times H$ instead of $[0,T] \times H$.
  We proceed as before:  
\begin{gather*}
 \| u\|_{E_0,0, S, T}  \le 
\sup_{t \in [S,T]}\int_t^T \frac{c  }{(s-t)^{\frac{4 - 3\alpha}{2}}} ds \, \sup_{t \in [0,T]}\| B(t,\cdot) \|_{\alpha, K }  \, (\| u_0 \|_{E_0, 0, S,T} + 1)  
\\ \le
h(T-S) \, \sup_{t \in [0,T]}\| B(t,\cdot) \|_{\alpha }  \, (\| u_0 \|_{E_0, 0, S,T} + 1),
\end{gather*}
where $ 
 h( r) = \int_0^r \frac{c  }{s^{\frac{4 - 3\alpha}{2}}} ds
$; now \eqref{de22} follows since we have   $\frac{3}{4}\| u\|_{E_0,0, S,T} \le 1/4.$
\end{proof}
   
 To prove the next result we will apply Lemma \ref{der12ab} (see also Remark \ref{hilsch}). To this purpose we fix  a basis $(f_m)$ of $K$ and set $u_m = \langle u , f_m \rangle_K $, where $u$ is the solution given in Lemma \ref{forse13}.    
   \begin{theorem} \label{forse111} Let Hypotheses
\ref{base} and  \ref{ip-b1} hold true. 
 Then the unique solution $u \in E_0$  to \eqref{bsde2uno} (see Lemma \ref{forse13})  verifies in addition  
    \begin{equation}\label{andiamo}
 \Big ( \sum_{m \ge 1} \sup_{|a|_U =1}| \nabla_k \nabla_{a}^G  u_m (t,x)|^2  \Big)^{1/2}\le
    C |k|_K  \, \sup_{t \in [0,T]}\| B(t, \cdot )\|_{ C^{\alpha}_b(H, U)},\;\; \;\;\; k \in K,     
\end{equation}  
where $C >0$ is independent of $x \in H, k \in K$, $t \in [0,T]$, $u$ and the basis $(f_m)$ in $K$.    
\end{theorem} 
\begin{proof} First following the proof of Lemma \ref{forse13} it is not difficult to  prove that there exists $C_T>0$ independent of $u$ such that 
\begin{equation}\label{sff}
 \sup_{t \in [0,T]}  \| \nabla^G u (t, \cdot ) \|_{C^{\alpha}_b(H, L(U,K))} \le  C_T \sup_{t \in [0,T]}\| B(t, \cdot )\|_{ C^{\alpha}_b(H, U)}.
\end{equation}
  Then we write $u =  v + w$, where  
$$
 v(t,x) = \int_t^T R_{s-t}\left[e^{-(s-t){A}}G B(s,\cdot)\right](x)\,ds,\;\;
  w(t,x)=
\int_t^T R_{s-t}\left[e^{-(s-t){A}} \nabla^G u(s,\cdot) B(s,\cdot)\right](x)\,ds. 
$$
We need to prove that \eqref{andiamo} holds when $u$ is replaced by $v$ and $w$. We  concentrate on  $w$ (the proof  for  $v$ is similar).
 We define, for any $s \in [0,T]$, $x\in H,$ $\Phi (s, x) = e^{-(s-t){A}} \nabla^G u(s ,x) B(s,x)$.
 
 Note that $\Phi \in  B_b([0,T]; C^{\alpha}_b(H,K)) $ (see the computations to verify \eqref{sww}). 
   By \eqref{sff} we obtain  
\begin{gather*}  
\sup_{s \in [0,T]}  \| \Phi(s, \cdot )\|_{ C^{\alpha}_b(H, K)} \le C_T \sup_{t \in [0,T]}\| B(t, \cdot )\|_{ C^{\alpha}_b(H, U)} \;\; 
\text{with}
\\ 
 w(t,x)= 
\int_t^T R_{s-t}[ \Phi(s,\cdot)](x)\,ds,\;\;\; t \in [0,T],\; x \in H. 
\end{gather*}  
Let us fix $k \in K$ and $t \in [0,T)$;  set  $w_m = \langle w , f_m \rangle_K $. 
By Lemma \ref{der12ab} (see also Remark \ref{hilsch}) 
 we know that, for any $s \in (t,T]$, the linear operator $\nabla_k \nabla_{}^G  R_{s-t}[ \Phi(s, \cdot)](x)$
  is well defined 
 from $K$ into $U$  by the formula
 $
 l \mapsto \nabla_k \nabla_{}^G  P_{s-t}[ \langle \Phi(s, \cdot ), l \rangle_K ](x) 
$
 and it is a Hilbert-Schmidt operator.  Moreover,
 \begin{gather}\label{aqe}
\|  \nabla_k \nabla_{}^G  R_{s-t}[ \Phi(s, \cdot)](x) \|_{L_2(K,U)} 
=  \Big (\sum_{m \ge 1} \sup_{|a|_U =1}| \nabla_k \nabla_{a}^G  P_{s-t}[\langle \Phi (s, \cdot), f_m \rangle_K ](x)|^2 \Big)^{1/2} 
\\ \nonumber 
 \le    \frac{C}{(s -t)^{\frac{4 - 3\alpha}{2}}}  |k|_K \, \sup_{r \in [0,T]}\| \Phi(r, \cdot)\|_{ C^{\alpha}_b(H, K)}
  \le   \frac{c  |k|_K }{(s -t)^{\frac{4 - 3\alpha}{2}}}  \sup_{t \in [0,T]}\| B(t, \cdot )\|_{ C^{\alpha}_b(H, U)}, 
 \;\; s > t. 
\end{gather} 
Note that also the linear operator $\nabla_k \nabla_{}^G w(t, x)$ is well defined from $K  $ into $U$ ($l \mapsto \nabla_k \nabla_{}^G   [\langle w(s, \cdot ), l \rangle_K ](x) $).

Moreover, we have, for any $x \in H,$
\begin{gather*}
 \nabla_k \nabla_{}^G w(t, x) = 
 \int_t^T  \nabla_k \nabla_{}^G  R_{s-t}[ \Phi(s, \cdot)](x) ds. 
\end{gather*}
Using \eqref{aqe} we deduce that $\nabla_k \nabla_{}^G w(t, x) \in L_2(K,U)$ and 
\begin{gather*}
   \|  \nabla_k \nabla_{}^G w(t, x)\|_{L_2(K,U)}
\le 
 \int_t^T  \|  \nabla_k \nabla_{}^G  R_{s-t}[ \Phi(s, \cdot)](x) \|_{L_2(K,U)} ds \le C_{\alpha, T} |k|_K \,   \sup_{t \in [0,T]}\| B(t, \cdot )\|_{ C^{\alpha}_b(H, U)},
\end{gather*}
where $ C_{\alpha, T} >0 $ is independent of $x \in H$, $t \in [0,T]$, $B$ and $k \in K$. This completes the proof.  
 \end{proof}
  
\begin{remark}  {\em   If try to prove the previous theorem using 
  Lemma \ref{interpola} instead of Lemma \ref{intervector} we could eventually obtain a bound for 
  $\sum_{m \ge 1} \sup_{|a|_U =1}| \nabla_k \nabla_{a}^G  u_m (t,x)|^2 $ 
  by requiring that 
  $$
 \sum_{j \ge 1}  \sup_{t \in [0,T]}  \| \langle B(t, \cdot), e_j \rangle_U \|_{C^{\alpha}_K(H)}^2 < \infty.
  $$
  However this condition is restrictive when it is applied to the examples of Section 3. 
    }
\end{remark}

\section{The related infinite dimensional forward-backward system }\label{Sez-FBSDE}

In a complete probability space $(\Omega, \calf, \P)$, let us consider the following forward-backward system (FBSDE) 
 \begin{equation}
  \left\lbrace\begin{array} 
 [c]{l}%
 d\Xi_\tau^{t,x}  =A\Xi_\tau^{t,x} d\tau+Gd W_\tau
 ,\text{ \ \ \ }\tau\in\left[  t,T\right], \\ 
 \dis
 \Xi_{t}^{t,x} =x,\\ 
 \dis
  -d Y_{\tau}^{t,x}=-{A}Y_\tau^{t,x}+GB(\tau,\Xi^{t,x}_\tau)\;d\tau
  + Z^{t,x}_{\tau} \, B(\tau,\Xi^{t,x}_\tau)  d\tau
  - Z^{t,x}_{\tau}\;dW_\tau,
  \qquad \tau\in [0,T],
   \\\dis
 Y_{T}^{t,x}=0,
 \end{array}
 \right.  \label{fbsde2-wave-notilde}%
 \end{equation}
 where $t \in [0,T]$, $x \in H$, and the forward equation is the abstract formulation of the wave equation (\ref{wave1}) given in (\ref{wave eq ab det}) under Hypothesis \ref{base}; here $B: [0,T] \times H \to U$ is (Borel) measurable  and satisfies
\begin{equation}
\label{bbb}
B(t, \cdot) \in C_b(H,U),\;\; t \in [0,T],\;\; \| B\|_{\infty} = \sup_{t \in [0,T] \times H} |B(t,x)|_U < \infty
\end{equation} 
(clearly, Hypothesis \ref{ip-b1} implies \eqref{bbb}); 
$G $ is defined by (\ref{G}) and $W$ is a cylindrical Wiener process in $U$.

 We extend $\Xi^{t,x}$ to the whole $[0,T]$ by setting $\Xi_\tau^{t,x}=x$ for $0\leq\tau\leq t$, in order to have $(Y^{t,x},Z^{t,x})$ well defined on  $[0,T]$.
The precise meaning of the BSDE in (\ref{fbsde2-wave-notilde}) is given by its mild formulation: for $ \tau\in [0,T]$
 \begin{equation} \label{bsde2-markovian-wave-notilde-mild} Y_{\tau}^{t,x}=\int_\tau^Te^{-(s-\tau){A}}G B(s,\Xi^{t,x}_s)\,ds+\int_\tau^Te^{-(s-\tau){A}} Z_s^{t,x}\, B(s,\Xi^{t,x}_s)\,ds- \int_\tau^Te^{-(s-\tau){A}} Z^{t,x}_{s}\;d W_s,
\end{equation}
 $\P$-a.s.
 (cf. \cite{HuPeng}, \cite{Gua1} and the references therein).
 The solution of \eqref{bsde2-markovian-wave-notilde-mild} will be a pair of processes $(Y^{t,x}, Z^{t,x})$ (see Proposition \ref{Teo:ex-regdep-BSDE}).
 Notice that in order to give sense to the BSDE in (\ref{fbsde2-wave-notilde})
as it is done in (\ref{bsde2-markovian-wave-notilde-mild}), we need that $A$ is the generator of a $C_0$-group of bounded linear operators, so that $-A$
is the generator of a $C_0$-semigroup of bounded linear operators. 

\noindent We also refer to this BSDE as BSDE in a Markovian framework, since the pair of processes $(Y^{t,x},Z^{t,x})$ depends on the Markov process $\Xi^{t,x}$.
  We endow $(\Omega, \calf, \P)$ with the natural filtration $({\cal F}_t^W)$ of $W$, augmented in the usual way with the family of $\P$-null sets of $\calf$. All the concepts of measurability, e.g. predictability, are referred to this filtration.
  
      Notice that  the forward equation evolve in $H$ as we have already discussed in Section \ref{sec-prel}, while the backward equation
takes values in $K$, indeed the term $GB(\cdot, \cdot)$ take values in $K$, and we can give an explicit representation of the solution in terms of $GB(\cdot)$, see \cite{HuPeng}, Sections 2 and 3.

  We denote by $L^2_\calp(\Omega, C([0,T],{K}))$ the  space of all predictable ${K}$-valued processes $Y$ with continuous paths and such that
$$
 \E [ \sup_{\tau\in[0,T]}\vert Y^{}_\tau\vert^2] = \| Y\|_{L^2_\calp(\Omega, C([0,T],{K}))}^2 < \infty.
$$
The space $L^2_\calp(\Omega, C([0,T],{K}))$ is a Banach space endowed with the norm $\| \cdot \|_{ L^2_\calp(\Omega, C([0,T],{K}))} $. On the other hand,    
 $L^2_\calp(\Omega\times[0,T],L_2(U,{K}))$  is the usual $L^2$-space of all predictable processes $Z$ with values in $L_2(U,{K})$.
 We also define the space $\calg^{0,1}([0,T]\times {H},{K})$, as in \cite{fute} Section 2.2, as the subspace of
 $C([0,T]\times {H,K})$ consisting of all functions $f$
which are G\^ateaux differentiable with respect to $x$ and such that the map $\nabla_xf:[0,T]\times H\to L({H},{K})$ is strongly continuous. Similarly one can define the space  $\calg^{0,1}([0,T]\times {H},\R) \subset C([0,T]\times {H},\R).$ 

Following \cite{HuPeng},  it is immediate to get existence and pathwise uniqueness of a solution  $(Y^{t,x},Z^{t,x})$ to the Markovian BSDE (\ref{bsde2-markovian-wave-notilde-mild}).
  Moreover we can show regular dependence on the initial datum $x$ of the solution to the forward equation in (\ref{fbsde2-wave-notilde}). 
\begin{proposition}\label{Teo:ex-regdep-BSDE}
Assume Hypothesis \ref{base} and let  $B$ be as in \eqref{bbb}. Let $t \in [0,T]$ and $x \in H.$ 
Consider the ${K}$-valued BSDE
(\ref{bsde2-markovian-wave-notilde-mild}). 
  Then
  there exists a unique solution 
$(Y^{t,x},Z^{t,x})\in L^2_\calp(\Omega, C([0,T],{K})) \times L^2_\calp(\Omega\times[0,T],L_2(U,{K}))$.
 Moreover,  
 the following estimates hold true: 
 \begin{equation}\label{est-YZ}  
  \E\big [ \sup_{\tau\in[0,T]}\vert Y^{t,x}_\tau\vert_{{K}}^2 \big ]+\E\int_0^T\Vert Z_\tau^{t,x}\Vert^2_{L_2(U,{K})} d \tau 
  \leq C_T \| B\|_{\infty}.
   \end{equation}
In addition, the map:  $(t,x)\mapsto Y_t^{t,x}  $,    $[0,T]\times  H \to {K}$, is deterministic.
If we further assume that
\begin{equation}
\label{diff} 
\text{the map:} \;\;  x\mapsto B(\tau,x),\quad H\rightarrow U, \;\; \text{is G\^ateaux differentiable on $H$,  for  all $\tau\in[0,T]$,}
\end{equation}
then, for any $t \in [0,T],$
the map 
\begin{equation}\label{map-x-toYZ}
 x\mapsto (Y^{t,x},Z^{t,x}),\quad H\rightarrow  L^2_\calp(\Omega, C([0,T],{K}))\times 
 L^2_\calp(\Omega\times[0,T],L_2(U,{K}))
\end{equation} 
is G\^ateaux differentiable on $H$. 
Moreover, assuming \eqref{diff}, the map:  $(t,x)\mapsto Y_t^{t,x}$  belongs to 
$\calg^{0,1}([0,T]\times H,{K})$.
 \end{proposition}
\begin{proof} Existence and uniqueness of a solution come directly from Lemma 2.1 and Proposition 2.1
in \cite{HuPeng}, that we can apply since
$B$ is bounded. Estimate (\ref{est-YZ}) follows from {\cite{GuaTess}, Remark 4.5, estimate (4.19).}
 Since the process $\Xi^{t,x}$ is ${\cal F}_{t,T}^W$-measurable (where ${\cal F}_{t,T}^W$ is the $\sigma$-algebra
generated by $W_r - W_t$, $r \in [t,T]$, augmented with the $\P$-null sets), it turns out that 
   $Y^{t,x}_t$
is measurable both with respect to ${\cal F}_{t, T }^W$ and ${\cal F}_t^W$; it follows that $Y^{t,x}_t$ is indeed
deterministic.

When $B$ is also differentiable with respect to $x$ (see \eqref{diff}) then the differentiability properties 
follow by \cite{fute}, Propositions 4.8 and  5.2, which can be applied in the same way also when the BSDE is ${K}$-valued. 
 \end{proof}  

Let $(Y^{t,x},Z^{t,x})$ be the solution of (\ref{fbsde2-wave-notilde}) assuming only Hypothesis \ref{base} and   \eqref{bbb}. 
By the previous result we can define the deterministic function $v : [0,T] \times H \to K$,  
\begin{equation}\label{v}
 v(t,x)= Y_t^{t,x} \in K, \quad (t,x)\in[0,T]\times H.
 \end{equation}    
 Assuming also the differentiability condition \eqref{diff},  the map defined in (\ref{map-x-toYZ}) is in particular continuous and
it is standard 
to check the following useful identities:  for any  $\, 0\leq t\leq s\leq\tau \leq T$,
\begin{equation}\label{Markov-prop1}
 Y_\tau^{t,x}= Y_\tau^{s,\Xi_s^{t,x}},\quad Z_\tau^{t,x}= Z_\tau^{s,\Xi_s^{t,x}},\; \P-\text{a.s.}.
\end{equation} 
The proof of  (\ref{Markov-prop1})
can be performed as for the real valued BSDEs (see \cite{fute}, formula (5.3)), and it is related to the
fact that the value of the processes $Y^{t,x}$ and $Z^{t,x}$ on the time interval $[s, T]$ is uniquely determined by the values
of $\Xi^{t,x}$ on the same interval.

 Moreover,
if  we assume  differentiability of $B$ (see \eqref{diff}),
 we get
  in particular that, 
 for $t\in [0,T]$, $v(t,\cdot):H\rightarrow K$ is G\^ateaux  
  differentiable on $H$, and, moreover, 
 applying  \eqref{Markov-prop1}, we have:
\begin{equation}\label{v-Markov}
v(\tau,\Xi_\tau^{t,x})=Y_\tau^{t,x},\, \;\;\; \tau \in [t,T], \; \P-a.s..
\end{equation}
Now   we want to prove that the  derivative $\nabla^G v(\tau,\Xi_\tau^{t,x})$ can be identified with $ Z_\tau ^{t,x}$ (see \eqref{dg}).
At first we prove such identification assuming that $B$ is also differentiable (see \eqref{diff}). Then in Theorem \ref{teo-identific-Z} we show that such identification
 holds true only assuming  
 \eqref{bbb}.

\begin{remark} {\em 
The  next identification property   with  $B$ differentiable (see \eqref{diff}) is here presented for the markovian BSDE (\ref{bsde2-markovian-wave-notilde-mild}) related to the Ornstein Uhlenbeck wave process; it remains true  for  general linear Markovian BSDEs with differentiable coefficients and final datum which is related to a forward stochastic equation with additive noise,  $A$ generator of a strongly continuous semigroup, $A$ instead of $-A$ in the backward equation, and Lipschitz continuous and G\^ateaux  differentiable drift. } 
 \end{remark}
  
\begin{lemma} 
\label{prop-identific-Z-diffle}Let $v$ be defined in (\ref{v}) and assume that all the hypotheses of Proposition
\ref{Teo:ex-regdep-BSDE}, including the differentiability of $B$ (see \eqref{diff}) hold true. Let $(Y^{t,x},Z^{t,x})$  be the solution of
(\ref{bsde2-markovian-wave-notilde-mild}).
Then, for any $\tau \in [0,T],$ a.e., we have, $\P$-a.s.,  
\begin{equation}\label{identific-Z}
\nabla^G v(\tau,\Xi_\tau^{t,x}) = Z_\tau ^{t,x}   \,\, \text{in} \; L_2(U,K) .
\end{equation}
\end{lemma} 
\dim The result can be seen as an extension of 
Theorem 6.1 in \cite{fute-infor} (see also Theorem 6.2 in \cite{fute})
to the case of a ${K}$-valued BSDE.
  Let  $\xi  \in  U$ and consider the real Wiener process 
$(W^\xi_\tau)_{\tau\ge 0}$, where
 \[
  W^\xi_\tau:= \langle \xi , W_{\tau} \rangle_U. 
\] 
Let $h \in {K}$. 
Using the group property if we set  $\widetilde  Y_{\tau}^{t,x} 
=  e^{- \tau {A}} Y_{\tau}^{t,x}$ we have, for   $\tau\in [0,T],$ $\P$-a.s.,
\begin{gather*}
\widetilde  Y_{\tau}^{t,x}=\int_\tau^Te^{-s{A}}G B(s,\Xi^{t,x}_s)\,ds+\int_\tau^Te^{-s{A}} 
 Z_s^{t,x}\, B(s,\Xi^{t,x}_s)\,ds -
 \int_\tau^Te^{-s {A}} Z^{t,x}_{s}\;d W_s
\\
= \widetilde  Y_{0}^{t,x} - \int_0^\tau e^{-s{A}}G B(s,\Xi^{t,x}_s)\,ds - \int_0^\tau e^{-s{A}} 
 Z_s^{t,x}\, B(s,\Xi^{t,x}_s)\,ds + 
 \int_0^\tau e^{-s {A}} Z^{t,x}_{s}\;d W_s
\end{gather*}    
and so  
\begin{gather*}
\langle \widetilde  Y_{\tau}^{t,x}, h \rangle =  
 \langle \widetilde  Y_{0}^{t,x} , h \rangle  - \int_0^\tau 
 \langle e^{-s{A}}G B(s,\Xi^{t,x}_s), h \rangle \,ds - \int_0^\tau \langle e^{-s{A}} 
 Z_s^{t,x}\, B(s,\Xi^{t,x}_s), h \rangle \,ds
\\ +
 \int_0^\tau \langle  d W_s , (Z^{t,x}_{s})^* e^{-s A^*  } h \rangle_U 
\end{gather*}  
($A^*$ denotes the adjoint of $A$, $A^* = -A$).
 We study  the joint quadratic variation between $\widetilde  Y_{}^{t,x}$
 and $ W^\xi$ (see, for instance, page 638 in \cite{fute-infor}). We find,
arguing as in the proof of Proposition 17 of \cite{DPFR}, 
\begin{equation} \label{joint}
\< \widetilde  Y_{}^{t,x}, W^{\xi}\>_{\tau} = \int_0^\tau \langle  \xi , (Z^{t,x}_{s})^* e^{-s A^* } h \rangle_U \, ds  = \int_0^\tau \langle e^{-s {A}} Z^{t,x}_{s} \xi ,   h \rangle_{ {K}} ds, \;\; \tau \in [0,T], \; \P\text{-a.s.}  
\end{equation}
Now we compute $\< \widetilde  Y_{}^{t,x}, W^{\xi}\>_{\tau}$
 in a different way, using \eqref{v-Markov0}.

Let us define $\widetilde v (t,x) = e^{-t {A}} v(t, x)$ so that we have $ \widetilde v(\tau,\Xi_\tau^{t,x})= \widetilde Y_\tau^{t,x}$. Moreover, we introduce the real function 
$$
\widetilde v^h(\tau, x) = \< \widetilde v(\tau, x), h \> = \langle v(\tau, x), e^{-\tau (A^*+ \lambda I)} h\rangle, \; \,\tau\in[0,T], \; x \in H.
$$
By \eqref{v} we know that  $\widetilde v^h \in   \calg^{0,1}([0,T]\times H,\R)$. Hence we can argue as in Lemma 6.3 of \cite{fute} (see also Lemma 6.4 in \cite{fute}) and obtain that the real 
process 
$$
(\widetilde v^h(\tau,\Xi_\tau^{t,x}))_{\tau\in[0,T]} = (\langle \widetilde  Y_{\tau}^{t,x}, h \rangle)_{\tau\in[0,T]} 
$$
admits joint quadratic variation with 
$W^{\xi}$ 
given by 
\[
 \<\widetilde v^h(\cdot, \Xi^{t,x}), W^{\xi}\>_{\tau}=
 \int_0^\tau \nabla^G \widetilde v^h(s,\Xi_s^{t,x}) \xi \,ds 
 = \int_0^\tau  \langle \nabla^G  v (s,\Xi_s^{t,x}) \xi, e^{-s{A}^*} h \rangle \,ds, \;\; \tau \in [0,T].
\]
Comparing this formula with \eqref{joint} we discover that
for $s \in [0,T]$, a.e., we have, $\P$-a.s.,
$$
\langle e^{-s {A}} Z^{t,x}_{s} \xi ,   h \rangle 
= \langle e^{- s {A}} \nabla^G  v(s,\Xi_s^{t,x}) \xi, h \rangle.
$$
Thanks to the separability of ${K}$ it follows that  $ e^{-s {A}} Z^{t,x}_{s} \xi 
=  e^{- s {A}} \nabla^G  v(s,\Xi_s^{t,x}) \xi $, for $s \in [0,T]$, a.e., 
$\P$-a.s.. The assertion follows easily.
\qed
    
We introduce now an approximation argument to smooth the coefficient $B$. We need such approximation  in the proof of next theorem.  
 Recall that for $s\in[0,T] $,
 $B(s,\cdot):H\rightarrow  U$, and
 \begin{equation}\label{B1}
GB(s,\cdot)= 
 \left(\begin{array}{l}
     0\\
 B(s,\cdot)
 \end{array}\right),
 \end{equation}
where  $B$ satisfies \eqref{bbb}.
 To perform the approximations of $B$ we follow
\cite{PZ}.  
For every  
$k\in\mathbb{N}$ we consider a nonnegative function $\rho_{k}\in C_{b}^{\infty}\left(  \mathbb{R}%
^{k}\right)  $ with compact support contained in the
ball of radius $\frac{1}{k}$ and such that $
{\displaystyle\int_{\mathbb{R}^{k}}}
\rho_{k}\left(  x\right)  dx=1$. Let $Q_{k}:H\longrightarrow\left\langle g_{1},...,g_{k}%
\right\rangle $ be the orthogonal projection on the linear space $\Lambda_k$ generated by
$g_{1},...,g_{k}$, where $(g_k)_{k\geq 1}$ is a  basis in $H$.
We identify $\Lambda_k $
with $\mathbb{R}^{k}$. For a bounded and continuous function $f:H\rightarrow U$ we set
\[
f^{k}\left(  x\right)  =\int_{\mathbb{R}^{k}}\rho_{k}\Big(  y-Q_{k}%
x\Big)  f\Big(  \sum_{i=1}^{k}y_{i}g_{i}\Big)  dy,
\]  
where for every $k\in\mathbb{N}$, $y_{k}=\left\langle y,g_{k}\right\rangle
_{H}$. It turns out that $f^{k}\in C_{b}^{\infty}\left( H,U\right)  $.

\noindent We will apply this approximation to $f=B(s,\cdot),\, s\in[0,T]$. By \eqref{bbb}
it follows that the sequence  $(B^k(t, \cdot))$ is
 equi-uniformly continuous on $H$, uniformly in $t \in [0,T]$, and $|B^k(t, x) - B(t,x)| \to 0$, as $k \to \infty$, for any $(t,x) \in [0,T] \times H$.
 Moreover, for $ s\in[0,T],$ $B^n(s, \cdot)$ is Fr\'echet differentiable on $H$
and    
\begin{equation}\label{stima-der-Bn}
\sup_{(s,x) \in [0,T]  \times H} \Vert \nabla B^n(s,x)\Vert_{L(H,U)} = c(n)
\end{equation}
where $c(n)\rightarrow \infty$ as $n\rightarrow \infty$.
 For any $n\geq 1$ let us consider the FBSDE (\ref{fbsde2-wave-notilde}) with $B^n$ in the place of $B$ 
  where again the precise meaning of the BSDE is
 given by its mild formulation: 
\begin{align} \label{bsde2-markovian-wave-notilde-mild-n}
 Y_{\tau}^{n,t,x}&=\int_\tau^Te^{-(s-\tau){A}}G B^n(s,\Xi^{t,x}_s)\,ds
+\int_\tau^Te^{-(s-\tau){A}} Z_s^{n,t,x}\, B^n(s,\Xi^{t,x}_s)\,ds\\&-
 \int_\tau^Te^{-(s-\tau){A}} Z^{n,t,x}_{s}\;d W_s.\nonumber
\end{align}
We know that the map in  (\ref{map-x-toYZ}) is G\^ateaux differentiable on $H$, since in  the BSDE (\ref{bsde2-markovian-wave-notilde-mild-n})  
 the coefficients are regular. Let $ \xi \in U$ and  consider the BSDE satisfied by the pair of processes
 $(\nabla_{G\xi}Y^{n,t,x},\nabla_{G\xi}Z^{n,t,x})$, which can be obtained by differentiating (\ref{bsde2-markovian-wave-notilde-mild-n})
 arguing as in \cite{fute}, Proposition 4.8, or following \cite{Gua1}, Proposition 4.4:
\begin{align} \label{bsde2-diff-markovian-wave-notilde-mild-n}
 \nabla_{G\xi}Y_{\tau}^{n,t,x}&=\int_\tau^Te^{-(s-\tau){A}}G \nabla_{} B^n(s,\Xi^{t,x}_s)e^{(s-t)A }G\xi\,ds
+\int_\tau^Te^{-(s-\tau){A}}\nabla_{G\xi} Z_s^{n,t,x}\, B^n(s,\Xi^{t,x}_s)\,ds\\&+
\int_\tau^Te^{-(s-\tau){A}} Z_s^{n,t,x}\,\nabla_{} B^n(s,\Xi^{t,x}_s)e^{(s-t)A }G\xi\,ds
 - \int_\tau^Te^{-(s-\tau){A}} \nabla_{G\xi}Z^{n,t,x}_{s}\;d W_s.\nonumber
\end{align}
 By applying estimate (\ref{est-YZ}) to (\ref{bsde2-diff-markovian-wave-notilde-mild-n}), and since $\Vert \nabla B^n(s,\cdot)\Vert_{L(H,U)}\leq
 c(n)$, where $c(n)$ does not depend on $s$ and $x$, we get
 \begin{equation}\label{est-diffYZ}
\E\big [ \sup_{\tau\in[0,T]}\vert  \nabla_{G\xi}Y^{n,t,x}_\tau\vert^2 \big ]+\E\int_0^T\Vert  \nabla_{G\xi} Z_\tau^{n,t,x}\Vert^2_{L_2(U,{K})} d \tau
  \leq c(n,T)^2 \,\vert \xi\vert_U^2,
 \end{equation}
where $c(n,T) >0$ is a constant that may blow up as $n\to\infty$, it depends on $T$ and $B$ but not on $x$ and $t$.
\newline Recalling  \eqref{v} we have 
\begin{equation*}\label{v_n}
 v^n(t,x)=Y^{n,t,x}_t, 
\end{equation*}
and we know  that $v^n\in\calg^{0,1}([0,T]\times H, {K})$;  by the previous computations we have, for any $n \ge 1$,
 \begin{equation}\label{est-nablaGvn}
\sup_{t \in [0,T], x\in H}\vert  \nabla_{\xi}^G v^n(t,x)\vert
  \leq c(n,T)\vert \xi\vert_U;
 \end{equation}
\newline Let $\tau=t$ and let us take the expectation in \eqref{bsde2-markovian-wave-notilde-mild-n}; we get 
\begin{equation} \label{bsde2-wave-semigroup}%
 Y^{n,t,x}_t=
\E\int_{t}^{T}%
e^{-(s-t){A}}G B^n(s,\Xi^{t,x}_s)\,ds
+\E\int_t^Te^{-(s-t){A}}  Z_s^{n,t,x} B^n(s,\Xi^{t,x}_s) \,ds, \;\;
t\in\left[  0,T\right]  ,
\;x\in H.
\end{equation}
Applying Lemma \ref{prop-identific-Z-diffle}, 
we can write (\ref{bsde2-wave-semigroup}) as
\begin{equation} \label{bsde2-wave-semigroup-bis}
 v^n(t,x)=
\E\int_{t}^{T}%
e^{-(s-t){A}}G B^n(s,\Xi^{t,x}_s)\,ds
+\E\int_t^Te^{-(s-t){A}}  \nabla^G v^n\left(s, \Xi_s^{t,x}\right) B^n(s,\Xi^{t,x}_s) \,ds,
\end{equation}
for $t\in\left[  0,T\right],
\;x\in H.$ Using  
 the ${K} $-valued OU transition semigroup  $(R_t)$  defined in
(\ref{wave-Hsemigroup}), we obtain
\begin{equation} \label{bsde2-wave-semigroup-vn}
v^n(t,x)
=\int_t^T
 R_{s-t}\left[ e^{-(s-t){A}}G B^n(s,\cdot)+ e^{-(s-t){A}} 
\nabla^Gv^n(s,\cdot)\, B^n(s,\cdot)
 \right] (  x) \, ds,
 \;t\in\left[  0,T\right]  ,
\;x\in H.
\end{equation}
Now   we want to show a convergence result of  
  $ Y^{n,t,x}_t$ to $ Y^{t,x}_t$ and  $ Z^{n,t,x}$ to $Z^{t,x}$.
The BSDE satisfied by $( Y^{n,t,x}- Y^{t,x},  Z^{n,t,x}- Z^{t,x})$
is 
\begin{equation*}
  \left\lbrace\begin{array}
 [c]{l}%
  -d\left(Y_{\tau}^{n,t,x}- Y_{\tau}^{t,x}\right)=-{A}\left(Y_{\tau}^{n,t,x}- Y_{\tau}^{t,x}\right)\,d\tau
  +\left( G B^n(\tau,\Xi^{t,x}_\tau)-G B(\tau,\Xi^{t,x}_\tau)\right) d\tau\\
   \quad+\left( Z^{n,t,x}_{\tau} \, B^n(\tau,\Xi^{t,x}_\tau)
    -  Z^{t,x}_{\tau} \, B(\tau,\Xi^{t,x}_\tau)\right) d\tau
    -\left( Z^{n,t,x}_{\tau}-Z^{t,x}_{\tau}\right)\;dW_\tau,
   \\\dis
  Y_{T}^{n,t,x}- Y_{T}^{t,x}=0.
 \end{array}
 \right.  
 \end{equation*} 
 By adding and subtracting $ Z^{t,x}_{\tau}B^n(\tau,\Xi^{t,x}_\tau)$
 this equation can be rewritten as 
 \begin{equation}
  \label{do1}
  \left\lbrace\begin{array}
 [c]{l}%
  -d\left( Y_{\tau}^{n,t,x}- Y_{\tau}^{t,x}\right)=-{A}\left(Y_{\tau}^{n,t,x}- Y_{\tau}^{t,x}\right)\,d\tau
  +\left( G B^n(\tau,\Xi^{t,x}_\tau)-G B(\tau,\Xi^{t,x}_\tau)\right) d\tau\\
   \quad -  Z^{t,x}_{\tau} \, \Big (B(\tau,\Xi^{t,x}_\tau)-B^n(\tau,\Xi^{t,x}_\tau) \Big)\,d\tau
    +\left( Z^{n,t,x}_{\tau} - Z^{t,x}_{\tau} \right)  B^n(\tau,\Xi^{t,x}_\tau)\, d\tau
  -\left( Z^{n,t,x}_{\tau}- Z^{t,x}_{\tau}\right)\;dW_\tau,
   \\\dis
   Y_{T}^{n,t,x}- Y_{T}^{t,x}=0.
 \end{array}
 \right.  
 \end{equation} 
Let $f_n(\tau) = \left( G B^n(\tau,\Xi^{t,x}_\tau)-G B(\tau,\Xi^{t,x}_\tau)\right) 
  -  Z^{t,x}_{\tau} \, \Big (B(\tau,\Xi^{t,x}_\tau)-B^n(\tau,\Xi^{t,x}_\tau) \Big)$. In the sequel we write 
    $ \Vert \cdot \Vert_{} $ instead of 
    $\Vert \cdot \Vert_{L_2(U, {K}))}$ to simplify notation. Let us fix $t \in [0,T]$.  By Remark 4.5, estimate (4.19) in {\cite{GuaTess},  we know that there exists $M= M(A)$ such that, for any $ \tau  \in [0,T]$, 
\begin{equation}
 \label{pe1}
\E  \vert  Y^{n,t,x}_{\tau}- Y^{t,x}_{\tau} \vert ^2
 +\E \int_{\tau}^T\Vert  Z^{n,t,x}_{s}- Z^{t,x}_{s} \Vert ^2\,ds
 \leq M (T-\tau)\E \int_\tau^T  |f_n(s)|^2 d s. 
\end{equation}     
 Recall  that the sequence $(B^n)$ is uniformly bounded on $[0,T] \times H$ by $\| B\|_{\infty}$.
 Let $\eta >0$ be small enough ($\eta M  \| B \|_{\infty} \le 1/2$). If $\tau \in [T - \eta, T]$ we get
\begin{gather*}
 \E \vert  Y^{n,t,x}_{\tau}- Y^{t,x}_{\tau} \vert ^2
 +  \frac{1}{2}\E \int_\tau^T\Vert  Z^{n,t,x}_{s}- Z^{t,x}_{s} \Vert ^2\,ds
 \leq M \eta \E \int_\tau^T  | G B^n(s, \Xi^{t,x}_s)-G B(s,\Xi^{t,x}_s)|^2 d s
\\
  + M \eta \, \E \int_\tau^T  \| Z^{t,x}_{s} \|^2 \, | B(s,\Xi^{t,x}_s)-B^n(s,\Xi^{t,x}_s) |_U^2 \, d s. 
\end{gather*}
Since  $Z^{t,x}\in L^2_\calp(\Omega \times [0,T],L_2(U, {K}))$,
by the pointwise convergence of $B^n(\tau,\cdot)$ to $B(\tau,\cdot)$ and  the dominated convergence theorem we get
\begin{equation}\label{convergence}
  \E \vert  Y^{n,t,x}_{\tau}- Y^{t,x}_{\tau} \vert ^2
 +\E \int_\tau^T\Vert  Z^{n,t,x}_{s}- Z^{t,x}_{s} \Vert ^2\,ds\rightarrow 0 \text{ as }n\rightarrow\infty,
\end{equation}
$\tau \in [T -\eta,T]$. Let now  $\tau \in  [(T - 2\eta)  \vee 0, T - \eta  ]$. We  consider \eqref{do1} on $[0, T - \eta]$ with the terminal condition  $Y_{T - \eta}^{n,t,x}- Y_{T - \eta}^{t,x}=0$. Arguing as before we obtain  \eqref{convergence} when   
$\tau  \in  [(T - 2\eta)  \vee 0, T - \eta  ]$.   Proceeding in this way we finally get \eqref{convergence} for any $\tau \in [0,T]$.

We mention two consequences of \eqref{convergence}. The first one is obtained with $\tau =t$ and gives, 
 setting  $v^n(t,x) = Y^{n,t,x}_{t}$,  
$v^n(t,x) \to v(t,x)$ pointwise as $n \to \infty$ on $[0,T] \times H$.
The second one is that, for $\tau \in [t,T]$, possibly passing to a subsequence,   we can pass to the limit, $\P$-a.s., in 
$
v^n(\tau,\Xi_\tau^{t,x})=Y_\tau^{n,t,x}
$
 and get 
 \begin{equation}\label{v-Markov0}
v(\tau,\Xi_\tau^{t,x})=Y_\tau^{t,x},\, \;\;\; \tau \in [t,T], \; \P-a.s..
\end{equation}
 


Now we are ready to prove the following result.

\begin{theorem}
 \label{teo-identific-Z}  Assume Hypothesis \ref{base} and  that  $B$ satisfy  \eqref{bbb}.
Let  $v$ be the function defined in (\ref{v}).

Then  $v \in B_b([0,T] \times H, H)$ and,  for any $t \in [0,T],$
$v (t, \cdot ) :  H \to {K}$ is $G$-differentiable on $H$ (see the definition after \eqref{dg}). 
   Moreover, for any $(t,x) \in [0,T] \times H$, the map: $\xi \mapsto $   
 $\nabla_{G \xi}v(t,x) = \nabla_{ \xi}^Gv(t,x)$  $ \in L(U, {K})$ and, for any $\xi \in U$,  $\nabla^G_{\xi} v \in B_b([0,T] \times H, {K})$ with  $\sup_{(t,x) \in [0,T] \times H} \| \nabla^Gv(t,x) \|_{L(U, {K})} < \infty$.
Finally, for any $ \tau \in [0,T] $, a.e., we have
\begin{equation}\label{identific-Z--nodiffle}
 \nabla^G v(\tau,\Xi_\tau^{t,x})=  Z_\tau ^{t,x}, \;\; \,\P \text{-a.s.}.
\end{equation}
\end{theorem}
 \dim To prove the result we will  pass to the limit as $n\rightarrow\infty$ in \eqref{bsde2-wave-semigroup-vn}.
 
We first note that
 by estimate (\ref{convergence}) $v^n(t,x)\rightarrow v(t,x)$ pointwise and so $v \in B_b([0,T] \times H, {K})$. 
 
 Then we show that there exists $L:[0,T]\times H\rightarrow L(U, {K})$ which is a bounded mapping,  such that, 
  for any $\xi\in U$, $L(t,x)\xi$ is measurable in $(t,x)\in [0,T]\times H$, and  
    $
    \nabla^G v^n(t,x)\rightarrow L(t,x)
    $  in $L(U, {K})$
 pointwise  as $n \to \infty$ on $[0,T] \times H$. 
 
 In order to obtain the assertion  we start to  study
 the difference $ v^n(t,x)-v^{n+p}(t,x)$, $n,p\geq 1$. In the sequel we set 
 ${\cal N}(0, Q_t) = \mu_t$. We have 
\begin{gather*}
 v^n(t,x)-v^{n+p}(t,x) =\int_{t}^{T}\int_{H}e^{-(s-t){A}}\left[ \nabla^G v^n(s,z+e^{(s-t){A}}x) \, 
B^n(s,z+e^{(s-t){A}}x)\right.\\
\left.-
 \nabla^G v^{n+p}(s,z+e^{(s-t){A}}x)\, B^{n+p}(s,z+e^{(s-t){A}}x)\right]
\mu_{s-t}(dz)\,ds\\
+\int_t^T\int_{H}e^{-(s-t){A}}\left(GB^n(s,z+e^{(s-t){A}}x)-GB^{n+p}(s,z+e^{(s-t){A}}x)\right)
\mu_{s-t}(dz).
\end{gather*} 
Since $v^n$ and $v^{n+p}$ are G\^ateaux differentiable in the space variable, by the smoothing properties of
the transition semigroup $(R_{s})$, we can differentiate both sides and 
 obtain for all $\xi\in U$ (cf. \eqref{nabla-R}) 
\begin{align*}
& \nabla^G_{\xi} v^n(t,x)-\nabla^G_{\xi} v^{n+p}(t,x) 
=\int_t^T\int_{H}e^{-(s-t){A}}\left(GB^n(s,z+e^{(s-t){A}}x)-GB^{n+p}(s,z+e^{(s-t){A}}x)\right)\\
&\quad\left\langle Q_{s-t}^{-1/2}%
e^{(s-t){A}}G\xi,Q_{s-t}^{-1/2}z\right\rangle 
\mu_{s-t}(dz)\,ds
\\
&+\int_{t}^{T}\int_{H}e^{-(s-t){A}}\left[  \nabla^G v^n(s,z+e^{(s-t){A}}x) \,
B^n(s,z+e^{(s-t){A}}x)\right.\\
&\quad- \left.\nabla^G v^{n+p}(s,z+e^{(s-t){A}}x) \,  B^{n+p}(s,z+e^{(s-t){A}}x) 
\right]
\left\langle Q_{s-t}^{-1/2}%
e^{(s-t){A}}G\xi,Q_{s-t}^{-1/2}z\right\rangle 
\mu_{s-t}(dz)\,ds.
\end{align*}
Now in order to apply the Cauchy criterion we note that by  {(\ref{hi11}) }
\begin{align*}
& \sup_{p \ge 1} \sup_{|\xi|_U =1}\vert \nabla^G_\xi v^n(t,x)-\nabla^G_\xi v^{n+p}(t,x)\vert \\
&\leq C \,\int_t^T(s-t)^{-\frac{1}{2}}
\left( \int_H \sup_{p \ge 1} 
\vert GB^n(s,z+e^{(s-t){A}}x)-GB^{n+p}(s,z+e^{(s-t){A}}x)\vert^2\mu_{s-t}(dz)\right)^{\frac{1}{2}}\,ds
\\
&+C \, \int_t^T(s-t)^{-\frac{1}{2}}
\left(\int_H \sup_{p \ge 1} \vert  \nabla^G v^n(s,z+e^{(s-t){A}}x) \, 
B^n(s,z+e^{(s-t){A}}x) \right.\\
&\left.-
\nabla^G v^{n+p}(s,z+e^{(s-t){A}}x)\,  B^{n+p}(s,z+e^{(s-t){A}}x) \vert
^2\mu_{s-t}(dz)\right)^{\frac{1}{2}}\,ds=I_n+II_n
\end{align*}
(to simplify notation we drop the dependence of $I_n$ and $II_n$ from  $(t,x)$).
We can easily apply the dominated convergence theorem, 
and letting $n\rightarrow\infty $ we get $ I_n\rightarrow 0$.
%
Concerning   $II_n$, by adding and subtracting 
$ \nabla^G v^n(s,z+e^{(s-t){A}}x) B^{n+p}(s,z+e^{(s-t){A}}x)$ we get 
\begin{align*}
II_n &\leq C \, \int_{t}^{T} (s-t)^{-\frac{1}{2}}
\left(\int_H \sup_{p \ge 1} \vert \nabla^G v^n(s,z+e^{(s-t){A}}x) \big[ 
B^n(s,z+e^{(s-t){A}}x)\right. \\ 
&\left.\quad-B^{n+p}(s,z+e^{(s-t){A}}x) \big] \vert
^2
 \mu_{s-t}(dz)\right)^{\frac{1}{2}}\,ds\\& +C \, \sup_{p \ge 1} \int_{t}^{T}(s-t)^{-\frac{1}{2}}
 \left(\int_H \vert  \big (\nabla^G v^n(s,z+e^{(s-t){A}}x)-\nabla^G v^{n+p}(s,z+e^{(s-t){A}}x) \big)\right.\\
 &\left.\quad
 B^{n+p}(s,z+e^{(s-t){A}}x)\vert
 ^2\mu_{s-t}(dz)\right)^{\frac{1}{2}}\,ds=II_n^a +II_n^b.
\end{align*}
In order to estimate $II_n^a$ we  show  that  the sequence $\left( \nabla^G_{} v^n \right)$
is equi-bounded from $[0,T] \times H$ with values in $L(U,K)$ 
(cf. \eqref{est-nablaGvn}). 
 We find,  setting $\| \cdot \|_{L} = \| \cdot\|_{L(U,K)}$,    with $\beta >0$,
\begin{align*}
& e^{\beta t}\Vert  \nabla^G v^n(t,x) \Vert_{L}
\\
 &\leq C \, e^{\beta t} \int_t^T(s-t)^{-\frac{1}{2}}
 \left(\int_H \vert B^n(s,z+e^{(s-t){A}}x)\vert^{2}_U
 \mu_{s-t}(dz)\right)^{\frac{1}{2}}\,ds
\\
 &+ C e^{\beta t}  \, \int_{t}^{T}(s-t)^{-\frac{1}{2}}
 \left(\int_{H}\vert \nabla^G v^n(s,z+e^{(s-t){A}}x)
 B^n(s,z+e^{(s-t){A}}x) \vert^2 \mu_{s-t}(dz)\right)^{\frac{1}{2}}\,ds
\\
  &\leq \, C e^{\beta t} \|B \|_{\infty} \int_t^T(s-t)^{-\frac{1}{2}} ds
   + \, C \|B \|_{\infty} \int_{t}^{T} \frac{ e^{- \beta (s-t)}} {(s-t)^{-\frac{1}{2}} } ds \; 
 \cdot  \sup_{t \in [0,T],\, y\in H} e^{\beta t}\Vert \nabla^G v^n(t,y) \Vert_{L}.
\end{align*}
Since the sequence $(B^n)$ is uniformly bounded,  by taking the supremum over $(t,x)\in [0,T] \times H$ and $\beta $ large enough we get 
on the left hand side we get
\begin{align} \label{bond}
\sup_{n \ge 1}\sup_{t \in [0,T], y\in H}\Vert  \nabla^G_{} v^n(t,y)  \Vert_{L(U,{K})} < \infty.
\end{align}
Coming back to the estimate of $II_n^a$, we find 
\begin{align*}
II_n^a& \leq  C_0 \int_t^T(s-t)^{-\frac{1}{2}}
\left(\int_H  \sup_{p \ge 1} \vert B^n(s,z+e^{(s-t){A}}x)-B^{n+p}(s,z+e^{(s-t){A}}x)\vert_U
^2\mu_{s-t}(dz)\right)^{\frac{1}{2}}\,ds,
\end{align*}
where $C_0$ is independent of $t, x$ and $n$.
By the dominated convergence theorem, using  the  pointwise convergence 
of the approximating sequence $(B^n)$, we find 
that $II_n^a \rightarrow 0$ as $n\rightarrow \infty$. 

Concerning  $II_n^b$, since $(B_n)$ and $(\nabla^G v^n)$ are equi-bounded we have:
\begin{align*}
 &\vert B^{n+p}(s,z+e^{(s-t){A}}x) \vert_U^2\, \Vert \nabla^G v^n(s,z+e^{(s-t){A}}x)-\nabla^G v^{n+p}(s,z+e^{(s-t){A}}x)\Vert^2_{L(U, {K})}
\\&\leq C_1
 \Vert \nabla^G v^n(s,z+e^{(s-t){A}}x)-\nabla^G v^{n+p}(s,z+e^{(s-t){A}}x)\Vert_{L(U, {K})},
 \end{align*}
 where $C_1$ is independent of $t, x$ and $n$, 
 Next, by the H\"older inequality, 
\begin{align*}
II_n^b& \leq c' \sup_{p \ge 1} \int_{t}^{T}(s-t)^{-\frac{1}{2}}
\left(\int_H    \Vert\ \nabla^G v^n(s,z+e^{(s-t){A}}x)-\nabla^G v^{n+p}(s,z+e^{(s-t){A}}x)\Vert_{L} \,
\mu_{s-t}(dz)\right)^{\frac{1}{2}}\,ds \\ 
\nonumber
&\leq c' \sup_{p \ge 1}
\left( \int_{t}^{T}(s-t)^{-\frac{2}{3}}\,ds\right)^{\frac{3}{4}}
\\& \cdot \, \left( \int_{t}^{T}\left(\int_H    \Vert \nabla^G v^n(s,z+e^{(s-t){A}}x)-\nabla^G v^{n+p}(s,z+e^{(s-t){A}}x)\Vert_{L} 
 \; \mu_{s-t}(dz)\right)^{2}\,ds\right)^{\frac{1}{4}} \\ \nonumber
 &  \leq c' \sup_{p \ge 1}
 \left( \int_{t}^{T}\int_H   \Vert \nabla^G v^n(s,z+e^{(s-t){A}}x)-\nabla^G v^{n+p}(s,z+e^{(s-t){A}}x)\Vert^{2}_{L} \, 
 \mu_{s-t}(dz)\,ds\right)^{\frac{1}{4}} \\ \nonumber
 &  = c'\sup_{p \ge 1}\left(
 \E\int_{t}^{T}  \Vert Z^{n,t,x}_s-Z^{n+p,t,x}_s\Vert^{2} 
 \,ds\right)^{\frac{1}{4}}\to 0
\end{align*}
as $n\to\infty$, for $(t,x) \in [0,T] \times H$, having applied (\ref{identific-Z}) and estimate (\ref{convergence}).
 Putting together these estimates we find
\[
\sup_{p \ge 1} \Vert \nabla^G v^n(t,y)-\nabla^G v^{n+p}(t,y)\Vert_{L(U, {K})} \rightarrow 0  \text{ as } n\rightarrow \infty,\;\;\; (t,y) \in [0,T] \times H.
\]
So we have proved that $ \nabla^G v^n(t,x)$ converges as $n\rightarrow \infty$ in $L(U, {K})$.
We set
\[
 L(t,x)=\lim_{n\rightarrow\infty} \nabla^G v^n(t,x).
\]
Clearly, for any $\xi \in U$,  $\nabla^G_{\xi} v^n \in B_b([0,T] \times H,  {K})$ and so $(t,x) \mapsto L(t,x) \xi \in B_b([0,T] \times H, {K})$.

Recall that $v^n$ satisfies (\ref{bsde2-wave-semigroup-vn}),
so, by passing to the limit in (\ref{bsde2-wave-semigroup-vn}),  by  the dominated convergence theorem,
we get
\begin{equation}
v(t,x)=\int_{t}^{T}%
R_{s-t}\left[  e^{-(s-t){A}}GB(s,\cdot)  + e^{-(s-t){A}}L(s,\cdot)B\left(  s,\cdot\right) \right] (  x) ds.
 \label{bsde2-wave-semigroup-limit}
\end{equation}
By differentiating (\ref{bsde2-wave-semigroup-vn}) in the direction $G\xi$,
we get for all $\xi\in U$
\begin{align*}
 \nabla^G_{\xi} v^n(t,x) =\int_t^T\int_{H}e^{-(s-t){A}}GB^n(s,z+e^{(s-t){A}}x)\left\langle Q_{s-t}^{-1/2}%
e^{(s-t){A}}G\xi,Q_{s-t}^{-1/2}z\right\rangle 
\mu_{s-t}(dz)\\
+\int_{t}^{T}\int_{H}e^{-(s-t){A}}  \nabla^G v^n(s,z+e^{(s-t){A}}x)
B^n(s,z+e^{(s-t){A}}x)
\left\langle Q_{s-t}^{-1/2}%
e^{(s-t){A}}G\xi,Q_{s-t}^{-1/2}z\right\rangle 
\mu_{s-t}(dz)\,ds.
\end{align*}
By passing to the limit,
we get
\begin{align*}
 L(t,x)\xi=\int_t^T\int_{H}e^{-(s-t){A}}GB(s,z+e^{(s-t){A}}x)\left\langle Q_{s-t}^{-1/2}%
e^{(s-t){A}}G\xi,Q_{s-t}^{-1/2}z\right\rangle 
\mu_{s-t}(dz)\\
+\int_{t}^{T}\int_{H}e^{-(s-t){A}} L(s,z+e^{(s-t){A}}x)
B(s,z+e^{(s-t){A}}x) \left\langle Q_{s-t}^{-1/2} 
e^{(s-t){A}}G\xi,Q_{s-t}^{-1/2}z\right\rangle 
\mu_{s-t}(dz)\,ds.
\end{align*}
By considering (\ref{bsde2-wave-semigroup-limit}) and taking into account the smoothing properties 
of the semigroup $(R_{t})$  (recall that, for any $s \in [0,T]$,
 $L(s, \cdot) B(s, \cdot) \in B_b(H, {K})$, and  Lemma \ref{teo:derHsemigroup} and \eqref{bbb}),   
we can easily obtain the desired differentiability  of 
 $v(t, \cdot ) : H  \rightarrow {K}$ along the directions $G \xi$, $\xi \in U$. 
 
Taking the directional derivative in (\ref{bsde2-wave-semigroup-limit}) we get, for all $\xi\in U$,
\[
\nabla^G_{\xi} v(t,x) =\int_{t}^{T} 
 \nabla^G_{\xi}
R_{s-t} \Big [ e^{-(s-t){A}}GB(s,\cdot)+e^{-(s-t){A}}L(s,\cdot)B\left(  s,\cdot\right) \big]  (  x) ds,
\;\;\; (t,x) \in [0,T] \times H,
 \]
and we finally deduce that $\nabla^G v(t,x)=L(t,x)$, $(t,x) \in [0,T] \times H$.
By (\ref{convergence}) we know that   
\[
  Z^{n,t,x}\rightarrow  Z^{t,x} \text{ in }L^2(\Omega\times [0,T]; L_2(U, {K})).
\]
Since, for any $n \ge 1,$ $ \tau \in [0,T] $, a.e., we have
$$
 \nabla^G v^n(\tau,\Xi_\tau^{t,x})=  Z_\tau ^{n,t,x}, \;\; \,\P\text{-a.s.},
$$
we get easily that \eqref{identific-Z--nodiffle} holds. The proof is complete.
\qed

\subsection{Additional regularity for the function $v(t,x) = Y_t^{t,x}$}\label{sez-reg-y}

Here we prove additional regularity  properties for the function $v(t,x)$ defined in (\ref{v}).
By the representation formula  given in (\ref{bsde2-markovian-wave-notilde-mild}) using the OU semigroup $(R_t)$
we know that 
\begin{align}\label{bsde2unobis}
v(t,x)
=&\int_t^T R_{s-t}\left[e^{-(s-t){A}}G B(s,\cdot)\right](x)\,ds+
\int_t^T R_{s-t}\left[e^{-(s-t){A}} \nabla^Gv(s,\cdot) B(s,\cdot) \right](x)\,ds. 
\end{align}
Hence $v$ satisfies  an integral equation like \eqref{bsde2uno} which has been studied 
in Theorem \ref{forse111}.
\begin{lemma}\label{forse11}
 Let Hypotheses \ref{base} and 
\ref{ip-b1} hold true. Then the function $v$ defined in (\ref{v}) coincides with the function $u$,  unique solution 
of (\ref{bsde2uno}) given in Lemma  \ref{forse13} and Theorem \ref{forse111}. 
 
\end{lemma}  
\dim 
By Theorem \ref{teo-identific-Z} we know that 
that  $v(t,x)$  belongs to $B_b([0,T] \times H, {K})$ and, moreover, there exists 
$\nabla^G v : [0,T] \times H \to L(U, {K})$ which is bounded and such that, for any $\xi \in U$, $ \nabla^G_{\xi} v\in B_b([0,T] \times H, { K})$.

We consider the difference 
\begin{equation} \label{di33}
 u(t,x)-  v(t,x)= 
\int_t^T R_{s-t}\left[e^{-(s-t){A}} \big(    
 \nabla^Gu(s,\cdot) B(s,\cdot) -\nabla^G v(s,\cdot) B(s,\cdot) \big)\right](x)\,ds
\end{equation} 
 and take the $\nabla^G$-derivative:  
\begin{align*}
\Vert \nabla^G u(t,x)- \nabla^G v(t,x)\Vert_{L(U, {K})}=
\Big \Vert\int_t^T \nabla^GR_{s-t}\left[e^{-(s-t){A}} \big(
 \nabla^Gu(s,\cdot) B(s,\cdot) -\nabla^G v(s,\cdot) B(s,\cdot) \big)\right](x)\,ds \Big\Vert_{L},
\end{align*}
where $\| \cdot \|_{L} = \| \cdot \|_{L(U,   {K})} $.
 Since,  $\nabla^Gu(\cdot) B(\cdot)$ and $\nabla^G v(\cdot) B(\cdot)$ both belong to $B_b([0,T] \times H, {K})$ we can apply Lemma \ref{teo:derHsemigroup} and obtain, for  $\beta >0$, $t \in [0,T]$,
$$
\sup_{ x \in H} e^{\beta t}\Vert \nabla^G u(t,x)- \nabla^G v(t,x)\Vert_{L}
$$$$\le  C \| B\|_{\infty} \int_t^T \frac {e^{-\beta (s-t)}} {(s-t)^{\frac{1}{2}} } ds  \cdot \sup_{x\in H , s \in [0,T]} e^{\beta s}\Vert \nabla^Gu(s,x)-\nabla^Gv(s,x)\Vert_{L}
$$$$\le C_{\beta , T} \sup_{x\in H , s \in [0,T]} e^{\beta s}\Vert \nabla^Gu(s,x)-\nabla^Gv(s,x)\Vert_{L }, \;\;\; t \in [0,T],
$$
where $C_{\beta, T} \to 0 $ as $\beta \to \infty$. By choosing $\beta$ large enough, we get $\sup_{x\in H , s \in [0,T]} \Vert \nabla^Gu(s,x)-\nabla^Gv(s,x)\Vert_{L}$ $ =0$. So by \eqref{di33} we get that $u$ and $v$ coincide.
 \qed

\section{Strong uniqueness for the wave equation}\label{sez-uniq-wave}

In this section 
we show how to remove the ``bad'' term $B$ of equation (\ref{wa1}), i.e.,
\begin{equation}
\left\{
\begin{array}
[c]{l}%
dX_\tau^{x}  =AX_\tau ^{x} d\tau+GB(t,X_\tau ^{x} )d\tau+GdW_\tau  ,\text{ \ \ \ }\tau\in\left[
0,T\right], \\
X_0^{x}   =x,
\end{array}
\right.  \label{waveeqabstract-holder-bis}
\end{equation}
and get the  main pathwise uniqueness result.
Let $x \in H$. We consider a (weak) mild solution $(X_\tau^{t,x})$
$=(X_\tau^{t,x})_{\tau \in [0,T]}$ as in \eqref{mild}:
\begin{equation}X_\tau ^{t,x} =e^{(\tau-t)A}x+\int_t^\tau e^{(\tau-s)A}GB(s,X_s)ds+\int_t^\tau e^{(\tau-s)A}GdW_s  ,\text{  }\tau\in\left[
t,T\right], \;\;  X_\tau ^{t,x} =x, \; \;\tau \le t.
\label{waveeqabstract-holder-bis-mild}%
\end{equation}
This in particular is a continuous $H$-valued process defined and adapted on  a  stochastic basis  
$\left(
\Omega, {\mathcal F},
 ({ \mathcal F}_{t}), \P \right) $,  
  on which it is defined a
 cylindrical $U$-valued ${\mathcal F}_{t}$-Wiener process $W$.
Let us consider the FBSDE
\begin{equation}
 \left\lbrace\begin{array}
[c]{l}%
dX_{\tau}^{t,x}  =AX_{\tau}^{t,x} d\tau+GB(\tau,X_{\tau}^{t,x})d\tau+GdW_\tau
,\text{ \ \ \ }\tau\in\left[  t,T\right], \\ \dis
X_{\tau}^{t,x} =x, \;\;\; \tau \in [0,t],\\ \dis
 -d\widetilde Y_{\tau}^{t,x}=- A\widetilde Y_{\tau}^{t,x} d\tau+G B(\tau,X^{t,x}_\tau)\,d\tau 
-\widetilde Z^{t,x}_{\tau}\;dW_\tau,
 \qquad \tau\in [0,T],
  \\\dis
  \widetilde Y_{T}^{t,x}=0.
\end{array}
\right.  \label{fbsde-wave}%
\end{equation}
  The precise meaning of the BSDE in equation (\ref{fbsde-wave}) is
 given by its mild formulation 
 \begin{equation} \label{bsde-markovian-wave-mild}
 \widetilde Y_{\tau}^{t,x}=\int_\tau^Te^{-(s-\tau){A}}G B(s,X^{t,x}_s)\,ds
-\int_\tau^T e^{-(s-\tau){A}}\widetilde Z^{t,x}_{s} dW_s,
 \quad \tau\in [0,T].
\end{equation}
 Let us set
$$\widetilde W_{\tau}=W_\tau+\int_0^\tau B(s,X_s^{t,x})ds,\;\; \tau \in [0,T].
$$ By the Girsanov theorem, see e.g. \cite{DPsecond}
and \cite{Ondre04}, 
 there exists
a probability measure $\widetilde\P$
such that on $(\Omega, \calf,  ({\cal F}_t)_{t \in [0,T]}, \widetilde \P)$ the process $(\widetilde W_{\tau})$
is a cylindrical Wiener process  in $U$ up to time $T$. 
In the stochastic basis $(\Omega, \calf, ({\cal F}_t)_{t \in [0,T]}, \widetilde \P)$ the FBSDE (\ref{fbsde-wave}) can be rewritten as
\begin{equation}
\left\{
\begin{array}
[c]{l}%
dX_{\tau}^{t,x}  =AX_{\tau}^{t,x} d\tau+Gd\widetilde W_{\tau}
,\text{ \ \ \ }\tau\in\left[  t,T\right], \\ \dis
X_{\tau}^{t,x} =x,\;\; \tau \in [0,t],\\ \dis
 -d \widetilde Y_{\tau}^{t,x}=-{A}\widetilde Y^{t,x}_\tau\,d\tau+G B(\tau,X^{t,x}_{\tau})\;d\tau +
   \,\widetilde Z^{t,x}_{\tau} B(\tau,X^{t,x}_{\tau}) d \tau 
 -\widetilde Z^{t,x}_{\tau}\;d\widetilde W_{\tau},
 \qquad \tau\in [0,T],
  \\\dis
  \widetilde Y_{T}^{t,x}=0, 
\end{array}
\right.  \label{fbsde1-wave}
\end{equation}
 \begin{equation} \label{bsde1-markovian-wave-mild}
\widetilde Y_{\tau}^{t,x}=\int_\tau^Te^{-(s-\tau){A}}G B(s,X^{t,x}_s)\,ds+\int_\tau^Te^{-(s-\tau){A}} \widetilde Z_s^{t,x}\, B(s,X^{t,x}_s)\,ds-
 \int_\tau^Te^{-(s-\tau){A}}Z^{t,x}_{s}\;d\widetilde W_s,
\end{equation}
$\tau\in [0,T]$. By the strong uniquenes for equation \eqref{wave eq ab det},  $X^{t,x}_{}$ is an Ornstein-Uhlenbeck process starting from $x$ at  $t$ which  is    ${\cal F}_{t,T}^{\widetilde W}$-measurable (where ${\cal F}_{t,T}^{\widetilde W}$ is the completed $\sigma$-algebra
generated by $\widetilde W_r - \widetilde W_t$, $r \in [t,T]$).
The law of $(X^{t,x}, \widetilde Y^{t,x},\widetilde Z^{t,x})$
depends only on the coefficients of the FBSDE (\ref{fbsde1-wave}) and does not depend on the probability
space on which it is defined  the cylindrical Wiener process. 
Thus the law of $(X^{t,x}, \widetilde Y^{t,x},\widetilde Z^{t,x})$ coincides with
the one  of 
$(\Xi^{t,x}, Y^{t,x},  Z^{t,x})$ solution of the FBSDE (\ref{fbsde2-wave-notilde}).

Moreover  $Y_{t}^{t,x}$ and $\widetilde Y_{t}^{t,x}$ are both deterministic and so they define a unique function  $v(t,x)$ given in \eqref{v}.  Moreover, we have, for any $\tau \in [0,T]$, 
\begin{equation} \label{def-tildeYtildeZ}
\widetilde Y_\tau^{t,x}=v(\tau,X_{\tau}^{t,x}), \;\;\; \P-a.s.;
 \; \text{  for any $\tau \in [0,T]$ a.e.,} \;\; 
\widetilde Z_\tau^{t,x}=\nabla^G v(\tau,X_{\tau}^{t,x}), \;\; \P-a.s. 
\end{equation}
(cf. \eqref{v-Markov0} and \eqref{identific-Z--nodiffle}).
In order to prove strong existence of a mild solution to equation (\ref{waveeqabstract-holder-bis}),
we will rewrite in a different way (\ref{waveeqabstract-holder-bis-mild}),
removing the term $\dis\int_t^\tau e^{(t-s)A}GB(s,X_s)ds$
by means of the BSDE in (\ref{bsde1-markovian-wave-mild}); when
 $t=0$ we denote, for brevity, by $(\widetilde Y^x, \widetilde Z^x)$ the process $(\widetilde Y^{0,x}, \widetilde Z^{0,x})$. 
\begin{proposition}
 Let Hypotheses \ref{base} and \ref{ip-b1} hold true.
Then a (weak) mild solution $X^x = (X_{\tau}^x)$ of 
(\ref{waveeqabstract-holder-bis-mild})  
starting at $t=0$  satisfies, for any $\tau \in [0,T]$, $\P$-a.s., 
 \begin{align}\label{waveeq-nobad}
  X_\tau ^{x}&= {e^{\tau A}x+e^{\tau A}v(0,x)-v(\tau,X_\tau^{x})+\int_0^\tau e^{(\tau-s)A}\widetilde Z_s^x\;dW_s 
  +\int_0^\tau e^{(\tau-s)A}GdW_s}\\ \nonumber
  &= e^{\tau A}x+e^{\tau A}v(0,x)-v(\tau,X_\tau^{x})+\int_0^\tau e^{(\tau-s)A}\nabla^G v(s,X_s^x)\;dW_s 
  +\int_0^\tau e^{(\tau-s)A}GdW_s
 \end{align}
\end{proposition}
\dim 
Let us fix $\tau \in [0,T]$.  Writing (\ref{bsde-markovian-wave-mild}) for $t=0$ and $\tau =0$ we find, $\P$-a.s., 
\begin{equation} \label{bsde-markovian-wave-mild-bis-0}
 v(0,x)=\widetilde Y_{0}^{x}=\int_0^Te^{-sA}G B(s,X^{x}_s)\,ds - \int_0^Te^{-sA}\widetilde Z^{x}_{s}\;dW_s
\end{equation}
$$
= \int_0^{\tau}e^{-sA}G B(s,X^{x}_s)\,ds - \int_0^{\tau}e^{-sA}\widetilde Z^{x}_{s}\;dW_s + \int_{\tau}^Te^{-sA}G B(s,X^{x}_s)\,ds - \int_{\tau}^Te^{-sA}\widetilde Z^{x}_{s}\;dW_s.
$$
In (\ref{bsde-markovian-wave-mild}) with $t=0$ we apply to both sides the bounded linear operator $e^{-\tau A}$, we get
\begin{equation} \label{bsde-markovian-wave-mild-ter}
 e^{-\tau A}\widetilde Y_{\tau}^{x}=\int_\tau^Te^{-sA}G B(s,X^{x}_s)\,ds-\int_\tau^Te^{-sA}\widetilde Z^{x}_{s}\;dW_s,
 \qquad \tau\in [0,T].s
\end{equation}
Using \eqref{def-tildeYtildeZ} we obtain, $\P$-a.s., 
\begin{align}\label{v0}
 v(0,x)&=e^{-\tau A} \widetilde Y_\tau^{x}+\int_0^\tau e^{-sA}G B(s,X^{x}_s)\,ds -\int_0^\tau e^{-sA}\widetilde Z ^{x}_{s}\;dW_s\\ 
 \nonumber
&=e^{-\tau A}v(\tau,X_\tau^{x})+\int_0^\tau e^{-sA}G B(s,X^{x}_s)\,ds-\int_0^\tau e^{-sA}\nabla^G v(s,X_s^x)\;dW_s. 
 \end{align}
In particular from (\ref{v0}) we get 
\begin{equation}\label{v0-I}
\int_0^\tau e^{-sA}G B(s,X^{x}_s)\,ds=v(0,x)-e^{-\tau A}v(\tau,X_\tau^{x})+\int_0^\tau e^{-sA}\nabla^G v(s,X_s^x)\;dW_s
 \end{equation}
and by applying the bounded linear operator $e^{\tau A}$ to both sides we deduce that, $\P$-a.s., 
\begin{equation}\label{v0-II}
\int_0^\tau e^{(\tau-s)A}G B(s,X^{x}_s)\,ds=e^{\tau A}v(0,x)-v(\tau,X_\tau^{x})+\int_0^\tau e^{(\tau-s)A}\nabla^G v(s,X_s^x)\;dW_s.
 \end{equation}
Since $
X_{\tau} ^{x} - e^{\tau A}x - \dis\int_0^\tau e^{(\tau-s) A}GdW_s = \dis\int_0^\tau e^{(\tau -s) A}GB(s,X_s)ds
$
we get (\ref{waveeq-nobad}).
\qed
\begin{remark}\label{remark-final1} {\em 
  Notice that formula (\ref{waveeq-nobad}) does not coincide with formula (7) in \cite{DPF}, which is obtained by the so-called It\^o-Tanaka trick.
In fact our function $v$ 
   (see \ref{v}) and the function
 $U$ used in the paper \cite{DPF} are different, and we can see this by comparing  
 (\ref{bsde2unobis})  in the present paper with the mild formula (16) in \cite{DPF}.  
Following the procedure in \cite{DPF}, one  should  consider $U:[0,T]\times H\to H$  represented by the real functions
$U_n:=\<U,e_n \>: [0,T]\times H \rightarrow \R$, where $(e_n)_{n\geq1}$
 is a  basis in $H$, and $U_n$ is the solution to the linear Kolmogorov equation
 \begin{equation}\label{KolmoformaleDPF}
  \left\{\begin{array}{l}\dis
\frac{\partial U_n(t,x)}{\partial t}+\call_t [U_n(t,\cdot)](x)= - G B_n(t,x)
,\,
x\in H, \;\; t \in [0,T],
\\
\dis U_n(T,x)=0.
\end{array}\right.
\end{equation}
where
$ \call_t[f](x)=\frac{1}{2} Tr (\;GG^*\; \nabla^2f(x))
+ \< Ax,\nabla f(x)\>+ \<GB(t,x), \nabla f(x)\>;$ one can  solve (\ref{KolmoformaleDPF}) with techniques similar to the ones used also in \cite{G}.
On the other hand, from \eqref{bsde2unobis}
we   formally see that  $v$
is a
$K$-valued solution of the following  equation which contains the   operator $A$:
 \begin{equation}\label{KolmoformaleHnostra}
  \left\{\begin{array}{l}\dis
\frac{\partial v(t,x)}{\partial t}+\call_t [v(t,\cdot)](x)=Av(t,x) - G B(t,x)
,\,
x\in H, \;\; t \in [0,T],
\\
\dis v(T,x)=0.
\end{array}\right.
\end{equation}
}
 \end{remark} 

\begin{theorem} \label{uni1}
Let Hypotheses  \ref{base} and \ref{ip-b1} hold true.
Then for equation (\ref{wa1}) pathwise uniqueness holds (starting from any initial condition $x \in H$).

  Moreover, there exists $c_T>0$ such that if  $X_\tau^{x_1}$  and $ X_\tau^{x_2}$ are two (weak) mild solutions starting from $x_1$ and $x_2 \in H$ at $t=0$ (defined on the same stochastic basis)  {such that $x_1-x_2\in K$ then $\P$-a.s. $X^{x_1}_{\tau}-X^{x_2}_{\tau}  \in K $ for any   $\tau \in [0,T]$ and}      
  \begin{equation}\label{lip-dependence}
  {  \sup_{\tau \in [0,T]} \E \vert X_\tau^{x_1}-X_\tau^{x_2}\vert^2_K \leq c_T \vert x_1-x_2\vert^2_{K}. } 
  \end{equation}
\end{theorem}
\begin{proof}   
We prove \eqref{lip-dependence} which implies the pathwise uniqueness starting from any $x \in H$. Indeed if $x_1 = x_2 =x$ then \eqref{lip-dependence}  implies that 
 $\P$-a.s., 
  $X_\tau^{x_1}=X_\tau^{x_2}$, $\tau \in [0,T]$.  
  
Let us fix $x_1, x_2 \in H$ with {$x_1-x_2 \in K$} and consider two (weak) mild solutions $X^1$ and $X^2$ defined on the same stochastic basis, with respect to the same cylindrical Wiener process and starting respectively from $x_1$ and $x_2$ at time $t=0.$
Notice that 
$$
X^1_{\tau} - X^2_{\tau}
= e^{\tau A}(x_1 - x_2) +\int_0^\tau e^{(\tau-s)A}\Big( GB(s,X^1_s)-GB(X^2_s)\Big)\,ds,
$$
and both $e^{\tau A}(x_1 - x_2) $ and the integral take their values in $K$. 
 Indeed $x_1 - x_2 \in K$,  $GB(s, \cdot ) \in K$ and $e^{rA} : K \to K$
 (cf. \eqref{evolve}).

Let  $T_0 \in (0,T]$  be such that   
 $h(T_0) \cdot ( \sup_{t \in [0,T]}\| B(t,\cdot) \|_{\alpha }) $ $  \le 1/4$ (see \eqref{de22}).  
     
We  consider the FBSDE (\ref{fbsde2-wave-notilde})
 with $T = T_0$ {and we denote its solution  again $(\widetilde Y^{x},\widetilde Z^{x})$}. We find the  function $v^{(0)} : [0,T_0] \times H \to K$ according to \eqref{v} with $T= T_0$.
By  \eqref{waveeq-nobad}  we know that  
\begin{gather} \label{see} 
X^1_{\tau} - X^2_{\tau}
= e^{\tau A}(x_1 - x_2) + e^{\tau A}[v^{(0)}(0,x_1) - v^{(0)}(0, x_2)]
\\ \nonumber 
 - [v^{(0)}(\tau,X_\tau^{1}) - v^{(0)}(\tau,X_\tau^{2})]
+{\int_0^\tau e^{(\tau-s)A}[\widetilde Z^{x_1}_s- \widetilde Z^{x_2}_s]\;dW_s,}\;\; \tau \in [0,T_0], 
\end{gather}  
 where $\int_0^\tau e^{(\tau-s)A}[\widetilde Z^{x_1}_s- \widetilde Z^{x_2}_s]\;dW_s = \int_0^\tau e^{(\tau-s)A}[ \nabla^G v^{(0)}(s,X_s^1) - \nabla^G v^{(0)}(s,X_s^2) ]\;dW_s . $

 
By  the regularity properties of $v^{(0)}$, see Theorem  \ref{forse111}, 
Lemma \ref{forse13} and 
Lemma \ref{forse11}, we  get    
 \begin{gather*}
 \vert e^{\tau A}(x_1 - x_2)\vert_K + \vert e^{\tau A}[v^{(0)}(0,x_1) - v^{(0)}(0, x_2)]\vert_K +
\vert v^{(0)}(\tau,X_\tau^{1}) - v^{(0)}(\tau,X_\tau^{2})\vert_K \\ \leq C\vert x_1-x_2\vert_K +\frac{1}{3}\vert X_\tau^1-X_\tau^2\vert_K,\;\;\; \tau \in [0,T_0].
\end{gather*}
 Concerning  the stochastic integral, we have 
 (see  \cite{DZergo} page 57 or \cite{DPsecond}, Section 4.3) 
\begin{gather} \label{itoo}
\E \Big | \int_0^\tau e^{(\tau-s)A}[ \nabla^G v^{(0)}(s,X_s^1) - \nabla^G v^{(0)}(s,X_s^2) ]\;dW_s \Big|^2_K 
\\  \nonumber 
\le  \E  \int_0^{\tau} \| \nabla^G v^{(0)}(s,X_s^1) - \nabla^G v^{(0)}(s,X_s^2) \|_{L_2(U, K)}^2 ds.  
\end{gather}        
 It the sequel  we will prove that  $\E  \int_0^{\tau} \| \nabla^G v^{(0)}(s,X_s^1) - \nabla^G v^{(0)}(s,X_s^2) \|_{L_2(U, K)}^2 ds$ is finite and we will provide a bound for it.  
   
Let us consider a  basis $(e_k) $ in $U$; 
 by the regularity properties of $v^{(0)}$ we get      
  \begin{equation*}
\E  \int_0^{\tau} \| \nabla^G v^{(0)}(s,X_s^1) - \nabla^G v^{(0)}(s,X_s^2) \|_{L_2(U,K)}^2 
ds = \sum_{j \ge 1} \E  \int_0^{\tau}  | \nabla^G_{e_j} v^{(0)}(s,X_s^1) - \nabla^G_{e_j} v^{(0)}(s,X_s^2) |_{K}^2 ds.  
\end{equation*}
Now, using a basis $(f_m)$ in $K$, we write, for any $s \in [0,\tau]$, $\P$-a.s., 
 \begin{gather*} 
   \sum_{j \ge 1} | \nabla^G_{e_j} v^{(0)}(s,X_s^1) - \nabla^G_{e_j} v^{(0)}(s,X_s^2) |_{K}^2 
=  \sum_{j \ge 1} \sum_{m \ge 1} \langle \nabla^G_{e_j} v^{(0)}(s,X_s^1) - \nabla^G_{e_j} v^{(0)}(s,X_s^2), f_m    \rangle_K ^2 
\\  
= \sum_{m \ge 1} \sum_{j \ge 1}   \langle \nabla^G_{e_j} v^{(0)}(s,X_s^1) - \nabla^G_{e_j} v^{(0)}(s,X_s^2), f_m    \rangle_K ^2 =
\sum_{m \ge 1}   \sum_{j \ge 1} [     \nabla^G_{e_j} v^{(0)}_m (s,X_s^1) - \nabla^G_{e_j} v^{(0)}_m (s,X_s^2) ] ^2 
\\ 
= \sum_{m \ge 1}   |     \nabla^G v^{(0)}_m (s,X_s^1) - \nabla^G  v^{(0)}_m (s,X_s^2) |^2_U,   
\end{gather*}      
using  $ v^{(0)}_m = \langle   v^{(0)} , f_m \rangle_K $  and noting that 
$\nabla^G   v^{(0)}_m (s,X_s^1) \in L(U, \R)$ can be identified by the Riesz theorem with a unique element in $U.$    We have obtained
\begin{gather*}  
\sum_{j \ge 1}   \int_0^{\tau}  | \nabla^G_{e_j} v^{(0)}(s,X_s^1) - \nabla^G_{e_j} v^{(0)}(s,X_s^2) |_{K}^2 ds 
=   \int_0^{\tau} \sum_{m \ge 1}   |     \nabla^G v^{(0)}_m (s,X_s^1) - \nabla^G  v^{(0)}_m (s,X_s^2) |^2_U ds
\\
=      \int_0^{\tau} \sum_{m \ge 1} \sup_{|a|_U =1}  |     \nabla^G_a v^{(0)}_m (s,X_s^1) - \nabla^G_a  v^{(0)}_m (s,X_s^2) |^2  ds.
\end{gather*}
Hence we have, $\P$-a.s., 
 \begin{gather*}
 \sum_{j \ge 1}   \int_0^{\tau}  | \nabla^G_{e_j} v^{(0)}(s,X_s^1) - \nabla^G_{e_j} v^{(0)}(s,X_s^2) |_{K}^2 ds
\\  =  \sum_{m \ge 1}   \int_0^{\tau}   \sup_{|a|_U =1}  \Big |  \int_0^1 \nabla^K \nabla^G_a v^{(0)}_m (s,X_s^1 + r (X_s^2 - X_s^1)) \, [X_s^2 - X_s^1]dr
\Big|^2  ds 
\\
=  \int_0^{\tau}  \sum_{m \ge 1}   \sup_{|a|_U =1}  \Big |  \int_0^1 \nabla_{[X_s^2 - X_s^1]}\nabla^G_a v^{(0)}_m (s,X_s^1 + r (X_s^2 - X_s^1)) \, dr
\Big|^2  ds. 
\end{gather*}
   Moreover,  we find, $\P$-a.s., $r \in [0,1]$, $s \in [0, \tau]$, 
 \begin{gather}\label{sss1}
 \sum_{m \ge 1}   \sup_{|a|_U =1}  \Big |  \int_0^1 \nabla_{[X_s^2 - X_s^1]}\nabla^G_a v^{(0)}_m (s,X_s^1 + r (X_s^2 - X_s^1)) \, dr \Big|^2
 \\ \nonumber 
 \le \int_0^1
 \sum_{m \ge 1}   \sup_{|a|_U =1}   | \nabla_{[X_s^2 - X_s^1]}\nabla^G_a v^{(0)}_m (s,X_s^1 + r (X_s^2 - X_s^1)) |^2   \, dr  
 \\ \nonumber 
 \le  C_T |X_s^2 - X_s^1|_K^2  \, \sup_{t \in [0,T]}\| B(t, \cdot )\|_{ C^{\alpha}_b(H, U)}^2. 
\end{gather} 
 In the last inequality we have used Theorem \ref{forse111} with $u= v^{(0)}$ and $k = X_s^2 - X_s^1$.  
 
 Coming back to \eqref{itoo}  we  obtain  
\begin{gather*}
\E \Big | \int_0^\tau e^{(\tau-s)A}[ \nabla^G v^{(0)}(s,X_s^1) - \nabla^G v^{(0)}(s,X_s^2) ]\;dW_s \Big|^2_K 
\\ \le   \sum_{j \ge 1} \E  \int_0^{\tau}  | \nabla^G_{e_j} v^{(0)}(s,X_s^1) - \nabla^G_{e_j} v^{(0)}(s,X_s^2) |_{K}^2 ds
\\
\le   \E  \int_0^{\tau} ds \int_0^1
 \sum_{m \ge 1}   \sup_{|a|_U =1}   | \nabla_{[X_s^2 - X_s^1]}\nabla^G_a v^{(0)}_m (s,X_s^1 + r (X_s^2 - X_s^1)) |^2   \, dr  
 \\ \nonumber 
 \le  C_T \sup_{t \in [0,T]}\| B(t, \cdot )\|_{ C^{\alpha}_b(H, U)}^2 \,   \E  \int_0^{\tau}    |X_s^2 - X_s^1|_K^2   ds.
\end{gather*}
  Starting from \eqref{see} and using the previous estimates, we can  apply the  Gronwall lemma and  obtain 
   \begin{equation}\label{lip1} 
   \sup_{\tau \in [0,T_0]} \E \vert X_\tau^{x_1}-X_\tau^{x_2}\vert^2_K \leq c_T \vert x_1-x_2\vert^2_K. 
  \end{equation}
  If $T_0 <T$  we  consider the FBSDE (\ref{fbsde2-wave-notilde})
 with terminal time $ (2T_0) \wedge T $. We find  $v^{(1)} : [0,(2T_0) \wedge T] \times H \to K$ according to \eqref{bsde2unobis} with $T$ replaced by  $(2T_0) \wedge T$.
By  \eqref{waveeq-nobad}  we obtain in particular 
\begin{gather*} 
  X_\tau^{1} -  X_\tau^{2}
= e^{\tau A}(x_1 - x_2) + e^{\tau A}[v^{(1)}(0,x_1) - v^{(1)}(0, x_2)]
\\ \nonumber 
 - [v^{(1)}(\tau,X_\tau^{1}) - v^{(1)}(\tau,X_\tau^{2})]
+\int_0^\tau e^{(\tau-s)A}[ \nabla^G v^{(1)}(s,X_s^1) - \nabla^G v^{(1)}(s,X_s^2) ]\;dW_s, \;\; \tau \in [T_0,(2T_0) \wedge T]. 
\end{gather*} 
Arguing as before, we  get, for 
 $\tau \in [T_0,(2T_0) \wedge T]$,  
$$ 
 \vert e^{\tau A}(x_1 - x_2)\vert_K +\vert e^{\tau A}[v^{(1)}(0,x_1) - v^{(1)}(0, x_2)]\vert_K +
\vert v^{(1)}(\tau,X_\tau^{1}) - v^{(1)}(\tau,X_\tau^{2})\vert_K $$$$ \leq C_T\vert x_1-x_2\vert_K +\frac{1}{3}\vert X_\tau^1-X_\tau^2\vert_K
$$
and 
\begin{gather*}
\E \Big | \int_0^\tau e^{(\tau-s)A}[ \nabla^G v^{(1)}(s,X_s^1) - \nabla^G v^{(1)}(s,X_s^2) ]\;dW_s \Big|^2_K  
 \\ 
 \le   C_T \sup_{t \in [0,T]}\| B(t, \cdot )\|_{ C^{\alpha}_b(H, U)}^2 \int_0^{\tau} \E   |X_s^1 - X_s^2 |^2_K ds
\\
=  C_T \sup_{t \in [0,T]}\| B(t, \cdot )\|_{ C^{\alpha}_b(H, U)}^2  \Big (   \int_{0}^{T_0 } \E   |X_s^1 - X_s^2 |^2_K ds 
+
\int_{T_0 }^{\tau} \E   |X_s^1 - X_s^2 |^2_K ds 
 \Big)
\\ \le C_T \sup_{t \in [0,T]}\| B(t, \cdot )\|_{ C^{\alpha}_b(H, U)}^2  \Big (    
  c_T T_0 \, |x_1 - x_2|^2_K + \int_{T_0 }^{\tau} \E   |X_s^1 - X_s^2 |^2_K ds  \big).
\end{gather*}
We have obtained, for $\tau \in  [T_0,(2T_0) \wedge T]$,
\begin{gather*}
\E |X^1_{\tau} - X^2_{\tau}|^2_K \le C_T' |x_1 - x_2|^2_K  +  C_T'
\int_{T_0 }^{\tau} \E   |X_s^1 - X_s^2 |^2_K ds.
\end{gather*}
By the Gronwall lemma we find
  $ \sup_{\tau \in [T_0,  (2T_0) \wedge T ]} \E \vert X_\tau^{x_1}-X_\tau^{x_2}\vert^2_K \leq c_T \vert x_1-x_2\vert^2_K.$ 
Proceeding in this way, in finite steps, we  get \eqref{lip-dependence}.  
\end{proof}

\appendix

\section{Appendix: an estimate on the minimal control energy 
for the controlled wave equation}

\subsection{The control system in $K$}

Recall that we are  considering a positive self-adjoint operator ${\Lambda}$ on a real separable Hilbert space ${U}$, i.e., ${\Lambda} : \mathcal{D}\left(  {\Lambda}\right)  \subset {U} \to {U}$.  Moreover ${\Lambda}^{-1}: U \to U$ is of trace class. The classical  Hilbert space for the abstract wave equation is 
\[
K =\mathcal{D} (  {\Lambda}^{\frac{1}{2}} )  \times  {U} = V \times U   
\]
 endowed with the inner product 
 $\left\langle \left(
\begin{array}
[c]{c}%
y_{1}\\
z_{1}%
\end{array}
\right)  ,\left(
\begin{array}
[c]{c}%
y_{2}\\
z_{2}%
\end{array}
\right)  \right\rangle _{K}     =\left\langle {\Lambda}^{\frac{1}{2}}
y_{1},{\Lambda}^{\frac{1}{2}}y_{2}\right\rangle_{{U} } +\left\langle z_{1},z_{2}\right\rangle _{ {U}  }. 
$  
 This space is also denoted by $V \oplus U$. 
We also   introduce 
\begin {equation} \label{rit11}
\mathcal{D}\left(  A\right)  =\mathcal{D}\left(  {\Lambda}\right)
\times\mathcal{D}\left(  {\Lambda}^{\frac{1}{2}}\right)  ,\text{ \ \ \ \ }%
A\left(
\begin{array}
[c]{c}%
y\\
z 
\end{array}
\right)  =\left(
\begin{array}
[c]{cc}%
0 & I\\
-{\Lambda} & 0
\end{array}
\right)  \left(
\begin{array}
[c]{c}%
y\\
z
\end{array}
\right)  ,\text{ \ for every }\left(
\begin{array}
[c]{c}%
y\\ 
z
\end{array}
\right)  \in\mathcal{D}\left(  A\right)  
\end{equation}
 (cf. \cite{krabs}).      
The operator $A$ is the generator of the unitary  group $e^{tA}$ on $K$
\[
e^{tA}\left(
\begin{array}
[c]{c}%
y\\  
z
\end{array}
\right)  =\left(
\begin{array}
[c]{cc}%
\cos\sqrt{{\Lambda}}t & \frac{1}{\sqrt{{\Lambda}}}\sin\sqrt{{\Lambda}}t\\
-\sqrt{{\Lambda}}\sin\sqrt{{\Lambda}}t & \cos\sqrt{{S}}t
\end{array}
\right)  \left(
\begin{array} 
[c]{c}%
y\\
z
\end{array}
\right)  ,\text{ \ \ }t\in\mathbb{R}.
\]
For the study of stochastic wave equations   we  also need to consider the larger space  
$$
 H =  U \times V', 
$$
where  $V' $ is the dual space of $V$.
 We can extend $A $ on $H$: 
\begin {equation} \label{rit13a}
\mathcal{D}\left(  A\right)  = K =V 
\times {U}  ,\text{ \ \ \ \ }%
A\left(
\begin{array}
[c]{c}%
y\\
z
\end{array}
\right)  =\left(
\begin{array}
[c]{cc}%
0 & I\\
-  \Lambda & 0
\end{array}
\right)  \left(
\begin{array}
[c]{c}%
y\\
z
\end{array}
\right)  ,\text{ \ for every }\left(
\begin{array}
[c]{c}%
y\\
z
\end{array}
\right)  \in\mathcal{D}\left(  A\right), 
\end{equation}
and similarly we can  extend $e^{tA} : H \to H$, $t \ge 0.$  The  control system on $K$ we consider is the following one 
\begin{equation}
\left\{
\begin{array}
[c]{l}%
\overset{\cdot}{w}\left(  t\right)  =Aw\left(  t\right)  +Gu\left(  t\right)
,\\
w\left(  0\right)  =h \in K,
\end{array}
\right.  \label{null1}
\end{equation} 
where $u \in L^2_{loc}(0,\infty; U) $;  $G : U \to K$;   for any $a \in U$,   $ G a  =\left(
 \begin{array}
 [c]{c}%
 0\\
 a
 \end{array}   
 \right).   $
  The solution  of \eqref{null1} is given by 
\begin{gather*}   
 w(t) = e^{tA} h + \int_0^t e^{(t-s)A} Gu(s)ds
\end{gather*}
and takes value in  $K$; in particular, for any $T>0$, $ w \in C([0,T]; K)$ by     Lemma 3.1.5 in \cite{CW} .

Let $t>0$ and  define the operator $\mathcal L_t$,  $\mathcal L_t u:=\displaystyle\int_0^te^{(t-s)A}Gu(s)\,ds$, acting from $L^2(0,t;U)$ with values in $K \subset H$. We know that 
\begin{equation}\label{ffr}
  { \widetilde Q}_t^{1/2} (K) = Im \mathcal L_t = \mathcal L_t (L^2(0,t;U))
\end{equation}
(see, for instance, page 253 in \cite{Za}). 
 We are  considering 
 $\widetilde Q_t : K \to K $:
\begin{gather*}
 \widetilde Q_t  = \int_{0}^{t}e^{\left(  t-s\right)  A}  
 \left(
\begin{array}
[c]{cc}%
0 & 0\\
0 & I_U
\end{array}
\right)
     e^{\left(  t-s\right)  A^*}   
 ds 
\end{gather*}
 which is different from  $Q_t: H \to H$ considered  in \eqref{qttt}: 
 \begin{gather*}
 Q_t  = \int_{0}^{t}e^{\left(  t-s\right)  A}  
 \left(
\begin{array} 
[c]{cc}%
0 & 0\\
0 & \Lambda^{-1} 
\end{array}
\right)
     e^{\left(  t-s\right)  A^*}  
 ds.   
\end{gather*} 
Moreover, it is known that   the controlled abstract wave equation \eqref{null1}  is exactly controllable in $K$  (see \cite{DPZab} or  page 367  in \cite{BD1}).   
 
 Hence  ${ \widetilde Q}_t^{1/2} (K) = Im \mathcal L_t = K$. In particular the control system  is null controllable in $K$ (i.e., for any $h \in K$, $t>0$, there exists a control function $u: [0,t] \to  U$ such that the corresponding solution $w$ verifies $w(t)=0$).  
  We can define  $ \mathcal{E}_{C}\left(  t,h\right) $  as the infimum of 
$$
\Big (\int_0^t |u(r)|^2_{U}dr\Big)^{1/2} 
$$
over all controls  $u \in L^2(0,t; U )$ driving the solution $w$ from $h $ to $0$ in time $t$. By Theorem 15.3 in \cite{Za} it follows that 
\begin{equation}\label{dfff}
\mathcal{E}_{C}\left(  t, h\right) = |{ \widetilde  Q}_t^{-1/2} e^{tA} h|_K,\;\;\; h \in K.
\end{equation} 
  We are interested in the behaviour of  $|{ \widetilde Q}_t^{-1/2} e^{tA} h|_K$
 as $t \to 0^+$ (for related results we refer to  \cite{Tr} and  \cite{AL}).  
 The first estimate in the next result is known. We provide a self-contained proof using special control functions as in  \cite{Tr} (see also Chapter 1 in \cite{Za}).   
Recall that  $k = \left(  
 \begin{array}
 [c]{c}%
 v\\  
 a
 \end{array} 
 \right) \in K$ with $v\in V$ and $a \in U$. 
  \begin{theorem}
\label{ci1} Let $T_0>0$. (i) There
exists a positive constant $C_{T_0} >0$ such that, for any $v \in V$, we have
\begin{equation}\label{energy-G1}
\Big| { \widetilde Q}_T^{-1/2} e^{TA} \left(
 \begin{array}  
 [c]{c}%
 v\\
 0
 \end{array} 
 \right) \Big|_K \le   \frac {C_{T_0} |v|_V}{ T^{\frac{3}{2}}},\;\;     
 \;\; T \in  (0, T_0], 
\end{equation}
(ii) There
exists a positive constant $C >0$ such that,  for any $a \in U$,
 setting $ Ga = \left( 
\begin{array}  
[c]{c}%
 0\\   
a
\end{array}
\right) $
\begin{equation}\label{energy-G}
| { \widetilde Q}_T^{-1/2} e^{TA} Ga|_K \le     \frac {C |a|_U}{ T^{\frac{1}{2}}},    
 \;\; T >0.
\end{equation}
The previous estimates imply that there exists $C_{T_0}>0$ such that
\begin{equation}\label{cont2}
|{ \widetilde Q}_T^{-1/2} e^{TA} |_{L(K,K)}  \le  {C_{T_0}}\, {T^{-3/2}},\;\;\; T \in (0,T_0].  
\end{equation} 
\end{theorem}    
\dim (i)  Let $v \in V$   and $T>0$. We  consider $f(t) = t^2 (T-t)^2$ and 
$$ 
 \phi(t) = \frac{f(t)}{ \int_0^T f(s)ds},\;\;\; t \in [0,T]. 
$$
Note that $\phi(0) = \phi(T)$, $\int_0^T \phi(s) ds =1$ and  there exists $C>0$ (independent of $T>0$) such that $|\phi (t)| 
\le \frac{C}{T}$ and $|\phi'(t)| \le \frac{C}{T^2}$, $t \in [0,T].$ Let $\psi : [0,T] \to K$,
$$ 
\psi (t)= \left (\begin{array}  
[c]{c}
\psi_1(t)\\
 \psi_2(t)
\end{array}
\right) = - \phi(t) \, e^{tA} k =  - 
\left(
\begin{array}
[c]{c}
\phi(t)\cos(\sqrt{{\Lambda}}t)v
\\
- \phi(t){\sqrt{{\Lambda}}}\sin(\sqrt{{\Lambda}}\, t) \, v 
\end{array}
\right),\;\; t \in [0,T].
$$  
Using also the derivative $\psi_1'$ we introduce the control 
 \begin {equation} \label{uu1}
u(t) = \psi_2(t) + \psi_1'(t) \in {U},\;\; t \in [0,T].
 \end{equation}
 We show that it transfers $k$ to $0$ at time $T.$
We have
\[
\int_{0}^{T}e^{\left(  T-s\right)  A} G u(s) ds = 
\int_{0}^{T}e^{\left(  T-s\right)  A}   \left (\begin{array}
[c]{c}
 0 \\
 \psi_2(s)  
\end{array}
\right) ds  + \int_{0}^{T}e^{\left(  T-s\right)  A}  G   
 \psi_1'(s)  
 ds. 
\]
Since $  G   
 \psi_1'(s)  = 
\left(
\begin{array}
[c]{c}
 0
 \\ \psi'_1(s)
\end{array}
\right)   
$ is continuous from $[0,T] $ with values in $K$,  integrating by parts we get
\[
  \int_{0}^{T}e^{\left(  T-s\right)  A}  G   
 \psi_1'(s)   =  \int_{0}^{T}e^{\left(  T-s\right)  A}  A G   
 \psi_1(s)  
 ds
=  \int_{0}^{T}e^{\left(  T-s\right)  A} \left(
\begin{array}
[c]{cc}%
0 & I\\
-{{\Lambda} \,} & 0
\end{array}
\right)   \left(
\begin{array}
[c]{c}
 0
 \\ \psi_1(s)
\end{array}
\right)       
 ds 
\]
$$
= \int_{0}^{T}e^{\left(  T-s\right)  A}    \left(
\begin{array}
[c]{c}
 \psi_1(s) \\ 0
\end{array}
\right)       
 ds. 
$$
  Hence we find 
$$
\int_{0}^{T}e^{\left(  T-s\right)  A} G u(s) ds = 
- \int_{0}^{T} \phi(s) e^{\left(  T-s\right)  A} 
e^{sA} k \, ds  = - e^{TA} k.
$$
Now we compute the energy of the control $u$:
$
 \int_0^T |u(s)|^2_U ds.    
$
Note that 
$$
\int_0^T |\psi_2(t)|^2_U dt = \int_0^T \phi(t)^2    \, |  {\sqrt{{\Lambda}}}\sin(\sqrt{{\Lambda}}\, t) \,    v|^2_U dt \le \frac{c\,  |v|^2_V}{T}; 
$$    
$$
\int_0^T |\psi_1'(t)|^2_U dt = \int_0^T \Big|  - \phi(t) \sqrt{\Lambda}  \sin(\sqrt{{{\Lambda} \,}}t)v \,  - \, \phi'(s)   \, \cos(\sqrt{{{\Lambda} \,}}t)v  \Big|^2_U
dt
$$$$  
\le c   \Big( \frac{  |v|^2_V}{T} +    \int_0^T \Big| \phi'(t)  \cos(\sqrt{{{\Lambda} \,}}t)v    \Big|^2_U
dt \Big)
\le  \frac{ c_{T_0} \,  |v|^2_V}{T^3}, 
$$  
where $c_{T_0}$ is independent of $T$ and $v$.  Collecting the previous estimates we   obtain
$$
 | { \widetilde Q}_T^{-1/2} e^{TA} k|_K  \le \Big(\int_0^T |u(s)|^2_U ds \Big)^{1/2} \le \frac{C_{T_0}   \, |v|_V}{ \sqrt{T^3}}, \;\; T \in (0,T_0]. 
$$ 
(ii) Let us  fix $T>0$ and $k = \left (\begin{array}
[c]{c}
0\\ 
a
\end{array}
\right) $ with $a \in {U}$. Consider  $\phi(t)$ as before. Define 
    $$
\psi (t)= \Big (\begin{array}
[c]{c}
\psi_1(t)\\  
 \psi_2(t)
\end{array}
 \Big ) = - \phi(t) \, e^{tA} k =   - \Big(
\begin{array}
[c]{c}
 \phi(t) \frac{1}{\sqrt{{{\Lambda} \,}}}\sin(\sqrt{{{\Lambda} \,}}t) \, a \\
 \phi(t) \cos(\sqrt{{{\Lambda} \,}}t)a
\end{array}
\Big) ,\;\; t \in [0,T]. 
$$
Using also the derivative $\psi_1'$ we introduce as in the first part of the proof  the control 
 $u(t) = \psi_2(t) + \psi_1'(t) \in {U},$ $ t \in [0,T].$ 
 It transfers $k$ to $0$ at time $T$ since 
 \begin{gather*}
  \int_{0}^{T}e^{\left(  T-s\right)  A}  G   
 \psi_1'(s) ds  =  \int_{0}^{T}e^{\left(  T-s\right)  A}  A G   
 \psi_1(s)  
 ds
=  \int_{0}^{T}e^{\left(  T-s\right)  A} \left(
\begin{array}
[c]{cc}%
0 & I\\
-{{\Lambda} \,} & 0
\end{array}
\right)   \left(
\begin{array}
[c]{c}
 0
 \\ \psi_1(s)
\end{array}
\right)       
 ds 
\\
= \int_{0}^{T}e^{\left(  T-s\right)  A}    \left(
\begin{array}
[c]{c}
 \psi_1(s) \\ 0
\end{array} 
\right)       
 ds  
 \end{gather*} 
 and      
$$
\int_{0}^{T}e^{\left(  T-s\right)  A} G u(s) ds = 
- \int_{0}^{T} \phi(s) e^{\left(  T-s\right)  A} 
e^{sA} k \, ds  = - e^{TA} k.
$$
We compute the energy of the control $u$:
$
 \int_0^T |u(s)|^2_U ds.   
$
First note that
$$
\int_0^T |\psi_2(t)|^2_U dt = \int_0^T \phi^2(t)     \, |\cos(\sqrt{{{\Lambda} \,}}t)a |^2_U dt \le \frac{c\,  |a|^2_U}{T}.   
$$
On the other hand   we get
$$
\int_0^T |\psi_1'(t)|^2_U dt = \int_0^T \Big|  \phi(t)  \cos(\sqrt{{{\Lambda} \,}}t)a  + \phi'(t)  \frac{1}{\sqrt{{{\Lambda} \,}}}\sin(\sqrt{{{\Lambda} \,}}t) \, 
a \Big|^2_U
dt
$$$$ 
\le c \Big ( \frac{  |a|^2_U}{T} +   |a|^2_U \int_0^T \Big| \phi'(t) t \frac{1}{\sqrt{{{\Lambda} \,} } t }\sin(\sqrt{{{\Lambda} \,}}t)  \Big|^2_U
dt \Big)
\le  \frac{ c' |a|^2_U}{T}, 
$$
where $c'$ is independent of $T$ and $a$. Collecting the previous estimates we   obtain
$$
 | { \widetilde Q}_T^{-1/2} e^{TA} Ga|_K   \le \Big(\int_0^T |u(s)|^2_U ds \Big)^{1/2} \le \frac{C \, |a|_U}{ \sqrt{T}}, \;\; T>0. \qed
$$ 

\subsection{The control system in $H$} \label{refer}  
  
 Here we consider  the previous  control system \eqref{null1}   in 
 $H$, i.e., we take the initial condition $h \in H$. 
 The control function $u $ still belongs to 
  $ L^2_{loc}(0,\infty; U) $ with   $G : U \to K \subset H$. Let $t>0$.  We still  have   
  \begin{equation}\label{ffr2}
  {Q}_t^{1/2} (H) = Im \mathcal L_t = \mathcal L_t (L^2(0,t;U))
\end{equation}
 where  $\mathcal L_t$  is the same operator considered in \eqref{ffr} but we have to consider $Q_t$ instead of ${ \widetilde Q}_t$.
 
  It follows that  ${Q}_t^{1/2} (H) =K$. Moreover, any $k \in K$
   we  define  $ \mathcal{E}_{C}\left(  t,k\right) $  as we have done for the control system in $K$.  By Theorem 15.3 in \cite{Za} it follows that
   \begin{equation}\label{dfff1}
\mathcal{E}_{C}\left(  t, k \right) = |{ Q}_t^{-1/2} e^{tA} k|_H
\end{equation} 
  for any $k \in K$. By \eqref{dfff} and \eqref{ffr2} we infer that 
 \begin{equation}\label{abc} 
  |{ \widetilde  Q}_t^{-1/2} e^{tA} k|_K =
  |{ Q}_t^{-1/2} e^{tA} k|_H,\;\;\; k \in K, \; t>0. 
\end{equation} 
It follows  that 
  $Q_t^{-1/2} e^{tA} $      
 belongs to $L(K,H)$ and, applying Theorem \ref{ci1},
  for any  $T>0$, 
 there exists $c= c_T>0$ such that  for any $t \in (0,T]$ 
  we have    
\begin{gather} \label{se19f}   
 |Q^{-1/2}_t e^{tA} k|_H \le   \frac{c}{t^{3/2}} |k|_K,\; \;\; k \in K = V \times U;   
\\
 \nonumber  |Q^{-1/2}_t e^{tA} Ga|_H     \le \frac{c}{t^{1/2}} |a|_U,\; \;\; a \in  U.  
 \end{gather}

 \begin{remark} \label{cont}{\em 
   We  stress that the previous  control system \eqref{null1} when considered  in 
 $H$   is not null controllable in $H$ at any time $t>0$. 
  This follows by the group property of $(e^{tA})$, the fact that $e^{tA}(K)=K$  and  noting that   ${Q}_t^{1/2} (H) =K$, $t>0$.
 
 The lack of null controllability can also be deduced by  applying  the argument in 
 page 180 of  \cite{DZergo}. 
  Indeed    $G: U \to H$  is a Hilbert-Schmidt operator because  
 $$\sum_{k \ge 1} |G e_k|^2_H = \sum_{k \ge 1} \Big |\Big(
 \begin{array}
 [c]{c}%
 0\\  
 e_k
 \end{array} 
 \Big) \Big |^2_H = \sum_{k \ge 1} |  \Lambda^{-1/2} e_k|^2_U < \infty.
 $$
 Hence if $h \in H$ in general we cannot transfer $h$ to 0 at time $t>0$ by a control $u \in  L^2(0,t;U)$. We could transfer $ h$ to 0 by a control
 $
 u \in  L^2(0,t;V') 
$
but then we should change $G$ with $\tilde G: V' \to H$, $\tilde G v' = \Big(
 \begin{array}
 [c]{c}%
 0\\  
 v'
 \end{array} 
 \Big)$ (recall that $G$ is given  in \eqref{wa1}).    }
 \end{remark}

\end{document}